\documentclass[11pt,twoside,english,american]{article}
\usepackage[T1]{fontenc}
\usepackage[latin9]{inputenc}
\usepackage[a4paper]{geometry}
\usepackage{fancyhdr}
\usepackage{color}
\usepackage{array}
\usepackage{verbatim}
\usepackage{longtable}
\usepackage{refstyle}
\usepackage{mathrsfs}
\usepackage{mathtools}
\usepackage{amsthm}
\usepackage{amsmath}
\usepackage{amssymb}
\usepackage{graphicx}
\usepackage{setspace}
\usepackage{esint}
\usepackage[authoryear]{natbib}
\usepackage{xargs}[2008/03/08]
\makeatletter

\AtBeginDocument{\providecommand\subref[1]{\ref{sub:#1}}}
\AtBeginDocument{\providecommand\secref[1]{\ref{sec:#1}}}
\AtBeginDocument{\providecommand\appref[1]{\ref{app:#1}}}
\AtBeginDocument{\providecommand\modref[1]{\ref{mod:#1}}}
\AtBeginDocument{\providecommand\eqref[1]{\ref{eq:#1}}}
\AtBeginDocument{\providecommand\thmref[1]{\ref{thm:#1}}}
\AtBeginDocument{\providecommand\propref[1]{\ref{prop:#1}}}
\AtBeginDocument{\providecommand\assref[1]{\ref{ass:#1}}}
\AtBeginDocument{\providecommand\enuref[1]{\ref{enu:#1}}}
\AtBeginDocument{\providecommand\remref[1]{\ref{rem:#1}}}
\AtBeginDocument{\providecommand\figref[1]{\ref{fig:#1}}}
\AtBeginDocument{\providecommand\lemref[1]{\ref{lem:#1}}}
\AtBeginDocument{\providecommand\corref[1]{\ref{cor:#1}}}
\providecommand{\tabularnewline}{\\}
\RS@ifundefined{subref}
  {\def\RSsubtxt{section~}\newref{sub}{name = \RSsubtxt}}
  {}
\RS@ifundefined{thmref}
  {\def\RSthmtxt{theorem~}\newref{thm}{name = \RSthmtxt}}
  {}
\RS@ifundefined{lemref}
  {\def\RSlemtxt{lemma~}\newref{lem}{name = \RSlemtxt}}
  {}

\usepackage{enumitem}		
\numberwithin{equation}{section}
\numberwithin{table}{section}
\numberwithin{figure}{section}
  \theoremstyle{remark}
  \newtheorem*{note*}{\protect\notename}
 \theoremstyle{definition}
 \newtheorem{model}{\protect\modelname}
  \theoremstyle{plain}
  \newtheorem{assumption}{\protect\assumptionname}
  \theoremstyle{remark}
  \newtheorem{rem}{\protect\remarkname}[section]
  \theoremstyle{plain}
  \newtheorem{prop}{\protect\propositionname}[section]
  \theoremstyle{plain}
  \newtheorem{thm}{\protect\theoremname}[section]
  \theoremstyle{plain}
  \newtheorem{lem}{\protect\lemmaname}[section]
  \theoremstyle{plain}
  \newtheorem{cor}{\protect\corollaryname}[section]

\makeatletter
  \@ifpackageloaded{refstyle}{}{\usepackage{refstyle}}
\makeatother

\newref{thm}{name = Theorem~}
\newref{prop}{name = Proposition~}
\newref{lem}{name = Lemma~}
\newref{cor}{name = Corollary~}
\newref{ass}{name = Assumption~}
\newref{exa}{name = Example~}
\newref{rem}{name = Remark~}
\newref{def}{name = Definition~}

\newref{eq}{refcmd = (\ref{#1})}

\newref{sec}{name = Section~}
\newref{sub}{name = Section~}
\newref{app}{name = Appendix~}

\newref{fig}{name = Figure~}
\newref{tbl}{name = Table~}

\newref{mod}{name = Model~}

\makeatletter
  \@ifpackageloaded{mathrsfs}{}{\usepackage{mathrsfs}}
\makeatother

\usepackage{nextpage}

\allowdisplaybreaks[1]

\usepackage{footmisc}

\providecommand{\LyX}{L\kern-.0467em\lower.25em\hbox{Y}\kern-.025emX\@}

\makeatletter
\newcommand \cellfill {\leavevmode \leaders \hb@xt@ .44em{\hss .\hss }\hfill \kern \z@}
\makeatother

\usepackage{scalefnt}
\newcommand\smaller[2][0.85]{{\scalefont{#1}#2}}

\makeatother

\usepackage{babel}
  \addto\captionsamerican{\renewcommand{\assumptionname}{Assumption}}
  \addto\captionsamerican{\renewcommand{\lemmaname}{Lemma}}
  \addto\captionsamerican{\renewcommand{\notename}{Note}}
  \addto\captionsamerican{\renewcommand{\propositionname}{Proposition}}
  \addto\captionsamerican{\renewcommand{\remarkname}{Remark}}
  \addto\captionsenglish{\renewcommand{\assumptionname}{Assumption}}
  \addto\captionsenglish{\renewcommand{\lemmaname}{Lemma}}
  \addto\captionsenglish{\renewcommand{\notename}{Note}}
  \addto\captionsenglish{\renewcommand{\propositionname}{Proposition}}
  \addto\captionsenglish{\renewcommand{\remarkname}{Remark}}
  \providecommand{\assumptionname}{Assumption}
  \providecommand{\lemmaname}{Lemma}
  \providecommand{\notename}{Note}
  \providecommand{\propositionname}{Proposition}
  \providecommand{\remarkname}{Remark}
 \addto\captionsamerican{\renewcommand{\modelname}{Model}}
 \addto\captionsenglish{\renewcommand{\modelname}{Model}}
 \providecommand{\modelname}{Model}
\addto\captionsamerican{\renewcommand{\corollaryname}{Corollary}}
\addto\captionsamerican{\renewcommand{\theoremname}{Theorem}}
\addto\captionsenglish{\renewcommand{\corollaryname}{Corollary}}
\addto\captionsenglish{\renewcommand{\theoremname}{Theorem}}
\providecommand{\corollaryname}{Corollary}
\providecommand{\theoremname}{Theorem}

\begin{document}
\selectlanguage{english}%


\global\long\def\uwrite#1#2{\underset{#2}{\underbrace{#1}} }

\global\long\def\blw#1{\ensuremath{\underline{#1}}}

\global\long\def\abv#1{\ensuremath{\overline{#1}}}

\global\long\def\vect#1{\mathbf{#1}}


\global\long\def\smlseq#1{\{#1\} }

\global\long\def\seq#1{\left\{  #1\right\}  }

\global\long\def\smlsetof#1#2{\{#1\mid#2\} }

\global\long\def\setof#1#2{\left\{  #1\mid#2\right\}  }


\global\long\def\goesto{\ensuremath{\rightarrow}}

\global\long\def\ngoesto{\ensuremath{\nrightarrow}}

\global\long\def\uto{\ensuremath{\uparrow}}

\global\long\def\dto{\ensuremath{\downarrow}}

\global\long\def\uuto{\ensuremath{\upuparrows}}

\global\long\def\ddto{\ensuremath{\downdownarrows}}

\global\long\def\ulrto{\ensuremath{\nearrow}}

\global\long\def\dlrto{\ensuremath{\searrow}}


\global\long\def\setmap{\ensuremath{\rightarrow}}

\global\long\def\elmap{\ensuremath{\mapsto}}

\global\long\def\compose{\ensuremath{\circ}}

\global\long\def\cont{C}

\global\long\def\cadlag{D}

\global\long\def\Ellp#1{\ensuremath{\mathcal{L}^{#1}}}


\global\long\def\naturals{\ensuremath{\mathbb{N}}}

\global\long\def\reals{\mathbb{R}}

\global\long\def\complex{\mathbb{C}}

\global\long\def\rationals{\mathbb{Q}}

\global\long\def\integers{\mathbb{Z}}


\global\long\def\abs#1{\ensuremath{\left|#1\right|}}

\global\long\def\smlabs#1{\ensuremath{\lvert#1\rvert}}
 \global\long\def\bigabs#1{\ensuremath{\bigl|#1\bigr|}}
 \global\long\def\Bigabs#1{\ensuremath{\Bigl|#1\Bigr|}}
 \global\long\def\biggabs#1{\ensuremath{\biggl|#1\biggr|}}

\global\long\def\norm#1{\ensuremath{\left\Vert #1\right\Vert }}

\global\long\def\smlnorm#1{\ensuremath{\lVert#1\rVert}}
 \global\long\def\bignorm#1{\ensuremath{\bigl\|#1\bigr\|}}
 \global\long\def\Bignorm#1{\ensuremath{\Bigl\|#1\Bigr\|}}
 \global\long\def\biggnorm#1{\ensuremath{\biggl\|#1\biggr\|}}


\global\long\def\Union{\ensuremath{\bigcup}}

\global\long\def\Intsect{\ensuremath{\bigcap}}

\global\long\def\union{\ensuremath{\cup}}

\global\long\def\intsect{\ensuremath{\cap}}

\global\long\def\pset{\ensuremath{\mathcal{P}}}

\global\long\def\clsr#1{\ensuremath{\overline{#1}}}

\global\long\def\symd{\ensuremath{\Delta}}

\global\long\def\intr{\operatorname{int}}

\global\long\def\cprod{\otimes}

\global\long\def\Cprod{\bigotimes}


\global\long\def\smlinprd#1#2{\ensuremath{\langle#1,#2\rangle}}

\global\long\def\inprd#1#2{\ensuremath{\left\langle #1,#2\right\rangle }}

\global\long\def\orthog{\ensuremath{\perp}}

\global\long\def\dirsum{\ensuremath{\oplus}}


\global\long\def\spn{\operatorname{sp}}

\global\long\def\rank{\operatorname{rk}}

\global\long\def\proj{\operatorname{proj}}

\global\long\def\tr{\operatorname{tr}}


\global\long\def\smpl{\ensuremath{\Omega}}

\global\long\def\elsmp{\ensuremath{\omega}}

\global\long\def\sigf#1{\mathcal{#1}}

\global\long\def\sigfield{\ensuremath{\mathcal{F}}}
\global\long\def\sigfieldg{\ensuremath{\mathcal{G}}}

\global\long\def\flt#1{\mathcal{#1}}

\global\long\def\filt{\mathcal{F}}
\global\long\def\filtg{\mathcal{G}}

\global\long\def\Borel{\ensuremath{\mathcal{B}}}

\global\long\def\cyl{\ensuremath{\mathcal{C}}}

\global\long\def\nulls{\ensuremath{\mathcal{N}}}

\global\long\def\gauss{\mathfrak{g}}

\global\long\def\leb{\mathfrak{m}}


\global\long\def\prob{P}

\global\long\def\Prob{\ensuremath{\mathbb{P}}}

\global\long\def\Probs{\mathcal{P}}

\global\long\def\PROBS{\mathcal{M}}

\global\long\def\expect{\ensuremath{\mathbb{E}}}

\global\long\def\probspc{\ensuremath{(\smpl,\filt,\Prob)}}


\global\long\def\iid{\ensuremath{\textnormal{i.i.d.}}}

\global\long\def\as{\ensuremath{\textnormal{a.s.}}}

\global\long\def\asp{\ensuremath{\textnormal{a.s.p.}}}

\global\long\def\io{\ensuremath{\ensuremath{\textnormal{i.o.}}}}

\newcommand\independent{\protect\mathpalette{\protect\independenT}{\perp}}
\def\independenT#1#2{\mathrel{\rlap{$#1#2$}\mkern2mu{#1#2}}}

\global\long\def\indep{\independent}

\global\long\def\distrib{\ensuremath{\sim}}

\global\long\def\distiid{\ensuremath{\sim_{\iid}}}

\global\long\def\asydist{\ensuremath{\overset{a}{\distrib}}}

\global\long\def\inprob{\ensuremath{\overset{p}{\goesto}}}

\global\long\def\inprobu#1{\ensuremath{\overset{#1}{\goesto}}}

\global\long\def\inas{\ensuremath{\overset{\as}{\goesto}}}

\global\long\def\eqas{=_{\as}}

\global\long\def\inLp#1{\ensuremath{\overset{\Ellp{#1}}{\goesto}}}

\global\long\def\indist{\ensuremath{\overset{d}{\goesto}}}

\global\long\def\eqdist{=_{d}}

\global\long\def\wkc{\ensuremath{\rightsquigarrow}}

\global\long\def\wkcu#1{\overset{#1}{\ensuremath{\rightsquigarrow}}}

\global\long\def\plim{\operatorname*{plim}}


\global\long\def\var{\operatorname{var}}

\global\long\def\lrvar{\operatorname{lrvar}}

\global\long\def\cov{\operatorname{cov}}

\global\long\def\corr{\operatorname{corr}}

\global\long\def\bias{\operatorname{bias}}

\global\long\def\MSE{\operatorname{MSE}}

\global\long\def\med{\operatorname{med}}


\global\long\def\simple{\mathcal{R}}

\global\long\def\sring{\mathcal{A}}

\global\long\def\sproc{\mathcal{H}}

\global\long\def\Wiener{\ensuremath{\mathbb{W}}}

\global\long\def\sint{\bullet}

\global\long\def\cv#1{\left\langle #1\right\rangle }

\global\long\def\smlcv#1{\langle#1\rangle}

\global\long\def\qv#1{\left[#1\right]}

\global\long\def\smlqv#1{[#1]}


\global\long\def\trans{\ensuremath{\prime}}

\global\long\def\indic{\ensuremath{\mathbf{1}}}

\global\long\def\Lagr{\mathcal{L}}

\global\long\def\grad{\nabla}

\global\long\def\pmin{\ensuremath{\wedge}}
\global\long\def\Pmin{\ensuremath{\bigwedge}}

\global\long\def\pmax{\ensuremath{\vee}}
\global\long\def\Pmax{\ensuremath{\bigvee}}

\global\long\def\sgn{\operatorname{sgn}}

\global\long\def\argmin{\operatorname*{argmin}}

\global\long\def\argmax{\operatorname*{argmax}}

\global\long\def\Rp{\operatorname{Re}}

\global\long\def\Ip{\operatorname{Im}}

\global\long\def\deriv{\ensuremath{\mathrm{d}}}

\global\long\def\diffnspc{\ensuremath{\deriv}}

\global\long\def\diff{\ensuremath{\,\deriv}}

\global\long\def\i{\ensuremath{\mathrm{i}}}

\global\long\def\e{\mathrm{e}}

\global\long\def\sep{,\ }

\global\long\def\defeq{\coloneqq}

\global\long\def\eqdef{\eqqcolon}

\selectlanguage{american}%

\selectlanguage{english}%
\global\long\def\S{\mathcal{S}}

\global\long\def\M{\mathcal{M}}

\global\long\def\N{\mathcal{N}}

\global\long\def\V{\mathcal{V}}

\global\long\def\U{\mathcal{U}}

\global\long\def\R{\mathcal{R}}

\global\long\def\mg{\xi}

\global\long\def\minidx{k_{0}}

\global\long\def\minidxm{m_{0}}

\global\long\def\smltail#1{\smlnorm{#1}_{\omega}}

\global\long\def\alphnorm#1#2{\norm{#1}_{[#2]}}

\global\long\def\smlalph#1#2{\smlnorm{#1}_{[#2]}}

\global\long\def\smlmom#1#2{\smlnorm{#1}_{#2}}

\newcommandx\RV[1][usedefault, addprefix=\global, 1=]{\mathrm{RV}(#1)}

\global\long\def\BI{\mathrm{BI}}

\global\long\def\BInorm#1{\smlnorm{#1}_{\BI}}

\global\long\def\BIalph#1{\mathrm{BI}_{[#1]}}

\global\long\def\BILalph#1{\mathrm{BIL}_{[#1]}}

\global\long\def\BImom#1{\mathrm{BI}_{#1}}

\global\long\def\BILmom#1{\mathrm{BIL}_{#1}}

\global\long\def\BIz{\mathrm{BI_{0}}}

\global\long\def\BIc{\mathrm{BIC}}

\global\long\def\BIcz{\mathrm{BIC}_{0}}

\global\long\def\BIl{\mathrm{BIL}}

\global\long\def\BIlz{\mathrm{BIL}_{0}}

\global\long\def\Lip{\mathrm{Lip}}

\global\long\def\Reg{\mathrm{BI_{R}}}

\global\long\def\locest{\mathcal{L}_{n}}

\global\long\def\loctime{\mathcal{L}}

\global\long\def\lmest{\mathcal{\mu}_{n}}

\global\long\def\lmeas{\mathcal{\mu}}

\global\long\def\minpr{\mathfrak{m}}

\global\long\def\maxpr{\mathfrak{M}}

\global\long\def\boundpr{\mathfrak{U}}

\global\long\def\major#1{\mathcal{#1}}

\global\long\def\grid{\mathbb{G}}

\global\long\def\ucc{\mathrm{ucc}}

\global\long\def\Hspc{\mathscr{H}}

\global\long\def\Bspc{\mathscr{B}}

\global\long\def\Aspc{\mathscr{A}}

\global\long\def\Cspc{\mathscr{C}}

\global\long\def\gauss{\phi_{\mathfrak{g}}}

\global\long\def\Gauss{\Phi_{\mathfrak{g}}}

\global\long\def\cmpct{\varphi}

\global\long\def\Fset{\mathscr{F}}

\global\long\def\Gset{\mathscr{G}}

\global\long\def\Uset{\mathscr{U}}

\global\long\def\locp{\kappa}

\global\long\def\signal{S}

\global\long\def\cfidx{p_{0}}

\global\long\def\smlfloor#1{\lfloor#1\rfloor}

\global\long\def\smlceil#1{\lceil#1\rceil}

\global\long\def\bkt{[\,]}

\global\long\def\bktpower{\beta}

\global\long\def\ctrproc{\nu_{n}}

\newcommand{\SM}{\textnormal{(SM)}}

\newcommand{\AP}{\textnormal{(AP)}}

\newcommand{\LM}{\textnormal{(LM)}}

\newcommand{\IN}{\textnormal{(IN)}}

\newcommand{\CR}{\textnormal{(CR)}}

\global\long\def\diam{\operatorname{diam}}

\global\long\def\nseq{e}

\global\long\def\etamom{q_{0}}

\global\long\def\bandmax{r_{0}}

\global\long\def\isbigop{\lesssim_{p}}

\global\long\def\derivbnd#1{\abv m_{#1}}

\global\long\def\truncnorm#1{\smlnorm{#1}_{n}}

\global\long\def\intfnal{\iota}

\global\long\def\class#1{\mathcal{#1}}

\global\long\def\emp{\mathbb{G}}

\global\long\def\kernel{k}

\global\long\def\Kernel{K}

\global\long\def\spc#1{\mathscr{#1}}

\global\long\def\avg{\mu}

\global\long\def\Qrob{\mathbb{Q}}

\global\long\def\kprod{\mathbb{K}}

\global\long\def\lprod{\mathbb{L}}

\global\long\def\set#1{\mathcal{#1}}

\global\long\def\Beta{\mathrm{B}}

\global\long\def\vek{\operatorname{vec}}

\global\long\def\vekh{\operatorname{vech}}

\global\long\def\vidx{\nu}

\global\long\def\widx{\omega}

\global\long\def\err{\eta}

\global\long\def\a{\alpha}

\global\long\def\b{\beta}

\global\long\def\t{\theta}

\global\long\def\L{\mathcal{L}}

\global\long\def\l{\lambda}

\global\long\def\bind{\theta}

\global\long\def\Wald{\mathrm{W}}

\global\long\def\LR{\mathrm{LR}}

\global\long\def\LM{\mathrm{LM}}

\global\long\def\stat{T}

\global\long\def\Like{\mathcal{L}}

\global\long\def\like{\ell}

\global\long\def\trans{\mathsf{T}}

\selectlanguage{american}%
\newcommand\BZ{\textnormal{(B$$0$$)}}

\newcommand\GN{\textnormal{\smaller[0.76]{GN}}}

\newcommand\QN{\textnormal{\smaller[0.76]{QN}}}

\newcommand\TR{\textnormal{\smaller[0.76]{TR}}}

\selectlanguage{english}%
\global\long\def\gn{\mathrm{GN}}

\global\long\def\qn{\mathrm{QN}}

\global\long\def\tn{\mathrm{TR}}
\selectlanguage{american}%

\global\long\def\varmat{\mathrm{\Sigma}}

\global\long\def\covmat{\mathrm{R}}

\renewcommand\thepage{\roman{page}}

\thispagestyle{plain}

\title{Generalized Indirect Inference for Discrete Choice Models\thanks{The authors thank Debopam Battacharya, Martin Browning and Liang Chen
for helpful comments. Keane's work on this project has been funded
by the Australian Research Council under grant FL110100247. The manuscript
was prepared with \LyX{}~2.1.3 and JabRef~2.7b.}}

\author{Marianne Bruins\thanks{Nuffield College and Department of Economics, University of Oxford}\\
James A.\ Duffy\footnotemark[2] \thanks{Institute for New Economic Thinking at the Oxford Martin School}\\
Michael P.\ Keane\footnotemark[2]\\
Anthony A.\ Smith, Jr.\thanks{Department of Economics, Yale University and National Bureau of Economic
Research}}

\date{July 2015}
\maketitle
\begin{abstract}
This paper develops and implements a practical simulation-based method
for estimating dynamic discrete choice models. The method, which can
accommodate lagged dependent variables, serially correlated errors,
unobserved variables, and many alternatives, builds on the ideas of
indirect inference. The main difficulty in implementing indirect inference
in discrete choice models is that the objective surface is a step
function, rendering gradient-based optimization methods useless. To
overcome this obstacle, this paper shows how to smooth the objective
surface. The key idea is to use a smoothed function of the latent
utilities as the dependent variable in the auxiliary model. As the
smoothing parameter goes to zero, this function delivers the discrete
choice implied by the latent utilities, thereby guaranteeing consistency.
We establish conditions on the smoothing such that our estimator enjoys
the same limiting distribution as the indirect inference estimator,
while at the same time ensuring that the smoothing facilitates the
convergence of gradient-based optimization methods. A set of Monte
Carlo experiments shows that the method is fast, robust, and nearly
as efficient as maximum likelihood when the auxiliary model is sufficiently
rich. 
\end{abstract}
\vspace{\fill}
\begin{note*}
An earlier version of this paper was circulated as the unpublished
manuscript \citet{KS03mimeo}. That paper proposed the method of generalized
indirect inference (GII), but did not formally analyze its asymptotic
or computational properties. The present work, under the same title
but with two additional authors (Bruins and Duffy), rigorously establishes
the asymptotic and computational properties of GII. It is thus intended
to subsume the 2003 manuscript. Notably, the availability of the 2003
manuscript allowed GII to be used in numerous applied studies (see
\subref{proposal}), even though the statistical foundations of the
method had not been firmly established. The present paper provides
these foundations and fills this gap in the literature.
\end{note*}
\newpage{}

{\singlespacing\tableofcontents{}

}

\newpage{}

\setcounter{page}{1}

\renewcommand\thepage{\arabic{page}}

\section{Introduction}

Many economic models have the features that (i) given knowledge of
the model parameters, it is easy to simulate data from the model,
but (ii) estimation of the model parameters is extremely difficult.
Models with discrete outcomes or mixed discrete/continuous outcomes
commonly fall into this category. A good example is the multinomial
probit (MNP), in which an agent chooses from among several discrete
alternatives the one with the highest utility. Simulation of data
from the model is trivial: simply draw utilities for each alternative,
and assign to each agent the alternative that gives them the greatest
utility. But estimation of the MNP, via either maximum likelihood
(ML) or the method of moments (MOM), is quite difficult.

The source of the difficulty in estimating the MNP, as with many other
discrete choice models, is that, from the perspective of the econometrician,
the probability an agent chooses a particular alternative is a high-dimensional
integral over multiple stochastic terms (unobserved by the econometrician)
that affect utilities the agent assigns to each alternative. These
probability expressions must be evaluated many times in order to estimate
the model by ML or MOM. For many years econometricians worked on developing
fast simulation methods to evaluate choice probabilities in discrete
choice models (see \citealp{LM81ch}). It was only with the development
of fast and accurate smooth probability simulators that ML or MOM-based
estimation in these models became practical (see \citealp{McF89Ecta},
and \citealp{Keane94Ecta}).

A different approach to inference in discrete choice models is the
method of ``indirect inference.'' This approach (see \citealp{Smith90Thesis,Smith93JAE};
\citealp{GMR93JAE}; \citealp{GT96ET}), circumvents the need to construct
the choice probabilities generated by the economic model, because
it is not based on forming the likelihood or forming moments based
on choice frequencies. Rather, the idea of indirect inference (II)
is to choose a statistical model that provides a rich description
of the patterns in the data. This descriptive model is estimated on
both the actual observed data and on simulated data from the economic
model. Letting $\b$ denote the vector of parameters of the structural
economic model, the II estimator is that $\hat{\b}$ which makes the
simulated data ``look like'' the actual data---in the sense (defined
formally below) that the descriptive statistical model estimated on
the simulated data ``looks like'' that same model estimated on the
actual data. (The method of moments is thus a special case of II,
in which the descriptive statistical model corresponds to a vector
of moments.)

Indirect inference holds out the promise that it should be practical
to estimate any economic model from which it is practical to simulate
data, even if construction of the likelihood or population moments
implied by the model is very difficult or impossible. But this promise
has not been fully realized because of limitations in the II procedure
itself. It is very difficult to apply II to models that include discrete
(or discrete/continuous) outcomes for the following reason: small
changes in the structural parameters of such models will, in general,
cause the data simulated from the model to change discretely. Such
a discrete change causes the parameters of a descriptive model fit
to the simulated data to jump discretely, and these discontinuities
are inherited by the criterion function minimized by the II estimator. 

Thus, given discrete (or discrete/continuous) outcomes, the II estimator
cannot be implemented using gradient-based optimization methods. One
instead faces the difficult computational task of optimizing a $d_{\beta}$-dimensional
step function using much slower derivative-free methods. This is very
time-consuming and puts severe constraints on the size of the structural
models that can be feasibly estimated. Furthermore, even if estimates
can be obtained, one does not have derivatives available for calculating
standard errors.

In this paper we propose a ``generalized indirect inference'' (GII)
procedure to address this important problem (Sections~\ref{sec:gii}
and \ref{sec:refinements}). The key idea is to generalize the original
II method by applying two different descriptive statistical models
to the simulated and actual data. As long as the two descriptive models
share the same vector of pseudo-true parameter values (at least asymptotically),
the GII estimator based on minimizing the distance between the two
models is consistent, and will enjoy the same asymptotic distribution
as the II estimator. 

While the GII idea has wider applicability, here we focus on how it
can be used to resolve the problem of non-smooth objective functions
of II estimators in the case of discrete choice models. Specifically,
the model we apply to the simulated data does not fit the discrete
outcomes in that data. Rather, it fits a ``smoothed'' version of
the simulated data, in which discrete choice indicators are replaced
by smooth functions of the underlying continuous latent variables
that determine the model's discrete outcomes. In contrast, the model
we apply to the actual data is fit to observed discrete choices (obviously,
the underlying latent variables that generate actual agents' observed
choices are not seen by the econometrician). 

As the latent variables that enter the descriptive model applied to
the simulated data are smooth functions of the model parameters, the
non-smooth objective function problem is obviously resolved. However,
it remains to show that the GII estimator based on minimizing the
distance between these two models is consistent and asymptotically
normal. We show that, under certain conditions on the parameter regulating
the smoothing, the GII estimator has the same limiting distribution
as the II estimator, permitting inferences to be drawn in the usual
manner (\secref{asymptotics}).

Our theoretical analysis goes well beyond merely deriving the limiting
distribution of the minimizer of the GII criterion function. Rather,
in keeping with computational motivation of this paper, we show that
the proposed smoothing facilitates the convergence of derivative-based
optimizers, in the sense that the smoothing leads to a sample optimization
problem that is no more difficult than the corresponding population
problem, where the latter involves the minimization of a necessarily
smooth criterion (\secref{asymptotics}). We also provide a detailed
analysis of the convergence properties of selected line-search and
trust-region methods. Our results on the convergence of these derivative-based
optimizers seem to be new to the literature. (While our work here
is in some respects related to the theory of $k$-step estimators,
we depart significantly from that literature, for example by dropping
the usual requirement that the optimizations commence from the starting
values provided by some consistent initial estimator.)

Finally, we provide Monte Carlo evidence indicating that the GII procedure
performs well on a set of example models (\secref{monte}). We look
at some cases where simulated maximum likelihood (SML) is also feasible,
and show that efficiency losses relative to SML are small. We also
show how judicious choice of the descriptive (or auxiliary) model
is very important for the efficiency of the estimator. This is true
not only here, but for II more generally.

Proofs of the theoretical results stated in the paper are given in
Appendices~\ref{app:mainproofs}--\ref{app:ullnproof}. An index
of key notation appears in \appref{notation}. All limits are taken
as $n\goesto\infty$.

\section{The model}

\label{sec:model}

We first describe a class of discrete choice models that we shall
use as test cases for the estimation method that we develop in this
paper. As will become clear, however, the ideas underlying the method
could be applied to almost any conceivable model of discrete choice,
including models with mixed discrete/continuous outcomes, and even
models in which individuals' choices solve forward-looking dynamic
programming problems.

We henceforth focus mainly on panel data models with $n$ individuals,
each of whom selects a choice from a set of $J$ discrete alternatives
in each of $T$ time periods. Let $u_{itj}$ be the (latent) utility
that individual $i$ attaches to alternative $j$ in period $t$.
Without loss of generality, set the utility of alternative $J$ in
any period equal to $0$. In each period, each individual chooses
the alternative with the highest utility. Let $y_{itj}$ be equal
to $1$ if individual $i$ chooses alternative $j$ in period $t$
and be equal to $0$ otherwise. Define $u_{it}\defeq(u_{it1},\ldots,u_{it,J-1})$
and $y_{it}\defeq(y_{it1},\ldots,y_{it,J-1})$. The econometrician
observes the choices $\{y_{it}\}$ but not the latent utilities $\{u_{it}\}$.

The vector of latent utilities $u_{it}$ is assumed to follow a stochastic
process 
\begin{equation}
u_{it}=f(x_{it},y_{i,t-1},\ldots,y_{i,t-l},\epsilon_{it};\beta),\qquad t=1,\ldots,T,\label{eq:latentutil}
\end{equation}
where $x_{it}$ is a vector of exogenous variables.\footnote{The estimation method proposed in this paper can also accommodate
models in which the latent utilities in any given period depend on
lagged values of the latent utilities.} For each individual $i$, the vector of disturbances $\epsilon_{it}\defeq(\epsilon_{it1},\ldots,\epsilon_{it,J-1})$
follows a Markov process $\epsilon_{it}=g(\epsilon_{i,t-1},\eta_{it};\beta)$,
where $\{\eta_{it}\}_{t=1}^{T}$ is a sequence of i.i.d.\ random
vectors (of dimension $J-1$) having a specified distribution (which
does \emph{not} depend on $\beta$). The functions $f$ and $g$ depend
on a set of $k$ structural parameters $\b\in\Beta$. The sequences
$\{\eta_{it}\}_{t=1}^{T}$, $i=1,\ldots,n$, are independent across
individuals and independent of $x_{it}$ for all $i$ and $t$. The
initial values $\epsilon_{i0}$ and $y_{it}$, $t=0,-1,\ldots,1-l$,
are fixed exogenously.

Although the estimation method proposed in this paper can (in principle)
be applied to any model of this form, we focus on four special cases
of the general model. Three of these cases (Models \ref{mod:serialprobit},
\ref{mod:dynprobit}, and \ref{mod:trichotomous} below) can be feasibly
estimated using simulated maximum likelihood, allowing us to compare
its performance with that of the proposed method.
\begin{model}
\label{mod:serialprobit}$J=2$, $T>1$, and $u_{it}=bx_{it}+\epsilon_{it}$,
where $x_{it}$ is a scalar, $\epsilon_{it}=r\epsilon_{i,t-1}+\eta_{it}$,
$\eta_{it}\distiid N[0,1]$, and $\epsilon_{i0}=0$. This is a two-alternative
dynamic probit model with serially correlated errors; it has two unknown
parameters $b$ and $r$.
\end{model}
 
\begin{model}
\label{mod:dynprobit}$J=2$, $T>1$, and $u_{it}=b_{1}x_{it}+b_{2}y_{i,t-1}+\epsilon_{it}$,
where $x_{it}$ is a scalar and $\epsilon_{it}$ follows the same
process as in \modref{serialprobit}. The initial value $y_{i0}$
is set equal to 0. This is a two-alternative dynamic probit model
with serially correlated errors and a lagged dependent variable; it
has three unknown parameters $b_{1}$, $b_{2}$, and $r$.
\end{model}
 
\begin{model}
\label{mod:initialprob} Identical to \modref{dynprobit} except that
the econometrician does not observe the first $s<T$ of the individual's
choices. Thus there is an ``initial conditions'' problem (see \citealp{Heck81ch}).
\end{model}
 
\begin{model}
\label{mod:trichotomous}$J=3$, $T=1$, and the latent utilities
obey: 
\begin{align*}
u_{i1} & =b_{10}+b_{11}x_{i1}+b_{12}x_{i2}+\eta_{i1}\\
u_{i2} & =b_{20}+b_{21}x_{i1}+b_{22}x_{i3}+c_{1}\eta_{i1}+c_{2}\eta_{i2},
\end{align*}
where $(\eta_{i1},\eta_{i2})\distiid N[0,I_{2}]$. (Since $T=1$ in
this model, the time subscript has been omitted.) This is a static
three-alternative probit model; it has eight unknown parameters $\{b_{1k}\}_{k=0}^{2}$,
$\{b_{2k}\}_{k=0}^{2}$, $c_{1}$, and $c_{2}$.
\end{model}

The techniques developed in this paper may also be applied to models
with a mixture of discrete and continuous outcomes. A leading example
is the Heckman selection model:
\begin{model}
\label{mod:selection}A selection model with two equations: The first
equation determines an individual's wage and the second determines
his/her latent utility from working: 
\begin{align*}
w_{i} & =b_{10}+b_{11}x_{1i}+c_{1}\eta_{1i}+c_{2}\eta_{i2}\\
u_{i} & =b_{20}+b_{21}x_{2i}+b_{22}w_{i}+\eta_{i2},
\end{align*}
Here $x_{1i}$ and $x_{2i}$ are exogenous regressors and $(\eta_{i1},\eta_{i2})\distiid N[0,I_{2}]$.
The unknown parameters are $\{b_{1k}\}_{k=0}^{1}$, $\{b_{2k}\}_{k=0}^{2}$,
$c_{1}$, and $c_{2}$. Let $y_{i}\defeq I(u_{i}\ge0)$ be an indicator
for employment status. The econometrician observes the outcome $y_{i}$
but not the latent utility $u_{i}$. In addition, the econometrician
observes a person's wage $w_{i}$ if and only if he/she works (i.e.\ if
$y_{i}=1$).
\end{model}

\section{Generalized indirect inference\label{sec:gii}}

We propose to estimate the model in \secref{model} via a generalization
of indirect inference. First, in Section 3.1 we exposit the method
of indirect inference as originally formulated. In Section 3.2 we
explain the difficulty of applying the original approach to discrete
choice models. Then, Section 3.3 presents our generalized indirect
inference estimator that resolves this difficulty.

\subsection{Indirect inference\label{sub:ii}}

Indirect inference exploits the ease and speed with which one can
typically simulate data from even complex structural models. The basic
idea is to view both the observed data and the simulated data through
the ``lens'' of a descriptive statistical (or auxiliary) model characterized
by a set of $d_{\theta}$ auxiliary parameters $\t$. The $d_{\beta}\le d_{\theta}$
structural parameters $\b$ are then chosen so as to make the observed
data and the simulated data look similar when viewed through this
lens.

To formalize these ideas, assume the observed choices $\{y_{it}\}$,
$i=1,\ldots,n$, $t=1,\ldots,T$, are generated by the structural
discrete choice model described in \eqref{latentutil}, for a given
value $\b_{0}$ of the structural parameters. An auxiliary model can
be estimated using the observed data to obtain parameter estimates
$\hat{\t}_{n}$. Formally, $\hat{\t}_{n}$ solves: 
\begin{equation}
\hat{\t}_{n}\defeq\argmax_{\theta\in\Theta}\L_{n}(y,x;\t)=\argmax_{\theta\in\Theta}\frac{1}{n}\sum_{i=1}^{n}\like(y_{i},x_{i};\theta),\label{eq:thetahat}
\end{equation}
where $\L_{n}(y,x;\t)$ is the average log-likelihood function (or
more generally, some statistical criterion function) associated with
the auxiliary model, $y\defeq\{y_{it}\}$ is the set of observed choices,
and $x\defeq\{x_{it}\}$ is the set of observed exogenous variables.

Let $\eta^{m}\defeq\{\eta_{it}^{m}\}$ denote a set of simulated draws
for the values of the unobservable components of the model, with these
draws being independent across $m\in\{1,\ldots,M\}$. Then given $x$
and a parameter vector $\b$, the structural model can be used to
generate $M$ corresponding sets of simulated choices, $y^{m}(\beta)\defeq\{y_{it}^{m}(\b)\}$.
(Note that the same values of $x$ and $\{\eta^{m}\}$ are used for
all $\beta$.) Estimating the auxiliary model on the $m$th simulated
dataset thus yields 
\begin{equation}
\hat{\t}_{n}^{m}(\b)\defeq\argmax_{\theta\in\Theta}\L_{n}(y^{m}(\b),x;\t).\label{eq:thetatilde}
\end{equation}
Let $\abv{\t}_{n}(\b)\defeq\tfrac{1}{M}\sum_{m=1}^{M}\hat{\t}_{n}^{m}(\b)$
denote the average of these estimates. Under appropriate regularity
conditions, as the observed sample size $n$ grows large (holding
$M$ and $T$ fixed), $\abv{\t}_{n}(\b)$ converges uniformly in probability
to a non-stochastic function $\bind(\beta)$, which \citet{GMR93JAE}
term the \emph{binding function}.

Loosely speaking, indirect inference generates an estimate $\hat{\b}_{n}$
of the structural parameters by choosing $\b$ so as to make $\hat{\t}_{n}$
and $\abv{\t}_{n}(\b)$ as close as possible, with consistency following
from $\hat{\t}_{n}$ and $\abv{\t}_{n}(\b_{0})$ both converging to
the same pseudo-true value $\theta_{0}\defeq\bind(\beta_{0})$. To
implement the estimator we require a formal metric of the distance
between $\hat{\t}_{n}$ and $\abv{\t}_{n}(\b)$. There are three approaches
to choosing such a metric, analogous to the three classical approaches
to hypothesis testing: the Wald, likelihood ratio (LR), and Lagrange
multiplier (LM) approaches.\footnote{This nomenclature is due to Eric Renault. The Wald and LR approaches
were first proposed in \citet{Smith90Thesis,Smith93JAE} and later
extended by \citet{GMR93JAE}. The LM approach was first proposed
in \citet{GT96ET}.} 

The Wald approach to indirect inference chooses $\b$ to minimize
the weighted distance between $\abv{\theta}_{n}(\b)$ and $\hat{\t}_{n}$,
\[
Q_{n}^{\Wald}(\beta)\defeq\smlnorm{\abv{\theta}_{n}(\b)-\hat{\t}_{n}}_{W_{n}}^{2},
\]
where $\smlnorm x_{A}^{2}\defeq x^{\trans}Ax$, and $W_{n}$ is a
sequence of positive-definite weight matrices. 

The LR approach forms a metric implicitly by using the average log-likelihood
$\L_{n}(y,x;\theta)$ associated with the auxiliary model. In particular,
it seeks to minimize 
\[
Q_{n}^{\LR}(\beta)\defeq-\L_{n}(y,x;\abv{\theta}_{n}(\b))=-\frac{1}{n}\sum_{i=1}^{n}\like(y_{i},x_{i};\abv{\theta}_{n}(\beta))
\]

Finally, the LM approach does not work directly with the estimated
auxiliary parameters $\abv{\theta}_{n}(\b)$ but instead uses the
score vector associated with the auxiliary model.\footnote{When the LM approach is implemented using an auxiliary model that
is (nearly) correctly specified in the sense that it provides a (nearly)
correct statistical description of the observed data, \citet{GT96ET}
refer to this approach as efficient method of moments (EMM).} Given the estimated auxiliary model parameters $\hat{\t}$ from the
observed data, the score vector is evaluated using each of the $M$
simulated data sets. The LM estimator then minimizes a weighted norm
of the average score vector across these datasets, 
\[
Q_{n}^{\LM}(\beta)\defeq\norm{\frac{1}{M}\sum_{m=1}^{M}\dot{\L}_{n}(y^{m}(\b),x;\hat{\t}_{n})}_{V_{n}}^{2},
\]
where $\dot{\L}_{n}$ denotes the gradient of $\L_{n}$ with respect
to $\t$, and $V_{n}$ is a sequence of positive-definite weight matrices.

All three approaches yield consistent and asymptotically normal estimates
of $\b_{0}$, and are first-order asymptotically equivalent in the
exactly identified case in which $d_{\beta}=d_{\theta}$. In the over-identified
case, when the weight matrices $W_{n}$ and $V_{n}$ are chosen optimally
(in the sense of minimizing asymptotic variance) both the Wald and
LM estimators are more efficient than the LR estimator. However, if
the auxiliary model is correctly specified, all three estimators are
asymptotically equivalent not only to each other but also to maximum
likelihood (provided that $M$ is sufficiently large).

\subsection{Indirect inference for discrete choice models}

Step functions arise naturally when applying indirect inference to
discrete choice models because any simulated choice $y_{it}^{m}(\b)$
is a step function of $\b$ (holding fixed the set of random draws
$\{\eta_{it}^{m}\}$ used to generate simulated data from the structural
model). Consequently, the sample binding function $\abv{\t}_{n}(\b)$
is discontinuous in $\b$. Obviously, this discontinuity is inherited
by the criterion functions minimized by the II estimators in Section
3.1. 

Thus, given discrete outcomes, II cannot be implemented using gradient-based
optimization methods. One must instead rely on derivative-free methods
(such as the Nelder-Mead simplex method); random search algorithms
(such as simulated annealing); or abandon optimization altogether,
and instead implement a Laplace-type estimator, via Markov Chain Monte
Carlo (MCMC; see \citealp{CH03JoE}). But convergence of derivative-free
methods is often very slow; while MCMC, even when it converges, may
produce (in finite samples) an estimator substantially different from
the optimum of the statistical criterion to which it is applied (see
\citealp{KN12mimeo}). Thus, the non-smoothness of the criterion functions
that define II estimators render them very difficult to use in the
case of discrete data.

Despite the difficulties in applying II to discrete choice models,
the appeal of the II approach has led some authors to push ahead and
apply it nonetheless. Some notable papers that apply II by optimizing
non-smooth objective functions are \citet{MRV95JAE}, \citet{AL00RESt},
\citet{Nagypal07RES}, \citet{EHM15IER}, \citet{LZ2015AEJM} and
\citet{Skira2015IER}. Our work aims to make it much easier to apply
II in these and related contexts.

\subsection{A smoothed estimator (GII)\label{sub:proposal}}

Here we propose a generalization of indirect inference that is far
more practical in the context of discrete outcomes. The fundamental
idea is that the estimation procedures applied to the observed and
simulated data sets need not be identical, provided that they both
provide consistent estimates of the same binding function. (\citealp{GR03JASA},
use a similar insight to develop robust estimation procedures in the
context of indirect inference.) We exploit this idea to smooth the
function $\abv{\t}_{n}(\b)$, obviating the need to optimize a step
function when using indirect inference to estimate a discrete choice
model.

Let $u_{itj}^{m}(\b)$ denote the latent utility that individual $i$
attaches to alternative $j\in\{1,\ldots,$$J-1\}$ in period $t$
of the $m$th simulated data set, given structural parameters $\b$
(recall that the utility of the $J$th alternative is normalized to
$0$). Rather than use the simulated choice $y_{itj}^{m}(\b)$ when
computing $\abv{\t}_{n}(\b)$, we propose to replace it by the following
smooth function of the latent utilities,
\[
y_{itj}^{m}(\beta,\lambda)\defeq K_{\lambda}[u_{itj}^{m}(\beta)-u_{it1}^{m}(\beta)\sep\ldots\sep u_{itj}^{m}(\beta)-u_{it,J-1}^{m}(\beta)],
\]
where $K:\reals^{J-1}\setmap\reals$ is a smooth, mean-zero multivariate
cdf, and $K_{\lambda}(v)\defeq K(\lambda^{-1}v)$. As the smoothing
parameter $\l$ goes to $0$, the preceding converges to $y_{itj}^{m}(\beta,0)=y_{itj}^{m}(\beta)$.
Defining $\abv{\theta}_{n}(\beta,\lambda)\defeq\tfrac{1}{M}\sum_{m=1}^{M}\hat{\t}_{n}^{m}(\b,\lambda)$,
where
\begin{equation}
\hat{\t}_{n}^{m}(\b,\lambda)\defeq\argmax_{\theta\in\Theta}\L_{n}(y^{m}(\b,\lambda),x;\t),\label{eq:simlest}
\end{equation}
we may regard $\abv{\theta}_{n}(\beta,\lambda)$ as a smoothed estimate
of $\bind(\beta)$, for which it is consistent so long as $\l=\lambda_{n}\goesto0$
as $n\goesto\infty$. Accordingly, an indirect inference estimator
based on $\abv{\theta}_{n}(\beta,\lambda_{n})$, which we shall henceforth
term the \emph{generalized indirect inference} (GII) estimator, ought
to be consistent for $\beta_{0}$.

Each of the three approaches to indirect inference can be generalized
simply by replacing each simulated choice $y_{itj}^{m}(\b)$ with
its smoothed counterpart $y_{itj}^{m}(\beta,\lambda_{n})$. For the
Wald and LR estimators, this entails using the smoothed sample binding
function $\abv{\theta}_{n}(\beta,\lambda_{n})$ in place of the unsmoothed
estimate $\abv{\theta}_{n}(\b)$. (See \subref{bias} below for the
exact forms of the criterion functions.) The remainder of this paper
is devoted to studying the properties of the resulting estimators,
both analytically (\secref{asymptotics}) and through a series of
simulation exercises (\secref{monte}).

The GII approach was first suggested in an unpublished manuscript
by \citet{KS03mimeo}, but they did not derive the asymptotic properties
of the estimator. Despite this, GII has proven to be popular in practice,
and has already been applied in a number of papers, such as \citet{GG07NBER},
\citet{Cassidy12mimeo}, \citet{ASV13Ecta}, \citet{Morten13mimeo},
\citet{Yypma2013thesis}, \citet{LM14ERE} and \citet{LG15mimeo}.
Given the growing popularity of the method, a careful analysis of
its asymptotic properties is obviously needed.

\subsection{Related literature}

Our approach to smoothing in a discrete choice model bears a superficial
resemblance to that used by \citet{Hor92Ecta} to develop a smoothed
version of \citeauthor{Manski85JoE}'s \citeyearpar{Manski85JoE}
maximum score estimator for a binary response model. As here, the
smooth version of maximum score is constructed by replacing discontinuous
indicators with smooth cdfs in the sample criterion function. 

However, there is a fundamental difference in the statistical properties
of the minimization problems solved by Manski's estimator, and the
(unsmoothed) indirect inference estimator. Specifically, $n^{-1/2}$-consistent
estimators are available for the \emph{unsmoothed} problem considered
in this paper (see \thmref{limitdist} below, or \citealp{PakP89Ecta});
whereas, in the case of \citeauthor{Manski85JoE}'s \citeyearpar{Manski85JoE}
maximum score estimator, only $n^{-1/3}$-consistency is obtained
without smoothing (see \citealp{KP90AS}), and smoothing yields an
estimator with an improved rate of convergence.

A potentially more relevant analogue for the present paper is smoothed
quantile regression. This originates with \citeauthor{Hor98Ecta}'s
\citeyearpar{Hor98Ecta} work on the smoothed least absolute deviation
estimator, extended to more general quantile regression and quantile-IV
models by \citet{Wha06ET}, \citet{Ots08JoE} and \citet{KSun12mimeo}.
The latter papers do not smooth the criterion function, but rather
the estimating equations (approximate first-order conditions) that
equivalently define the estimator. These first-order conditions involve
indicator-type discontinuities like those in our problem, smoothed
in the same way. Insofar as the problem of solving the estimating
equations is analogous to the minimum-distance problem solved by the
II estimator, the effects of smoothing are similar: in each case smoothing
(if done appropriately) affects neither the rate of convergence nor
the limiting distribution of the estimator, relative to its unsmoothed
counterpart.

The motivation for smoothing in the quantile regression case involves
the potential for higher-order asymptotic improvements.\footnote{While potential computational benefits have been noted in passing,
we are not aware of any attempt to demonstrate these formally, in
the manner of Theorems~\ref{thm:rootdist}--\ref{thm:optim} below.} In contrast, in the present setting, which involves structural models
of possibly great complexity, the potential for higher-order improvements
is limited.\footnote{This is particularly evident when the auxiliary model consists of
a system of regression equations, as per \subref{auxiliarymodel}
below. For while smoothing does indeed reduce the variability of the
simulated (discrete) outcomes $y_{it}^{m}(\beta,\lambda)$, this may
\emph{increase} the variance with which some parameters of the auxiliary
model are estimated, if $y_{it}$ appears as a regressor in that model:
as will be the case for Models~\ref{mod:dynprobit} and \ref{mod:initialprob}
(see Sections~\ref{sub:dynprobit} and \ref{sub:initialprob} below).
(Note that any such increase, while certainly possible, is of only
second-order importance, and disappears as $\lambda_{n}\goesto0$.)} The key motivation for smoothing in our case is computational. 

Accordingly, much of this paper is devoted to a formal analysis of
the potential computational gains from smoothing. In particular, Sections~\ref{sub:performance}--\ref{sub:convergence}
are devoted to providing a theoretical foundation for our claim that
smoothing facilitates the convergence of standard derivative-based
optimization that are widely used to solve (smooth) optimization problems
in practice. 

For the class of models considered in this paper, two leading alternative
estimation methods that might be conisidered are simulated maximum
likelihood (SML) in conjunction with the Geweke, Hajivassiliou and
Keane (GHK) smooth probability simulator (see Section~4 in \citealp{KG01Hdbk}),
and the nonparametric simulated maximum likelihood (NPSML) estimator
(\citealp{DG84JRSSB,FS04ET,KS12JoE}). However, the GHK simulator
can only be computed in models possessing a special structure -- which
is true for Models~\ref{mod:serialprobit}, \ref{mod:dynprobit}
and \ref{mod:trichotomous} above, but \emph{not} for \modref{initialprob}
-- while in models that involve a mixture of discrete and continuous
outcomes, NPSML may require the calculation of rather high-dimensional
kernel density estimates in order to construct the likelihood, the
accuracy of which may require simulating the model a prohibitively
large number of times.

Finally, an alternative approach to smoothing the II estimator is
importance sampling, as in \citet{KS10IER} and \citet{ST13mimeo}.
The basic idea is to simulate data from the structural model only
once (at the initial estimate of $\b$). One holds these simulated
data fixed as one iterates. Given an updated estimate of $\b$, one
re-weights the original simulated data points, so those initial simulations
that are more (less) likely under the new $\b$ (than under the initial
$\b$) get more (less) weight in forming the updated objective function. 

In our view the GII and importance sampling approaches both have virtues.
The main limitation of the importance sampling approach is that in
many models the importance sample weights may themselves be computationally
difficult to construct. \citet{KS10IER}, when working with models
similar to those in Section 2, assume that all variables are measured
with error, which gives a structure that impies very simple weights.
In many contexts such a measurement error assumption may be perfectly
sensible. But the GII method can be applied directly to the models
of Section 2 without adding any auxillary assumptions (or parameters).

\section{Further refinements and the choice of auxiliary model\label{sec:refinements}}

\subsection{Smoothing in dynamic models\label{sub:dynamic}}

For models in which latent utilities depend on past choices (as distinct
from past \emph{utilities}, which are already smooth), such as Models~\ref{mod:dynprobit}
and \ref{mod:initialprob} above, the performance of GII may be improved
by making a further adjustment to the smoothing proposed in \subref{proposal}.
The nature of this adjustment is best illustrated in terms of the
example provided by \modref{dynprobit}. In this case, it is clear
that setting
\[
y_{it}^{m}(\beta,\lambda)\defeq K_{\lambda}[b_{1}x_{it}+b_{2}y_{i,t-1}^{m}(\beta)+\epsilon_{it}^{m}],
\]
where $y_{i,t-1}^{m}(\beta)$ denotes the \emph{unsmoothed} choice
made at date $t-1$, will yield unsatisfactory results, insofar as
the $y_{it}^{m}(\beta,\lambda)$ so constructed will remain discontinuous
in $\beta$. To some extent, this may be remedied by modifying the
preceding to
\begin{equation}
y_{it}^{m}(\beta,\lambda)\defeq K_{\lambda}[b_{1}x_{it}+b_{2}y_{i,t-1}^{m}(\beta,\lambda)+\epsilon_{it}^{m}],\label{eq:dynwrong}
\end{equation}
with $y_{i0}^{m}(\beta,\lambda)\defeq0$, as per the specification
of the model. However, while the $y_{it}^{m}(\beta,\lambda)$'s generated
through this recursion will indeed be smooth (i.e., twice continuously
differentiable), the nesting of successive approximations entailed
by \eqref{dynwrong} implies that for large $t$, the derivatives
of $y_{it}^{m}(\beta,\lambda)$ may be highly irregular unless a relatively
large value of $\lambda$ is employed.

This problem may be avoided by instead computing $y_{it}^{m}(\beta,\lambda)$
as follows. Defining $v_{itk}^{m}(\beta)\defeq b_{1}x_{it}+b_{2}\indic\{k=1\}+\epsilon_{it}^{m},$
we see that the \emph{unsmoothed} choices satisfy
\begin{align*}
y_{it}(\beta) & =\indic\{v_{it0}^{m}(\beta)\geq0\}\cdot[1-y_{i,t-1}(\beta)]+\indic\{v_{it1}^{m}(\beta)\geq0\}\cdot y_{i,t-1}(\beta),
\end{align*}
which suggests using the following recursion for the smoothed choices,
\begin{equation}
y_{it}^{m}(\beta,\lambda)\defeq K_{\lambda}[v_{it0}^{m}(\beta)]\cdot[1-y_{i,t-1}^{m}(\beta,\lambda)]+K_{\lambda}[v_{it1}^{m}(\beta)]\cdot y_{i,t-1}^{m}(\beta,\lambda),\label{eq:dynright}
\end{equation}
with $y_{i0}^{m}(\beta,\lambda)\defeq0$. This indeed yields a valid
approximation to $y_{it}(\beta)$, as $\lambda\goesto0$. The smoothed
choices computed using \eqref{dynright} involve no nested approximations,
but merely sums of products involving terms of the form $K_{\lambda}[v_{isk}^{m}(\beta)]$.
The derivatives of these are well-behaved with respect to $\lambda$,
even for large $t$, and are amenable to the theoretical analysis
of \secref{asymptotics}. 

Nonetheless, we find that even if smoothing is done by simply using
\eqref{dynwrong}, GII appears to work well in practice. This will
be shown in the simulation exercises reported in \secref{monte}.

\subsection{Bias reduction via jackknifing\label{sub:bias}}

As we noted in Section 3.3, GII inherits the consistency of the II
estimator, provided that $\lambda_{n}\goesto0$ as $n\goesto\infty$.
However, as smoothing necessarily imparts a bias to the sample binding
function $\abv{\theta}_{n}(\beta,\lambda_{n})$, and thence to the
GII estimator, we need $\lambda_{n}$ to shrink to zero at a sufficiently
fast rate if GII is to enjoy the same limiting distribution as the
unsmoothed estimator. On the other hand, if $\lambda_{n}\goesto0$
too rapidly, derivatives of the GII criterion function will become
highly irregular, impeding the ability of derivative-based optimization
routines to locate the minimum.

Except for certain special cases, the smoothing bias is of the order
$\smlnorm{\theta(\beta_{0},\lambda)-\theta(\beta_{0},0)}=O(\lambda)$
and no smaller. Thus, it is only dominated by the estimator variance
if $n^{1/2}\lambda_{n}=o_{p}(1)$. On the other hand, it follows from
\propref{sufficiency} below that $n^{1-1/p_{0}}\lambda_{n}^{2l-1}\goesto\infty$
is necessary to ensure that the $l$th order derivatives of the GII
criterion function converge, uniformly in probability, to their population
counterparts. Here $p_{0}\in(1,\infty]$ depends largely on the order
of moments possessed by the exogenous covariates $x$ (see \assref{lowlevel}
below). Thus, even in the most favorable case of $p_{0}=\infty$,
one can only ensure asymptotic negligibility of the bias (relative
to the variance) at the cost of preventing second derivatives of the
sample criterion function from converging to their population counterparts,
a convergence that is necessary to ensure the good performance of
at least some derivative-based optimization routines (see \subref{convergence}
below).

Fortunately, these difficulties can easily be overcome by applying
Richardson extrapolation -- commonly referred to as ``jackknifing''
in the statistics literature -- to the smoothed sample binding function.
Provided that the \emph{population} binding function is sufficiently
smooth, a Taylor series expansion gives $\bind_{l}(\beta,\lambda)=\bind_{l}(\beta,0)+\sum_{r=1}^{s}\alpha_{rl}(\beta)\lambda^{r}+o(\lambda^{s})$
as $\lambda\goesto0$, for $l\in\{1,\ldots,d_{\theta}\}$. Then, for
a fixed choice of $\delta\in(0,1)$, we have the first-order extrapolation,
\[
\bind_{l}^{1}(\beta,\lambda)\defeq\frac{\bind_{l}(\beta,\delta\lambda)-\delta\bind_{l}(\beta,\lambda)}{1-\delta}=\bind_{l}(\beta,0)+\delta\sum_{r=2}^{s}(\delta^{r-1}-1)\alpha_{rl}(\beta)\lambda^{r}+o(\lambda^{s}),
\]
for every $l\in\{1,\ldots,d_{\theta}\}$. By an iterative process,
for $k\leq s-1$ we can construct a $k$th order extrapolation of
the binding function, which satisfies
\begin{equation}
\bind^{k}(\beta,\lambda)\defeq\sum_{r=0}^{k}\gamma_{rk}\bind(\beta,\delta^{r}\lambda)=\bind(\beta,0)+O(\lambda^{k+1}),\label{eq:kextrap}
\end{equation}
where the weights $\{\gamma_{rk}\}_{r=0}^{k}$ (which can be negative)
satisfy $\sum_{r=0}^{k}\gamma_{rk}=1$, and may be calculated using
Algorithm~1.3.1 in \citet{Sidi03}. It is immediately apparent that
the $k$th order jackknifed sample binding function,
\begin{equation}
\abv{\theta}_{n}^{k}(\beta,\lambda_{n})\defeq\sum_{r=0}^{k}\gamma_{rk}\abv{\theta}_{n}(\beta,\delta^{r}\lambda_{n})\label{eq:jackknifed}
\end{equation}
will enjoy an asymptotic bias of order $O_{p}(\lambda_{n}^{k+1})$,
whence only $n^{1/2}\lambda_{n}^{k+1}=o_{p}(1)$ is necessary for
the bias to be asymptotically negligible.

In the case where $\hat{\theta}_{n}^{m}(\beta,\lambda)=g(T_{n}^{m}(\beta,\lambda))$,
for some differentiable transformation $g$ of a vector $T_{n}^{m}$
of sufficient statistics (as in \secref{asymptotics} below), jackknifing
could be applied directly to these statistics. Thus, if we were to
set $\hat{\theta}_{n}^{mk}(\beta,\lambda_{n})\defeq g(\sum_{r=0}^{k}\gamma_{rk}T_{n}^{m}(\beta,\lambda_{n})),$
then $\frac{1}{M}\sum_{m=1}^{M}\hat{\theta}_{n}^{mk}(\beta,\lambda_{n})$
would also have an asymptotic bias of order $O_{p}(\lambda_{n}^{k+1})$.
This approach may have computational advantages if the transformation
$g$ is relatively costly to compute (e.g.\ when it involves matrix
inversion). Note that since $T_{n}^{m}$ will generally involve averages
of nonlinear transformations of kernel smoothers, it will not generally
be possible to achieve the same bias reduction through the use of
higher-order kernels; whereas if only linear transformations were
involved, both jackknifing and higher-order kernels would yield identical
estimators (see, e.g.,\ \citealp{JF93JNP}).

Jackknifed GII estimators of order $k\in\naturals_{0}$ may now be
defined as the minimizers of:
\begin{align}
Q_{nk}^{e}(\beta,\lambda_{n}) & \defeq\begin{cases}
\smlnorm{\abv{\theta}_{n}^{k}(\b,\lambda_{n})-\hat{\t}_{n}}_{W_{n}}^{2} & \text{if }e=\Wald\\
-\L_{n}(y;x,\abv{\theta}_{n}^{k}(\b,\lambda_{n})) & \text{if }e=\LR\\
\smlnorm{\tfrac{1}{M}\sum_{m=1}^{M}\dot{\L}_{n}^{mk}(\b,\lambda_{n};\hat{\t}_{n})}_{V_{n}}^{2} & \text{if }e=\LM
\end{cases}\label{eq:criteria}
\end{align}
where $\dot{\L}_{n}^{mk}(\beta,\lambda;\hat{\theta}_{n})\defeq\sum_{r=0}^{k}\gamma_{rk}\dot{\L}_{n}(y^{m}(\beta,\lambda),x;\hat{\theta}_{n})$
denotes the jackknifed score function; the un-jackknifed estimators
may be recovered by taking $k=0$. Let $Q_{k}^{e}(\beta,\lambda)$
denote the large-sample limit of $Q_{nk}^{e}(\beta,\lambda)$; note
that $\beta\elmap Q_{k}^{e}(\beta,0)$ is smooth and does not depend
on $k$.

\subsection{Bias reduction via a Newton-Raphson step\label{sub:onestep}}

By allowing the number of simulations $M$ to increase with the sample
size, we can accelerate the rate at which $\abv{\theta}_{n}^{k}$
converges to the binding function. The convergence of the smoothed
derivatives of $\abv{\theta}_{n}^{k}$ should then follow under less
restrictive conditions on $\lambda_{n}$. That is, it may be possible
for the derivatives to converge, while still ensuring that the bias
is $o(n^{-1/2})$, even with $k=0$. Since the evaluation of $Q_{nk}$
is potentially costly when $M$ is very large, one possible approach
would be to minimize $Q_{nk}$ using a very small initial value of
$M$ (e.g.\ $M=1$). One could then increase $M$ to an appropriately
large value, and then compute a new estimate by taking at least one
Newton-Raphson step (applied to the new criterion).

A rigorous analysis of this estimator is beyond the scope of this
paper; we assume that $M$ is \emph{fixed} throughout \secref{asymptotics}.
Heuristically, since $\abv{\theta}_{n}^{k}$ is computed using $nM$
observations, it should be possible to show that if $M=M_{n}\goesto\infty$,
then the conditions specified in \propref{sufficiency} below would
remain the same, except with $nM_{n}$ replacing every appearance
of $n$ in \eqref{lambdacond}.

\subsection{Choosing an auxiliary model\label{sub:auxiliarymodel}}

Efficiency is a key consideration when choosing an auxiliary model.
As discussed in \subref{ii}, indirect inference (generalized or not)
has the same asymptotic efficiency as maximum likelihood when the
auxiliary model is correctly specified in the sense that it provides
a correct statistical description of the observed data (Gallant and
Tauchen, 1996). Thus, from the perspective of efficiency, it is important
to choose an auxiliary model (or a class of auxiliary models) that
is flexible enough to provide a good description of the data.

Another important consideration is computation time. For the Wald
and LR approaches to indirect inference, the auxiliary parameters
must be estimated repeatedly using different simulated data sets.
For this reason, it is critical to use an auxiliary model that can
be estimated quickly and efficiently. This consideration is less important
for the LM approach, as it does not work directly with the estimated
auxiliary parameters, but instead uses the first-order conditions
(the score vector) that defines these estimates.

To meet the twin criteria of statistical and computational efficiency,
in Section 6 we use linear probability models (or, more accurately,
sets of linear probability models) as the auxillary model. This class
of models is flexible in the sense that an individual's current choice
can be allowed to depend on polynomial functions of lagged choices
and of current and lagged exogenous variables. These models can also
be very quickly and easily estimated using ordinary least squares.
\secref{monte} describes in detail how we specify the linear probability
models for each of Models~\ref{mod:serialprobit}--\ref{mod:trichotomous}.
For Model 5, the Heckman selection model, the auxillary model would
be a set of OLS regressions with mixed discrete/continuous dependent
variables.

\section{Asymptotic and computational properties\label{sec:asymptotics}}

While GII could in principle be applied to any model of the form \eqref{latentutil}
-- and others besides -- in order to keep this paper to a manageable
length, the theoretical results of this section will require that
some further restrictions be placed on the structure of the model.
Nonetheless, these restrictions are sufficiently weak to be consistent
with each of Models~\ref{mod:serialprobit}--\ref{mod:selection}
from \secref{model}. We shall only provide results for the Wald and
LR estimators, when these are jackknifed as per \eqref{jackknifed}
above; but it would be possible to extend our arguments so as to cover
the LM estimator, and the alternative jackknifing procedure (in which
the statistics $T_{n}^{m}$ are jackknifed) outlined in \subref{bias}.

\subsection{A general framework\label{sub:general}}

Individual $i$ is described by vectors $x_{i}\in\reals^{d_{x}}$
and $\err_{i}\in\reals^{d_{\err}}$ of observable and unobservable
characteristics; $x_{i}$ includes \emph{all} the covariates appearing
in either the structural model or the auxiliary model (or both). $\err_{i}$
is a vector of independent variates that are also independent of $x_{i}$,
and normalized to have unit variance. Their marginal distributions
are fully specified by the model, allowing these to be simulated.
Collect $z_{i}\defeq(x_{i}^{\trans},\err_{i}^{\trans})^{\trans}\in\reals^{d_{z}}$,
and define the projections $[x(\cdot),\err(\cdot)]$ so that $(x_{i},\err_{i})=[x(z_{i}),\err(z_{i})]$.
Individual $i$ has a vector $y(z_{i};\beta,\lambda)\in\reals^{d_{y}}$
of smoothed outcomes, parametrized by $(\beta,\lambda)\in\Beta\times\Lambda$,
with $\lambda=0$ corresponding to true, unsmoothed outcomes under
$\beta$. At this level of abstraction, we need not make any notational
distinction between choices made by an individual at the same date
(over competing alternatives), vs.\ choices made at distinct dates;
we note simply that each corresponds to some element of $y(\cdot)$.
With this notation, the $m$th simulated choices may be written as
$y(z_{i}^{m};\beta,\lambda)$; since the same $x_{i}$'s are used
across all simulations, we have $x(z_{i}^{m})=x(z_{i}^{m^{\prime}})$
but $\err(z_{i}^{m})\neq\err(z_{i}^{m^{\prime}})$ for $m^{\prime}\neq m$.

In line with the discussion in \subref{auxiliarymodel}, we shall
assume that the auxiliary model takes the form of a system of seemingly
unrelated regressions (SUR; see e.g.\ Section~10.2 in \citealp{Greene08})
\begin{equation}
y_{r}(z_{i};\beta,\lambda)=\alpha_{xr}^{\trans}\Pi_{xr}x(z_{i})+\alpha_{yr}^{\trans}\Pi_{yr}y(z_{i};\beta,\lambda)+\xi_{ri},\label{eq:reggeneral}
\end{equation}
where $\xi_{i}\defeq(\xi_{1i},\ldots,\xi_{d_{y}i})^{\trans}\distiid N[0,\Sigma_{\xi}]$,
and $\Pi_{xr}$ and $\Pi_{yr}$ are selection matrices (i.e.\ matrices
that take at most one unit value along each row, and have zeros everywhere
else); let $\alpha_{r}\defeq(\alpha_{xr}^{\trans},\alpha_{yr}^{\trans})^{\trans}$.
Typically, $\Sigma_{\xi}$ will be assumed block diagonal: for example,
we may only allow those $\xi_{ri}$'s pertaining to alternatives
from the same period to be correlated. The auxiliary parameter vector
$\theta$ collects a subset (or possibly all) of the elements of $(\alpha_{1}^{\trans},\ldots,\alpha_{d_{y}}^{\trans})^{\trans}$
and the (unrestricted) elements of $\Sigma_{\xi}^{-1}$. (For the
calculations involving the score vector in \appref{sufficiency},
it shall be more convenient to treat the model as being parametrized
in terms of $\Sigma_{\xi}^{-1}$.)

Several estimators of $\theta$ are available, most notably OLS, feasible
GLS, and maximum likelihood, all of which agree only under certain
conditions.\footnote{In \secref{monte}, exact numerical agreement between these estimators
is ensured by requiring the auxiliary model equations referring to
alternatives from the same period to have the same set of regressors.} For concreteness, we shall assume that both the data-based and simulation-based
estimates of $\theta$ are produced by maximum likelihood. However,
the results of this paper could be easily extended to cover the case
where either (or both) of these estimates are computed using OLS or
feasible GLS. (In those cases, the auxiliary estimator can be still
be written as a function of a vector of sufficient statistics, a property
that greatly facilitates the proofs of our results.)

We shall also need to restrict the manner in which $y(\cdot)$ is
parametrized. To that end, we introduce the following collections
of linear indices\begin{subequations}\label{eq:vwidx}
\begin{alignat}{2}
\vidx_{r}(z;\beta) & \defeq z^{\trans}\Pi_{\vidx r}\gamma(\beta) & \qquad & r\in\{1,\ldots,d_{\vidx}\}\label{eq:vidx}\\
\widx_{r}(z;\beta) & \defeq z^{\trans}\Pi_{\widx r}\gamma(\beta) &  & r\in\{1,\ldots,d_{\widx}\},\label{eq:widx}
\end{alignat}
\end{subequations}where $\gamma:\Beta\setmap\Gamma$, and $\Pi_{l}^{\vidx}$
and $\Pi_{l}^{\widx}$ are selection matrices. We shall generally
suppress the $z$ argument from $\vidx$ and $\widx$, and other quantities
constructed from them, throughout the sequel. Our principal restriction
on $y(\cdot)$ is that it should be constructed from $(\vidx,\widx)$
as follows. Let $d_{c}\geq d_{\widx}$; for each $r\in\{1,\ldots,d_{c}\}$,
let $\set S_{r}\subseteq\{1,\ldots,d_{v}\}$ and define
\begin{equation}
\tilde{y}_{r}(\beta,\lambda)\defeq\widx_{r}(\beta)\cdot\prod_{s\in\set S_{r}}K_{\lambda}[\vidx_{s}(\beta)]\label{eq:ytilde}
\end{equation}
collecting these in the vector $\tilde{y}(\beta,\lambda)$; where
now $K:\reals\setmap[0,1]$ is a smooth \emph{univariate} cdf, and
$K_{\lambda}(v)\defeq K(\lambda^{-1}v)$.\footnote{Keane and Smith (2003) suggested using the multivariate logistic cdf,
$L(v)\defeq1/(1+\sum_{j=1}^{J-1}\e^{-v_{j}})$, and this is used in
the simulation exercises presented in \secref{monte}. But $L$ has
no particular advantages over other choices of $K$, and, for the
theoretical results work we shall in fact assume that the smoothing
is implemented using suitable products of univariate cdfs. This assumption
eases some of our arguments (but it is unlikely that it is necessary
for our results).} Note that $d_{c}\geq d_{\widx}$, and that we have defined
\begin{equation}
\widx_{r}(z;\beta)\defeq1\qquad r\in\{d_{\widx}+1,\ldots,d_{c}\}.\label{eq:widx2}
\end{equation}
Let $\err_{\widx}\defeq\Pi_{\err\widx}\err$ select the elements of
$\err$ upon which $\widx$ actually depends (as determined by the
$\Pi_{\widx r}$ matrices), and let $W_{r}\geq1$ denote an envelope
for $\widx_{r}$, in the sense that $\smlabs{\widx_{r}(z;\beta)}\leq W_{r}(z)$
for all $\beta\in\Beta$. Let $\varrho_{\min}(A)$ denote the smallest
eigenvalue of a symmetric matrix $A$. 

\medskip{}

Our results rely on the following low-level assumptions on the structural
model:

\setcounter{assumption}{11}
\begin{assumption}[low-level conditions]
\label{ass:lowlevel}~
\begin{enumerate}[{label=\textnormal{\smaller[0.76]{L\arabic*}}}]
\item  \label{enu:L:iid}$(y_{i},x_{i})$ is i.i.d.\ over $i$, and $\err_{i}^{m}=\err(z_{i}^{m})$
is independent of $x_{i}$ and i.i.d.\ over $i$ and $m$;
\item \label{enu:L:linear} $y(\beta,\lambda)=D\tilde{y}(\beta,\lambda)$
for some $D\in\reals^{d_{y}\times d_{c}}$, for $\tilde{y}$ as in
\eqref{ytilde};
\item \label{enu:L:gamma} $\gamma:\Beta\setmap\Gamma$ in \eqref{vwidx}
is twice continuously differentiable;
\item \label{enu:L:density}for each $k\in\{1,\ldots,d_{\err}\}$, $\var(\err_{ki})=1$,
and $\err_{ki}$ has a density $f_{k}$ with 
\[
\sup_{u\in\reals}(1+\smlabs u^{4})f_{k}(u)<\infty;
\]

\item \label{enu:L:marginals}there exists an $\epsilon>0$ such that, for
every for every $r\in\{1,\ldots,d_{\vidx}\}$ and $\beta\in\Beta$,
\[
\var(\vidx_{r}(z_{i};\beta)\mid\err_{\widx i},x_{i})\geq\epsilon;
\]

\item \label{enu:L:moment} there exists a $p_{0}\geq2$ such that for each
$r\in\{1,\ldots,d_{c}\}$, $\expect(W_{r}^{4}+\smlnorm{z_{i}}^{4})<\infty$,
$\expect\smlabs{W_{r}\smlnorm{z_{i}}^{3}}^{p_{0}}<\infty$ and $\expect\smlabs{W_{r}^{2}\smlnorm{z_{i}}^{2}}^{p_{0}}<\infty$;
\item \label{enu:L:nonsingular} $\inf_{(\beta,\lambda)\in\Beta\times\Lambda}\varrho_{\min}[\expect\abv y(z_{i};\beta,\lambda)\abv y(z_{i};\beta,\lambda)^{\trans}]>0$,
where $\abv y(\beta,\lambda)\defeq[y(\beta,\lambda)^{\trans},x^{\trans}]^{\trans}$;
and
\item \label{enu:L:auxiliary} the auxiliary model is a Gaussian SUR, as
in \eqref{reggeneral}.
\end{enumerate}
\end{assumption}
 
\begin{rem}
\label{rem:reparam} \eqref{vwidx} entails that the estimator criterion
function $Q_{n}$ depends on $\beta$ only through $\gamma(\beta)$,
i.e.\ $Q_{n}(\beta)=\tilde{Q}_{n}(\gamma(\beta))$ for some $\tilde{Q}_{n}$.
Since the derivatives of $\tilde{Q}_{n}$ with respect to $\gamma$
take a reasonably simple form, we shall establish the convergence
of $\partial_{\beta}^{l}Q_{n}$ to $\partial_{\beta}^{l}Q$, for $l\in\{1,2\}$,
by first proving the corresponding result for $\partial_{\gamma}^{l}\tilde{Q}_{n}$
and then applying the chain rule. Here, as elsewhere in the paper,
$\partial_{\beta}f$ denotes the gradient of $f:\Beta\setmap\reals^{d}$
(the transpose of the Jacobian), and $\partial_{\beta}^{2}f$ the
Hessian; see Section~6.3 of \citealp{MN07}, for a definition of
the latter when $k\geq2$.
\end{rem}
 
\begin{rem}
\label{rem:discrete} \assref{lowlevel} is least restrictive in models
with purely discrete outcomes, for which we may take $d_{\widx}=0$.
In particular, \enuref{L:moment} reduces to the requirement that
$\expect\smlnorm{z_{i}}^{3p_{0}}<\infty$.
\end{rem}
 
\begin{rem}
As the examples discussed in \subref{applexamples} illustrate, except
in the case where current (discrete) choices depend on past choices,
it is generally possible to take $D=I_{d_{y}}$ in \enuref{L:linear},
so that $y(\beta,\lambda)=\tilde{y}(\beta,\lambda)$.
\end{rem}

Consistent with the notation adopted in the previous sections of this
paper, let $\eta^{m}$ denote the $m$th set of simulated unobservables,
and $y^{m}(\beta,\lambda)$ the associated smoothed outcomes, for
$m\in\{1,\ldots,M\}$. We may set $\Lambda=[0,1]$ below without loss
of generality. Let $\filt$ denote a $\sigma$-field with respect
to which the observed data and simulated variables are measurable,
for all $n$, and recall the definition of $\like$ given in \eqref{thetahat}
above. We then have the following:

\setcounter{assumption}{17}
\begin{assumption}[regularity conditions]
\label{ass:regularity}~
\begin{enumerate}[{label=\textnormal{\smaller[0.76]{R\arabic*}}}]
\item \label{enu:R:correct} The structural model is correctly specified:
$y_{i}=y(z_{i}^{0};\beta_{0},0)$ for some $\beta_{0}\in\intr\Beta$;
\item \label{enu:R:scorevar} $\theta_{0}\defeq\theta(\beta_{0},0)\in\intr\Theta$;
\item \label{enu:R:uniqueness} the binding function $\bind(\beta,\lambda)$
is single-valued, and is $(k_{0}+1)$-times differentiable in $\beta$
for all $(\beta,\lambda)\in(\intr\Beta)\times\Lambda$;
\item \label{enu:R:injective} $\beta\elmap\bind(\beta,0)$ is injective;
\item \label{enu:R:smoothing} $\{\lambda_{n}\}$ is an $\filt$-measurable
sequence with $\lambda_{n}\inprob0$;
\item \label{enu:R:jackknifing} the order $k\in\{1,\ldots,k_{0}\}$ of
the jackknifing is chosen such that $n^{1/2}\lambda_{n}^{k+1}=o_{p}(1)$;
\item \label{enu:R:kernel}$K$ in \eqref{ytilde} is a twice continuously
differentiable cdf, for a distribution having integer moments of all
orders, and density $\dot{K}$ symmetric about the origin; and
\item \label{enu:R:weight} $W_{n}\inprob W$, for some positive definite
$W$.
\end{enumerate}
\end{assumption}
 
\begin{rem}
\enuref{R:injective} formalizes the requirement that the auxiliary
model be ``sufficiently rich'' to identify the parameters of the
structural model; $d_{\theta}\geq d_{\beta}$ evidently is necessary
for \enuref{R:injective} to be satisfied.
\end{rem}
 
\begin{rem}
\label{rem:uniformity}\enuref{R:smoothing} permits the bandwidth
to be sample-dependent, as distinct from assuming it to be a ``given''
deterministic sequence. This means our results hold \emph{uniformly}
in smoothing parameter sequences satisfying certain growth rate conditions:
see \remref{derivunif} below for details. \enuref{R:jackknifing}
ensures that, in conjunction with the choice of $\lambda_{n}$, the
jackknifing is such as to ensure that the bias introduced by the smoothing
is asymptotically negligible. \enuref{R:kernel} will be satisfied
for many standard choices of $K$, such as the Gaussian cdf, and many
smooth, compactly supported kernels.
\end{rem}

Assumptions~\ref{ass:lowlevel} and \ref{ass:regularity} are sufficient
for all of our main results. But to allow these to be stated at a
higher level of generality -- and thus permitting their application
to a broader class of structural and auxiliary models than are consistent
with \assref{lowlevel} -- we shall find it useful to phrase our results
as holding under \assref{regularity} and the following high-level
conditions. To state these, define $\Like_{n}(\theta)\defeq\Like_{n}(y,x;\theta)$,
$\Like(\theta)\defeq\expect\Like_{n}(\theta)$ and $\like_{i}^{m}(\beta,\lambda;\theta)\defeq\like(y_{i}^{m}(\beta,\lambda),x_{i};\theta)$.
$\dot{\like}_{i}^{m}$ and $\ddot{\L}_{n}$ respectively denote the
gradient of $\like_{i}^{m}$ and the Hessian of $\Like_{n}$ with
respect to $\theta$, while for a metric space $(Q,d)$, $\ell^{\infty}(Q)$
denotes the space of bounded and measurable real-valued functions
on $Q$.

\setcounter{assumption}{7}
\begin{assumption}[high-level conditions]
\label{ass:highlevel}~
\begin{enumerate}[{label=\textnormal{\smaller[0.76]{H\arabic*}}}]
\item \label{enu:H:twicediff} $\Like_{n}$ is twice continuously differentiable
on $\intr\Theta$;
\item \label{enu:H:unifcmpct} for $l\in\{0,1,2\}$, $\partial_{\theta}^{l}\Like_{n}(\theta)\inprob\partial_{\theta}^{l}\Like(\theta)$,
and 
\[
\frac{1}{n}\sum_{i=1}^{n}\dot{\like}_{i}^{m_{1}}(\beta_{1},\lambda_{1};\theta_{1})\dot{\like}_{i}^{m_{2}}(\beta_{2},\lambda_{2};\theta_{2})^{\trans}\inprob\expect\dot{\like}_{i}^{m_{1}}(\beta_{1},\lambda_{1};\theta_{1})\dot{\like}_{i}^{m_{2}}(\beta_{2},\lambda_{2};\theta_{2})^{\trans}
\]
uniformly on $\Beta\times\Lambda$ and compact subsets of $\intr\Theta$,
for every $m_{1},m_{2}\in\{0,1,\ldots,M\}$;
\item \label{enu:H:stocheq} $\psi^{m}$ is a mean-zero, continuous Gaussian
process on $\Beta\times\Lambda$ such that
\[
\psi_{n}^{m}(\beta,\lambda)\defeq n^{1/2}[\hat{\theta}_{n}^{m}(\beta,\lambda)-\theta(\beta,\lambda)]\wkc\psi^{m}(\beta,\lambda)
\]
in $\ell^{\infty}(\Beta\times\Lambda)$, jointly in $m\in\{0,1,\ldots,M\}$;
\item \label{enu:H:limitdist} for any (possibly) random sequence $\beta_{n}=\beta_{0}+o_{p}(1)$
and $\lambda_{n}$ as in \enuref{R:smoothing}, 
\begin{equation}
\psi_{n}^{m}(\beta_{n},\lambda_{n})=-H^{-1}\frac{1}{n^{1/2}}\sum_{i=1}^{n}\dot{\like}_{i}^{m}(\beta_{0},0;\theta_{0})+o_{p}(1)\eqdef-H^{-1}\phi_{n}^{m}+o_{p}(1)\wkc-H^{-1}\phi^{m},\label{eq:psitophi}
\end{equation}
jointly in $m\in\{0,1,\ldots,M\}$,\footnote{The first equality in \eqref{psitophi} is only relevant for $m\geq1$.}
where $H\defeq\expect\ddot{\L}_{n}(\theta)=\ddot{\L}(\theta)$ and
$\{\phi^{m}\}_{m=0}^{M}$ is jointly Gaussian with 
\begin{align}
\varmat & \defeq\expect\phi^{m_{1}}\phi^{m_{1}\trans}=\expect\phi_{n}^{m_{1}}\phi_{n}^{m_{1}\trans} & \covmat & \defeq\expect\phi^{m_{1}}\phi^{m_{2}\trans}=\expect\phi_{n}^{m_{1}}\phi_{n}^{m_{2}\trans}\label{eq:scorecovar}
\end{align}
 for every $m_{1},m_{2}\in\{0,1,\ldots,M\}$; and
\item \label{enu:H:deriv}$\{\lambda_{n}\}$ is such that for some $l\in\{0,1,2\}$,
\[
\sup_{\beta\in\Beta}\smlnorm{\partial_{\beta}^{l}\hat{\theta}_{n}^{m}(\beta,\lambda_{n})-\partial_{\beta}^{l}\theta(\beta,0)}=o_{p}(1).
\]

\end{enumerate}
\end{assumption}

The sufficiency of our low-level conditions for the preceding may
be stated formally follows.
\begin{prop}
\label{prop:sufficiency} Suppose Assumptions~\ref{ass:lowlevel}
and \ref{ass:regularity} hold. Then \assref{highlevel} holds with
$l=0$ in \enuref{H:deriv}. Further, if $\lambda_{n}>0$ for all
$n$, with 
\begin{equation}
n^{1-1/p_{0}}\lambda_{n}^{2l^{\prime}-1}/\log(\lambda_{n}^{-1}\pmax n)\inprob\infty\label{eq:lambdacond}
\end{equation}
for some $l^{\prime}\in\{1,2\}$, then \enuref{H:deriv} holds with
$l=l^{\prime}$.
\end{prop}
 
\begin{rem}
\label{rem:deriv}It is evident from \eqref{lambdacond} that -- as
noted in \subref{bias} above -- the convergence of the higher-order
derivatives of the sample binding function requires more stringent
conditions on the smoothing sequence $\{\lambda_{n}\}$. We shall
be accordingly careful, in stating our results below, to identify
the weakest form of \enuref{H:deriv} (and correspondingly, of \eqref{lambdacond})
that is required for each of these.
\end{rem}
 
\begin{rem}
\label{rem:derivunif}Let $\{\blw{\lambda}_{n}\}$ and $\{\abv{\lambda}_{n}\}$
be deterministic sequences satisfying \eqref{lambdacond} and $\abv{\lambda}_{n}=o(1)$
respectively, and set $\Lambda_{n}\defeq[\blw{\lambda}_{n},\abv{\lambda}_{n}]$.
Then, as indicated in \remref{uniformity} above, $\filt$-measurability
of $\{\lambda_{n}\}$ entails that the convergence in \enuref{H:deriv}
holds uniformly over $\lambda\in\Lambda_{n}$, in the sense that
\[
\sup_{(\beta,\lambda)\in\Beta\times\Lambda_{n}}\smlnorm{\partial_{\beta}^{l}\hat{\theta}_{n}^{m}(\beta,\lambda)-\partial_{\beta}^{l}\theta(\beta,\lambda)}=o_{p}(1).
\]
A similar interpretation applies to Theorems~\ref{thm:rootdist}--\ref{thm:optim}
below.
\end{rem}

Proposition~5.1 is proved in \appref{sufficiency}.

\subsection{Application to examples\label{sub:applexamples}}

We may verify that each of the models from \secref{model} satisfy
\enuref{L:linear}--\enuref{L:moment}. In all cases, $x_{i}$ collects
all the (unique) elements of $\{x_{it}\}_{t=1}^{T}$, together with
any additional exogenous covariates used to estimate the auxiliary
model; while $\eta_{i}$ collects the elements of $\{\eta_{it}\}_{t=1}^{T}$.
Note that for the discrete choice Models~\ref{mod:serialprobit}--\ref{mod:trichotomous},
since the $\eta_{i}$ are Gaussian \enuref{L:moment} will be satisfied
if $\expect\smlnorm{x_{i}}^{3p_{0}}<\infty$. \enuref{L:nonsingular}
is a standard non-degeneracy condition.

\setcounter{model}{0}
\begin{model}
$u_{it}=bx_{it}+\sum_{s=1}^{t}r^{t-s}\eta_{is}$ by backward substitution.
So we set $(d_{\vidx},d_{\widx})=(T,0)$, with
\[
\vidx_{t}(z_{i};\beta)=x_{t}(z_{i})b(\beta)+\sum_{s=1}^{t}\eta_{s}(z_{i})d_{ts}(\beta),
\]
where $\beta=(b,r)$, $b(\beta)=b$ and $d_{ts}(\beta)=r^{t-s}$;
while $x_{t}(z_{i})$ and $\eta_{s}(z_{i})$ select the appropriate
elements of $z_{i}$, which collects $\{x_{it}\}$, $\{\eta_{it}\}$,
and any other exogenous covariates used in the auxiliary model. Thus
\enuref{L:linear} and \enuref{L:gamma} hold (formally, take $\gamma(\beta)=(b(\beta),\{d_{ts}(\beta)\})$).
\enuref{L:marginals} follows from the $\eta_{t}(z_{i})$'s being
standard Gaussian.
\end{model}
 
\begin{model}
As per the discussion in \subref{dynamic}, and \eqref{dynright}
in particular, we define
\[
\vidx_{tk}(z_{i};\beta)\defeq x_{t}(z_{i})b_{1}(\beta)+b_{2}(\beta)\indic\{k=1\}+\sum_{s=1}^{t}\eta_{s}(z_{i})d_{ts}(\beta)
\]
where the right-hand side quantities are defined by analogy with the
preceding example. Setting
\begin{equation}
y_{t}(\beta,\lambda)\defeq K_{\lambda}[\vidx_{t0}(\beta)]\cdot[1-y_{t-1}(\beta,\lambda)]+K_{\lambda}[\vidx_{t1}(\beta)]\cdot y_{t-1}(\beta,\lambda)\label{eq:recurse}
\end{equation}
with $y_{0}(\beta,\lambda)\defeq0$ thus yields smoothed choices having
the form required by \enuref{L:linear} and \enuref{L:gamma}, as
may be easily verified by backwards substitution. \enuref{L:marginals}
again follows from Gaussianity of $\eta_{t}(z_{i})$.
\end{model}

An identical recursion to \eqref{recurse} also works for \modref{initialprob}.
\modref{trichotomous} may be handled in a similar way to \modref{serialprobit},
but it is in certain respects simpler, because the errors are not
serially dependent. Finally, it remains to consider:

\setcounter{model}{4}
\begin{model}
From the preceding examples, it is clear that $\widx(z_{i};\beta)=w_{i}$
and $\vidx(z_{i};\beta)=u_{i}$ can be written in the linear index
form \eqref{vwidx}. The observable outcomes are the individual's
decision to work, and also his wage if he decides to work. These may
be smoothly approximated by: 
\begin{align*}
y_{1}(\beta,\lambda) & \defeq K_{\lambda}[\vidx(\beta)] & y_{2}(\beta,\lambda) & \defeq\widx(\beta)\cdot K_{\lambda}[\vidx(\beta)].
\end{align*}
respectively. Thus \enuref{L:linear}--\enuref{L:marginals} hold
just as in the other models. \enuref{L:moment} holds, in this case,
if $\expect\smlnorm{z_{i}}^{4p_{0}}<\infty$.
\end{model}

\subsection{Limiting distributions of GII estimators\label{sub:giidist}}

We now present our asymptotic results. Note that Assumptions~\ref{ass:regularity}
and \ref{ass:highlevel} are maintained throughout the following (even
if not explicitly referenced), though in accordance with \remref{deriv}
above, we shall always explicitly state the order of $l$ in \enuref{H:deriv}
that is required for each of our theorems.

Our first result concerns the limiting distributions of the minimizers
of the Wald and LR criterion functions, as displayed in \eqref{criteria}
above. For $e\in\{\Wald,\LR\}$, let $\hat{\beta}_{nk}^{e}$ be a
near-minimizer of $Q_{nk}^{e}$, in the sense that
\begin{equation}
Q_{nk}^{e}(\hat{\beta}_{nk}^{e},\lambda_{n})\leq\inf_{\beta\in\Beta}Q_{nk}^{e}(\beta,\lambda_{n})+o_{p}(n^{-1}).\label{eq:nearmin}
\end{equation}
The limiting variance of both estimators will have the familiar sandwich
form. To allow the next result to be stated succinctly, define
\begin{equation}
\Omega(U,V)\defeq(G^{\trans}UG)^{-1}G^{\trans}UH^{-1}VH^{-1}UG(G^{\trans}UG)^{-1}\label{eq:Omega}
\end{equation}
where $G\defeq[\partial_{\beta}\theta(\beta_{0},0)]^{\trans}$ denotes
the Jacobian of the binding function at $(\beta_{0},0)$, $H=\expect\ddot{\L}_{n}(\theta)$,
and $U$ and $V$ are symmetric matrices.
\begin{thm}[limiting distributions]
\label{thm:limitdist} Suppose \enuref{H:deriv} holds with $l=0$.
Then
\[
n^{1/2}(\hat{\beta}_{nk}^{e}-\beta_{0})\wkc N[0,\Omega(U_{e},V_{e})],
\]
where 
\begin{align}
U_{e} & \defeq\begin{cases}
W & \text{if }e=\Wald\\
H & \text{if }e=\LR
\end{cases} & V_{e} & \defeq\left(1+\frac{1}{M}\right)(\varmat-\covmat)\label{eq:UVe}
\end{align}
\end{thm}
\begin{rem}
In view of \propref{sufficiency} and the remark that follows it,
\thmref{limitdist} does \emph{not} restrict the rate at which $\lambda_{n}\inprob0$
from \emph{below}; indeed, it continues to hold even if $\lambda_{n}=0$
for all $n$, in which case the estimation problem is closely related
to that considered by \citet{PakP89Ecta}. Thus, while the theorem
provides the ``desired'' limiting distribution for our estimators,
it fails to provide a justification (or motivation) for the smoothing
proposed in this paper, and is in this sense unsatisfactory (or incomplete).
\end{rem}
 
\begin{rem}
Note that the order of jackknifing does not affect the limiting distribution
of the estimator: this has only a second-order effect, which vanishes
as $\lambda_{n}\goesto0$.
\end{rem}
 
\begin{rem}
\label{rem:twoaux}It is possible to define the LR estimator as the
minimizer of 
\[
Q_{n}^{\LR}(\beta)\defeq-\tilde{\L}_{n}(y;x,\abv{\theta}_{n}^{k}(\b,\lambda_{n})),
\]
where the average log-likelihood $\tilde{\L}_{n}$ need not correspond
to that maximized by $\hat{\theta}_{n}^{m}$, provided that the maximizers
of both $\L_{n}$ and $\tilde{\L}_{n}$ are consistent for the same
parameters. For example, $\hat{\theta}_{n}^{m}$ might be OLS (and
the associated residual covariance estimators), whereas $\tilde{\L}_{n}$
is the average log-likelihood for a SUR model. Suppose that the maximizer
$\tilde{\theta}_{n}$ of $\tilde{\L}_{n}$ satisfies the following
analogue of \eqref{psitophi},
\[
n^{1/2}(\tilde{\theta}_{n}-\theta_{0})=-\tilde{H}^{-1}\frac{1}{n^{1/2}}\sum_{i=1}^{n}\dot{\tilde{\like}}_{i}^{m}(\beta_{0},0;\theta_{0})+o_{p}(1)\wkc-\tilde{H}^{-1}\tilde{\phi}^{0},
\]
and define $\tilde{\varmat}\defeq\expect\tilde{\phi}^{0}\tilde{\phi}^{0\trans}$,
and $\tilde{\covmat}\defeq\expect\tilde{\phi}^{0}\phi^{m\trans}$
for $m\geq1$. Then the conclusions of \thmref{limitdist} continue
to hold, except that the $H^{-1}VH^{-1}$ appearing in \eqref{Omega}
must be replaced by
\begin{equation}
\tilde{H}^{-1}\tilde{\varmat}\tilde{H}^{-1}-(\tilde{H}^{-1}\tilde{\covmat}H^{-1}+H^{-1}\tilde{\covmat}^{\trans}\tilde{H}^{-1})+H^{-1}\left[\frac{1}{M}\varmat+\left(1-\frac{1}{M}\right)\covmat\right]H^{-1}\label{eq:innerpart}
\end{equation}
and $U_{\LR}=\tilde{H}$, where $\tilde{H}\defeq\expect\ddot{\tilde{\L}}_{n}(y,x;\theta_{0})$.
Regarding the estimation of these quantities, see \remref{twoauxest}
below. (Note that \eqref{innerpart} reduces to $H^{-1}VH^{-1}$ when
$\tilde{H}=H$, $\tilde{\varmat}=\varmat$ and $\tilde{\covmat}=\covmat$.)
\end{rem}

The proofs of \thmref{limitdist} and all other theorems in this paper
are given in \appref{mainproofs}.

\subsection{Convergence of smoothed derivatives and variance estimators\label{sub:performance}}

\thmref{limitdist} fails to indicate the possible benefits of smoothing,
because it simply posits the existence of a near-minimizer of $Q_{nk}$,
and thus entirely ignores how such a minimizer might be computed in
practice. Ideally, smoothing should be shown to facilitate the convergence
of derivative-based optimization procedures, when these are applied
to the problem of minimizing $Q_{nk}$, while still yielding an estimator
having the same limit distribution as in \thmref{limitdist}.

For the analysis of these procedures, the large-sample behavior of
the derivatives of $Q_{nk}$ will naturally play an important role.\footnote{Here, as throughout the remainder of this paper, we are concerned
exclusively with the limiting behavior of the \emph{exact} derivatives
of $Q_{nk}$, ignoring any errors that might be introduced by numerical
differentiation.} The uniform convergence of the derivatives of the sample binding
function -- and hence those of $Q_{nk}$ -- follows immediately from
\enuref{H:deriv}, and sufficient conditions for this convergence
are provided by \propref{sufficiency} above. Notably, when $l^{\prime}\in\{1,2\}$,
\eqref{lambdacond} imposes exactly the sort of lower bound on $\lambda_{n}$
that is absent from \thmref{limitdist}. 

\enuref{H:deriv}, with $l=1$, implies that the derivatives of the
smoothed criterion function can be used to estimate the Jacobian matrix
$G$ that appears in the limiting variances in \thmref{limitdist}.
The remaining components, $H$ and $V_{e}$, can be respectively estimated
using the data-based auxiliary log-likelihood Hessian, and an appropriate
transformation of the joint sample variance of all the auxiliary log-likelihood
scores (i.e.\ using both the data- and simulation-based models).
Define
\begin{gather*}
A^{\trans}\defeq\begin{bmatrix}I_{d_{\theta}} & -\tfrac{1}{M}I_{d_{\theta}} & \cdots & -\tfrac{1}{M}I_{d_{\theta}}\end{bmatrix}\\
s_{ni}^{\trans}\defeq\begin{bmatrix}\dot{\like}_{i}^{0}(\hat{\theta}_{n})^{\trans} & \dot{\like}_{i}^{1}(\hat{\beta}_{nk}^{e},\lambda_{n};\hat{\theta}_{n}^{1})^{\trans} & \cdots & \dot{\like}_{i}^{M}(\hat{\beta}_{nk}^{e},\lambda_{n};\hat{\theta}_{n}^{M})^{\trans}\end{bmatrix},
\end{gather*}
where $\hat{\theta}_{n}^{m}\defeq\hat{\theta}_{n}^{m}(\hat{\beta}_{nk}^{e},\lambda_{n})$,
and $\dot{\like}_{i}^{0}(\theta)$ denotes the gradient of $\like(y_{i},x_{i};\theta)$.
Then we have
\begin{thm}[variance estimation]
\label{thm:feasible} Suppose \enuref{H:deriv} holds with $l=0$.
Then
\begin{enumerate}
\item \label{enu:feas:hess}$\hat{H}_{n}\defeq\ddot{\Like}_{n}(\hat{\theta}_{n})\inprob H$; 
\item \label{enu:feas:score}$\hat{V}_{n}\defeq A^{\trans}\left(\frac{1}{n}\sum_{i=1}^{n}s_{ni}s_{ni}^{\trans}\right)A\inprob V$;
and
\end{enumerate}
if \enuref{H:deriv} holds with $l=1$, then
\begin{enumerate}[resume]
\item \label{enu:feas:G} $\hat{G}_{n}\defeq\partial_{\beta}\abv{\theta}_{n}(\hat{\beta}_{nk}^{e},\lambda_{n})\inprob G$,
for $e\in\{\Wald,\LR\}$.
\end{enumerate}
\end{thm}
 
\begin{rem}
\label{rem:twoauxest} For the situation envisaged in \remref{twoaux},
so long as the auxiliary model corresponding to $\tilde{\Like}_{n}$
satisfies Assumptions~\ref{ass:regularity} and \ref{ass:highlevel},
a consistent estimate of \eqref{innerpart} can be produced in the
manner of \enuref{feas:score} above, if we replace $s_{ni}$ by
\[
\tilde{s}_{ni}^{\trans}\defeq\begin{bmatrix}\dot{\tilde{\like}}_{i}^{0}(\tilde{\theta}_{n})^{\trans}\hat{H}_{n}^{-1} & \dot{\like}_{i}^{1}(\hat{\beta}_{nk}^{\LR},\lambda_{n};\hat{\theta}_{n}^{1})^{\trans}\tilde{H}_{n}^{-1} & \cdots & \dot{\like}_{i}^{M}(\hat{\beta}_{nk}^{\LR},\lambda_{n};\hat{\theta}_{n}^{M})^{\trans}\tilde{H}_{n}^{-1}\end{bmatrix},
\]
where $\tilde{H}_{n}\defeq\ddot{\tilde{\Like}}_{n}(\tilde{\theta}_{n})$
is consistent for $U_{\LR}$.
\end{rem}

\subsection{Performance of derivative-based optimization procedures\label{sub:optperform}}

The potential gains from smoothing may be assessed by comparing the
performance of derivative-based optimization procedures, as they are
applied to each of the following:
\begin{enumerate}[{label=\textnormal{\smaller[0.76]{P\arabic*}}}]
\item \label{enu:smplprob}the smoothed sample problem, of minimizing $\beta\elmap Q_{nk}(\beta,\lambda_{n})$;
and
\item \label{enu:popprob}its population counterpart, of minimizing $\beta\elmap Q_{k}(\beta,0)$.
\end{enumerate}
Since $Q_{k}$ is automatically smooth (even when $\lambda=0$, owing
to the smoothing effected by the expectation operator), derivative-based
methods ought to be particularly suited to solving \enuref{popprob},
and we may regard their performance when applied to this problem as
representing an upper bound for their performance when applied to
\enuref{smplprob}.

In the following section, we shall discuss in detail the convergence
properties of three popular optimization routines: Gauss-Newton; quasi-Newton
with BFGS updating; and a trust-region method. But before coming to
these, we first provide a result that is of relevance to a broader
range of derivative-based optimization procedures. Since such procedures
will typically be designed to terminate at (near) roots of the first-order
conditions, 
\begin{align*}
\partial_{\beta}Q_{nk}^{e}(\beta,\lambda_{n}) & =0\text{ in \enuref{smplprob}} & \partial_{\beta}Q^{e}(\beta,0) & =0\text{ in \enuref{popprob}}
\end{align*}
for $e\in\{\Wald,\LR\}$, we shall provide conditions on $\lambda_{n}$,
under which, for some $c_{n}=o_{p}(1)$,
\begin{enumerate}
\item \label{enu:roots}the set $R_{nk}^{e}\defeq\{\beta\in\Beta\mid\smlnorm{\partial_{\beta}Q_{nk}^{e}(\beta,\lambda_{n})}\leq c_{n}\}$
of near roots is ``consistent'' for subsets of $R^{e}\defeq\{\beta\in\Beta\mid\partial_{\beta}Q^{e}(\beta,0)=0\}$;
and
\item \label{enu:rootlim}if $c_{n}=o_{p}(n^{-1/2})$, then any $\tilde{\beta}_{n}\in R_{nk}^{e}$
with $\tilde{\beta}_{n}\inprob\beta_{0}$ has the limiting distribution
given by \thmref{limitdist}.
\end{enumerate}
We interpret \enuref{roots} as saying that smoothing yields a sample
problem \enuref{smplprob} that is ``no more difficult'' than the
population problem \enuref{popprob}, in the sense that the set of
points to which derivative-based optimizers may converge to in \enuref{smplprob}
approximates its counterpart in \enuref{popprob}, as $n\goesto\infty$.
This is the strongest consistency result we can hope to prove here:
as $Q$ may have multiple stationary points, only one of which coincides
with its (assumed interior) global minimum, it cannot generally be
true that the whole of $R_{nk}^{e}$ will be consistent for $\beta_{0}$.
On the other hand, if we can select a consistent sequence of (near)
roots from $R_{nk}^{e}$, as in \enuref{rootlim}, then we may reasonably
hope that this estimator sequence will enjoy the same limiting distribution
as a (near) minimizer of $Q_{nk}$.

For $A,B\subseteq\Beta$, let $d_{L}(A,B)\defeq\sup_{a\in A}d(a,B)$
denote the one-sided distance from $A$ to $B$, which has the property
that $d_{L}(A,B)=0$ if and only if $A\subseteq B$. Recall the definition
of $\hat{\beta}_{nk}^{e}$ given in \eqref{nearmin} above. Properties
\enuref{roots} and \enuref{rootlim} above can be more formally expressed
as follows.
\begin{thm}[near roots]
\label{thm:rootdist} Suppose \enuref{H:deriv} holds with $l=1$.
Then
\begin{enumerate}
\item \label{enu:setcons} $R_{nk}^{e}$ is nonempty w.p.a.1., and $d_{L}(R_{nk}^{e},R^{e})\inprob0$;
\item \label{enu:consistentroot}if $c_{n}=o_{p}(n^{-1/2})$, $\tilde{\beta}_{n}\in R_{nk}^{e}$
and $\tilde{\beta}_{n}\inprob\beta_{0}$, then $n^{1/2}(\tilde{\beta}_{n}-\hat{\beta}_{nk}^{e})=o_{p}(1)$,
and so $\tilde{\beta}_{n}$ has the limiting distribution given by
\thmref{limitdist}; and
\item \label{enu:consistexists} any $\tilde{\beta}_{n}\in R_{nk}^{e}$
satisfying $Q_{nk}^{e}(\tilde{\beta}_{n})\leq\inf_{\beta\in R_{nk}^{e}}Q_{nk}^{e}(\beta)+o_{p}(1)$
has $\tilde{\beta}_{n}\inprob\beta_{0}$.
\end{enumerate}
\end{thm}
\begin{rem}
Of course, the requirement that $\tilde{\beta}_{n}\inprob\beta_{0}$
cannot be verified in practice; but one may hope to satisfy it by
running the optimization routine from $L$ different starting points
located throughout $\Beta$, obtaining a collection of terminal values
$\{\beta_{nl}\}_{l=1}^{L}$, and then setting $\tilde{\beta}_{n}=\beta_{nl}$
such that $Q_{nk}(\tilde{\beta}_{n})\leq Q_{nk}(\beta_{nl^{\prime}})$
for all $l^{\prime}\in\{1,\ldots,L\}$.
\end{rem}

Some optimization routines, such as the trust-region method considered
in the next section, may only be allowed to terminate when the second-order
conditions for a minimum are also satisfied. Defining
\begin{align}
S_{nk}^{e}\defeq\{\beta & \in R_{nk}^{e}\mid\varrho_{\min}[\partial_{\beta}^{2}Q_{nk}^{e}(\beta,\lambda_{n})]\geq0\} & S^{e} & \defeq\{\beta\in R^{e}\mid\varrho_{\min}[\partial_{\beta}^{2}Q^{e}(\beta,0)]\geq0\},\label{eq:S}
\end{align}
we have the following
\begin{thm}[near roots satisfying second-order conditions]
\label{thm:root2nd} Suppose \enuref{H:deriv} holds with $l=2$.
Then parts \enuref{setcons} and \enuref{consistentroot} of \thmref{rootdist}
hold with $S_{nk}^{e}$ and $S^{e}$ in place of $R_{nk}^{e}$ and
$R^{e}$ respectively.
\end{thm}
 
\begin{rem}
The utility of this result may be seen by considering a case in which
$Q$ has many stationary points, but only a single local minimum at
$\beta_{0}$. Then while \thmref{rootdist} only guarantees convergence
to one of these stationary points, \thmref{root2nd} ensures consistency
for $\beta_{0}$ -- at a cost of requiring that the routine also check
the second-order conditions for a minimum. This is why stronger conditions
must be imposed on $\lambda_{n}$ in \thmref{root2nd}; we now need
the \emph{second} derivatives of $Q_{nk}$ to provide reliable information
about the curvature of $Q_{k}$ in large samples.
\end{rem}

\subsection{Convergence results for specific procedures\label{sub:convergence}}

Our final result concerns the question of whether certain optimization
routines, if initialized from within an appropriate region of the
parameter space and iterated to convergence, will yield the maximizer
of $Q_{nk}$, and thus an estimator having the limiting distribution
displayed in \thmref{limitdist}. In some respects, our work here
is related to previous work on $k$-step estimators, which studies
the limiting behavior of estimators computed as the outcome of a sequence
of quasi-Newton iterations (see e.g.\ \citealp{Rob88Ecta}). However,
we shall depart from that literature in an important respect, by \emph{not}
requiring that our optimization routines be initialized by a sequence
of starting values $\beta_{n}^{(0)}$ that are assumed consistent
for $\beta_{0}$ (often at some rate). Rather, we shall require only
that $\beta_{n}^{(0)}\in\Beta_{0}\subset\Beta$ for a \emph{fixed}
region $\Beta_{0}$ satisfying the conditions noted below.

We consider two popular line-search optimization methods -- Gauss-Newton,
and quasi-Newton with BFGS updating -- as well as a trust-region algorithm.
When applied to the problem of minimizing an objective $Q$, each
of these routines proceed as follows: given an iterate $\beta^{(s)}$,
locally approximate $Q$ by the following quadratic model,
\begin{equation}
f_{(s)}(\beta)\defeq Q(\beta^{(s)})+\nabla_{(s)}^{\trans}(\beta-\beta^{(s)})+\tfrac{1}{2}(\beta-\beta^{(s)})^{\trans}\Delta_{(s)}(\beta-\beta^{(s)}),\label{eq:fapprox}
\end{equation}
where $\nabla_{(s)}\defeq\partial_{\beta}Q(\beta^{(s)})$. A new iterate
$\beta^{(s+1)}$ is then generated by approximately minimizing $f_{(s)}$
with respect to $\beta$. The main differences between these procedures
concern the choice of approximate Hessian $\Delta_{(s)}$, and the
manner in which $f_{(s)}$ is (approximately) minimized. A complete
specification of each of the methods considered here is provided in
\appref{routines} (see also \citealp{Fletch87}, and \citealp{NW06});
note that the Gauss-Newton method can only be applied to the Wald
criterion function, since only this criterion has the least-squares
form required by that method.

We shall impose the following conditions on the population criterion
$Q$, which are sufficient to ensure that each of these procedures,
once started from some $\beta^{(0)}\in\Beta_{0}$, will converge to
the global minimizer of $Q$. As noted above, since $Q$ may have
many other stationary points, $\Beta_{0}$ must be chosen so as to
exclude these (except when the trust region method is used); hence
our convergence results are of an essentially local character. (Were
we to relax this condition on $\Beta_{0}$, then the arguments yielding
\thmref{optim} below could be modified to establish that these procedures
always converge to \emph{some} stationary point of $Q$.) To state
our conditions, let $\sigma_{\min}(D)\defeq\varrho_{\min}^{1/2}(D^{\trans}D)$
denote the smallest singular value of a (possibly non-square) matrix
$D$, and recall $G(\beta)=[\partial_{\beta}\bind(\beta,0)]^{\trans}$,
the Jacobian of the binding function.

\setcounter{assumption}{14}
\begin{assumption}[optimization routines]
\label{ass:routine} Let $Q\in\{Q_{k}^{\Wald},Q_{k}^{\LR}\}$. Then
$\Beta_{0}=\Beta_{0}(Q)$ may be chosen as any compact subset of $\intr\Beta$
for which $\beta_{0}\in\intr\Beta_{0}$ and $\Beta_{0}=\{\beta\in\Beta\mid Q(\beta)\leq Q(\beta_{1})\}$
for some $\beta_{1}\in\Beta$; and either
\begin{enumerate}
\item[\GN{}]  $\smlnorm{G(\beta)^{\trans}Wg(\beta)}\neq0$ for all $\beta\in\Beta_{0}\backslash\{\beta_{0}\}$
and $\inf_{\beta\in\Beta_{0}}\sigma_{\min}[G(\beta)]>0$;
\item[\QN{}] $Q$ is strictly convex on $\Beta_{0}$; or
\item[\TR{}]  for every $\beta\in\Beta_{0}\backslash\{\beta_{0}\}$, $\smlnorm{\partial_{\beta}Q(\beta)}=0$
implies $\varrho_{\min}[\partial_{\beta}^{2}Q(\beta)]<0$.
\end{enumerate}
\end{assumption}
\begin{rem}
Note that $\smlnorm{G(\beta)^{\trans}Wg(\beta)}\neq0$ is equivalent
to $\smlnorm{\partial_{\beta}Q_{k}^{\Wald}(\beta)}\neq0$. Both \GN{}
and \QN{} thus imply that $Q$ has no stationary points in $\Beta_{0}$,
other than that which corresponds to the minimum at $\beta_{0}$.
\TR{}, on the other hand, permits such points to exist, provided
that they are not local minima. In this respect, it places the weakest
conditions on $Q$, and does so because the trust-region method utilizes
second-derivative information in a manner that the other two methods
do not.
\end{rem}

Before analyzing the convergence properties of these optimization
routines, we must first specify the conditions governing their termination.
Let $\{\beta^{(s)}\}$ denote the sequence of iterates generated by
a given routine $r$, from some starting point $\beta^{(0)}$. When
$r\in\{\gn,\qn\}$, we shall allow the optimization to terminate at
the first $s$ -- denoted $s^{\ast}$ -- for which a near root is
located, in the sense that $\smlnorm{\partial_{\beta}Q_{nk}^{e}(\beta^{(s)})}\leq c_{n}$,
where $c_{n}=o_{p}(n^{-1/2})$. That is, $s^{\ast}$ is the smallest
$s$ for which $\beta^{(s)}\in R_{nk}^{e}$. This motivates the definition,
for $r\in\{\gn,\qn\}$, of 
\begin{equation}
\abv{\beta}_{nk}^{e}(\beta^{(0)},r)\defeq\begin{cases}
\beta^{(s^{\ast})} & \text{if }\beta^{(s)}\in R_{nk}^{e}\text{ for some }s\in\naturals\\
\beta^{(0)} & \text{otherwise},
\end{cases}\label{eq:terminal}
\end{equation}
which describes the terminal value of the optimization routine, with
the convention that this is set to $\beta^{(0)}$ if a near root is
never located. In the case that $r=\tn$, we shall allow the routine
to terminate only at those near roots at which the second-order sufficient
conditions for a local minimum are also satisfied. In this way, $s^{\ast}$
now becomes the smallest $s$ for which $\beta^{(s)}\in S_{nk}^{\tn}$,
and $\abv{\beta}_{nk}^{e}(\beta^{(0)},\tn)$ may be defined exactly
as in \eqref{terminal}, except with $S_{nk}^{\tn}$ in place of $R_{nk}^{\tn}$.\footnote{It may be asked why we do not also propose checking the second-order
conditions upon termination when $r\in\{\gn,\qn\}$. Such a modification
is certainly possible, but is perhaps of doubtful utility. Consider
the problem of minimizing some (deterministic) criterion function
that has multiple roots, only one of which corresponds to a local
(and also global) minimum, a scenario envisaged in \TR{}. In this
case, the best we can hope to prove is that the Gauss-Newton and quasi-Newton
routines will have \emph{some} of those roots as points of accumulation,
but they might \emph{never} enter the vicinity of the local minimum
(see Theorems~6.5 and 10.1 in \citealp{NW06}). On the other hand,
the trust-region algorithm considered here is guaranteed to have the
local minimum as a point of accumulation, under certain conditions
(see \citealp{MS83SIAM}, Theorem~4.13).}

For the purposes of the next result, let $\hat{\beta}_{nk}^{e}$ denote
the exact minimizer of $Q_{nk}^{e}$.
\begin{thm}[derivative-based optimizers]
\label{thm:optim} Suppose $r\in\{\gn,\qn,\tn\}$ and $e\in\{\Wald,\LR\}$,
and that the corresponding part of \assref{routine} holds for some
$\Beta_{0}$. Then 
\[
\sup_{\beta^{(0)}\in\Beta}\smlnorm{\abv{\beta}_{nk}^{e}(\beta^{(0)},r)-\hat{\beta}_{nk}^{e}}=o_{p}(n^{-1/2})
\]
holds if either
\begin{enumerate}
\item $(r,e)=(\gn,\Wald)$ and \enuref{H:deriv} holds with $l=1$; or
\item $r\in\{\qn,\tn\}$ and \enuref{H:deriv} holds with $l=2$.
\end{enumerate}
\end{thm}
\begin{rem}
\label{rem:cvgconditions} Convergence of the Gauss-Newton method
requires the weakest conditions on $\lambda_{n}$ of all three algorithms.
This is because the Hessian approximation $\Delta_{n,(s)}\defeq G_{n}(\beta^{(s)})^{\trans}W_{n}G_{n}(\beta^{(s)})$
used by Gauss-Newton is valid for criteria having the same minimum-distance
structure as $Q_{n}^{\Wald}$; here $G_{n}(\beta)\defeq\partial_{\beta}\abv{\theta}_{n}^{k}(\beta,\lambda_{n})$.
Thus the uniform convergence of $G_{n}$ is sufficient to ensure that
$\Delta_{n,(s)}$ behaves suitably in large samples, whence only \enuref{H:deriv}
with $l=1$ is required.
\end{rem}

\section{Monte Carlo results\label{sec:monte}}

This section conducts a set of Monte Carlo experiments to assess the
performance of the GII estimator, in terms of bias, efficiency, and
computation time. The parameters of Models~1--4 (see \secref{model})
are estimated a large number of times using ``observed'' data generated
by the respective models. For each model, the Monte Carlo experiments
are conducted for several sets of parameter configurations. For Models~\ref{mod:serialprobit},
\ref{mod:dynprobit}, and \ref{mod:trichotomous}, the parameters
are estimated in each Monte Carlo replication using both GII and simulated
maximum likelihood (SML) in conjunction with the GHK smooth probability
simulator (cf.\ \citealp{Lee97JoE}). \modref{initialprob}, which
cannot easily be estimated via SML, is estimated using only GII. We
omit Model~5, as \citet{ASV13Ecta} already present results showing
that GII performs well for Heckman selection-type models.

In all cases, we use the LR approach to (generalized) indirect inference
to construct our estimates. We do this for two reasons. First, unlike
the Wald and LM approaches, the LR approach does not require the estimation
of a weight matrix. In this respect, the LR approach is easier to
implement than the other two approaches. Furthermore, because estimates
of optimal weight matrices often do not perform well in finite samples
(see e.g.\ \citealp{AS96JBES}), the LR approach is likely to perform
better in small samples. Second, because the LR approach is asymptotically
equivalent to the other two approaches when the auxiliary model is
correctly specified, the relative inefficiency of the LR estimator
is likely to be small when the auxiliary model is chosen judiciously.

To optimize the criterion functions, we use a version of the Davidon-Fletcher-Powell
algorithm (as implemented in Chapter~10 of \citealp{Press93book}),
which is closely related to the quasi-Newton routine analyzed in \subref{convergence}.
The initial parameter vector in the hillclimbing algorithm is the
true parameter vector. Most of the computation time in generalized
indirect inference lies in computing ordinary least squares (OLS)
estimates. The main cost in computing OLS estimates lies, in turn,
in computing the $X^{\trans}X$ part of $(X^{\trans}X)^{-1}X^{\trans}Y$.
We use blocking and loop unrolling techniques to speed up the computation
of $X^{\trans}X$ by a factor of 2 to 3 relative to a ``naive''
algorithm.\footnote{To avoid redundant calculations, we also precompute and store for
later use those elements of $X^{\trans}X$ that depend only on the
exogenous variables. We are grateful to James MacKinnon for providing
code that implements the blocking and loop unrolling techniques.}

\subsection{Results for \modref{serialprobit}}

\modref{serialprobit} is a two-alternative panel probit model with
serially correlated errors and one exogenous regressor. It has two
unknown parameters: the regressor coefficient $b$, and the serial
correlation parameter $r$. We set $b=1$ and consider $r\in\{0,0.40,0.85\}$.
In the Monte Carlo experiments, $n=1000$ and $T=5$. As in all of
the simulation exercises carried out in this paper, we compute the
GII estimator via the two-step approach described in \subref{onestep},
using $(\lambda,M)=(0.03,10)$ in the first step, and $(\lambda,M)=(0.003,300)$
in the second. The exogenous variables (the $x_{it}$'s) are i.i.d.\ draws
from a $N[0,1]$ distribution, drawn anew for each Monte Carlo replication.

The auxiliary model consists of $T$ linear probability models of
the form
\[
y_{it}=z_{it}^{\trans}\a_{t}+\xi_{it}
\]
where $\xi_{it}\distiid N[0,\sigma_{t}^{2}]$, $z_{it}$ denotes the
vector of regressors for individual $i$ in time period $t$, and
$\a_{t}$ and $\sigma_{t}^{2}$ are parameters to be estimated. We
include in $z_{it}$ both lagged choices and polynomial functions
of current and lagged exogenous variables; the included variables
change over time, so as to allow the auxiliary model to incorporate
the additional lagged information that is available in later time
periods. (When estimating the model on simulated data, the simulated
lagged choices are of course replaced by their smoothed counterparts,
as per the discussion in \subref{dynamic} above.) The auxiliary model
is thus characterized by the parameters $\t=\{\a_{t},\sigma_{t}^{2}\}_{t=1}^{T}$;
these are estimated by maximum likelihood (which corresponds to OLS
here, under the distributional assumptions on $\xi_{it}$).

It is worth emphasizing that we include lagged choices (and lagged
$x$'s) in the auxiliary model despite the fact that the structural
model does not exhibit true state dependence. But in Model 1 it is
well-know that lagged choices are predictive of current choices (termed
``spurious state dependence'' by Heckman). This is a good illustration
of how a good auxiliary model should be designed to capture the correlation
patterns in the data, as opposed to the true structure.

To examine how increasing the ``richness'' of the auxiliary model
affects the efficiency of the structural parameter estimates, we conduct
Monte Carlo experiments using four nested auxiliary models. In all
four, we impose the restrictions $\a_{t}=\a_{q}$ and $\sigma_{t}^{2}=\sigma_{q}^{2}$,
$t=q+1,\ldots,T$, for some $q<T$. This is because the time variation
in the estimated coefficients of the linear probability models comes
mostly from the non-stationarity of the errors in the structural model,
and so it is negligible after the first few time periods (we do not
assume that the initial error is drawn from the stationary distribution
implied by the law of motion for the errors).

In auxiliary model \#1, $q=1$ and the regressors in the linear probability
model are given by: $z_{it}=(1,x_{it},y_{i,t-1})$, $t=1,\ldots,T$,
where the unobserved $y_{i0}$ is set equal to 0. We use this very
simple auxiliary model to illustrate how GII can produce very inefficient
estimates if one uses a poor auxiliary model. In auxiliary model \#2,
$q=2$ and the regressors are $z_{i1}=(1,x_{i1})$, and 
\begin{align*}
z_{it} & =(1,x_{it},y_{i,t-1},x_{i,t-1}),\ t\in\{2,\ldots,T\},
\end{align*}
giving a total of 18 parameters. Auxiliary model \#3 has $q=4$, regressors
\begin{align*}
z_{i1} & =(1,x_{i1},x_{i1}^{3}) & z_{i3} & =(1,x_{i3},y_{i2},x_{i2},y_{i1},x_{i1})\\
z_{i2} & =(1,x_{i2},y_{i1},x_{i1}) & z_{it} & =(1,x_{it},y_{i,t-1},x_{i,t-1},y_{i,t-2},x_{i,t-2},y_{i,t-3}),\ t\in\{4,\ldots,T\},
\end{align*}
and 24 parameters. Finally, auxiliary model \#4 has the same regressors
as \#3, except that
\begin{align*}
z_{i4} & =(1,x_{i4},y_{i3},x_{i3},y_{i2},x_{i2},y_{i1},x_{i1})\\
z_{it} & =(1,x_{it},y_{i,t-1},x_{i,t-1},y_{i,t-2},x_{i,t-2},y_{i,t-3},x_{i,t-3},y_{i,t-4}),\ t\in\{5,\ldots,T\}
\end{align*}
so $q=5$ and there are 35 parameters.

Table 1 presents the results of six sets of Monte Carlo experiments,
each with 2000 replications. The first two sets of experiments report
the results for simulated maximum likelihood, based on GHK, using
25 draws (SML \#1) and 50 draws (SML \#2). The remaining four sets
of experiments report the results for generalized indirect inference,
where GII \#$i$ refers to generalized indirect inference using auxiliary
model \#$i$. In each case, we report the average and the standard
deviation of the parameter estimates. We also report the efficiency
loss of GII \#$i$ relative to SML \#2 in the columns labelled $\sigma_{\mathrm{GII}}/\sigma_{\mathrm{SML}}$,
where we divide the standard deviations of the GII estimates by the
standard deviations of the estimates for SML \#2. Finally, we report
the average time (in seconds) required to compute estimates (we use
the Intel Fortran Compiler Version 7.1 on a 2.2GHz Intel Xeon processor
running Red Hat Linux).

Table 1 contains several key findings:

First, both SML and GII generate estimates with very little bias.

Second, GII is less efficient than SML, but the efficiency losses
are small provided that the auxiliary model is sufficiently rich.
For example, auxiliary model \#1 leads to large efficiency losses,
particularly for the case of high serial correlation in the errors
($r=0.85$). For models with little serial correlation ($r=0$), however,
auxiliary model \#2 is sufficiently rich to to make GII almost as
efficient as SML. When there is more serial correlation in the errors,
auxiliary model \#2 leads to reasonably large efficiency losses (as
high as 30\% when $r=0.85$), but auxiliary model \#3, which contains
more lagged information in the linear probability models than does
auxiliary model \#2, reduces the worst efficiency loss to 13\%. Auxiliary
model \#4 provides almost no efficiency gains relative to auxiliary
model \#3.

Third, GII is faster than SML: computing a set of estimates using
GII with auxiliary model \#3 takes about 30\% less time than computing
a set of estimates using SML with 50 draws.

For generalized indirect inference, we also compute (but do not report
in Table 1) estimated asymptotic standard errors, using the estimators
described in \thmref{feasible}. In all cases, the averages of the
estimated standard errors across the Monte Carlo replications are
very close to (within a few percent of) the actual standard deviations
of the estimates, suggesting that the asymptotic results provide a
good approximation to the behavior of the estimates in samples of
the size that we use.

\subsection{Results for \modref{dynprobit}\label{sub:dynprobit}}

\modref{dynprobit} is a panel probit model with serially correlated
errors, a single exogenous regressor, and a lagged dependent variable.
It has three unknown parameters: $b_{1}$, the coefficient on the
exogenous regressor, $b_{2}$, the coefficient on the lagged dependent
variable, and $r$, the serial correlation parameter. We set $b_{1}=1$,
$b_{2}=0.2$, and consider $r\in\{0,0.4,0.85\}$; $n=1000$ and $T=10$.

Table 2 presents the results of six sets of Monte Carlo experiments,
each with 1000 replications; the labels SML \#$i$ and GII \#$i$
are to be interpreted exactly as for Table 1. The results are similar
to those for Model 1. Both SML and GII generate estimates with very
little bias. SML is more efficient than GII, but the efficiency loss
is small when the auxiliary model is sufficiently rich (i.e., 17\%
at most for model \#3, 15\% at most for model \#4). However, auxiliary
model \#1 can lead to very large efficiency losses, as can auxiliary
model \#2 if there is strong serial correlation. 

Again, average asymptotic standard errors are close to the standard
deviations obtained across the simulations (not reported). Finally,
GII using auxiliary model \#3 is about 25\% faster than SML using
50 draws.

\subsection{Results for \modref{initialprob}\label{sub:initialprob}}

\modref{initialprob} is identical to \modref{dynprobit}, except
there is an ``initial conditions'' problem: the econometrician does
not observe individuals' choices in the first $s$ periods. This
is an excellent example of the type of problem that motivates this
paper: SML is extremely difficult to implement, due to the problem
of integrating over the initial conditions. But II is appealing, as
it is still trivial to simulate data from the model. However, we need
GII to deal with the discrete outcomes. 

To proceed, our Monte Carlo experiments are parametrized exactly as
for \modref{dynprobit}, except that we set $T=15$, with choices
in the first $s=5$ time periods being unobserved (but note that exogenous
variables \emph{are} observed in these time periods).

Auxiliary model \#1 is as for Models~\ref{mod:serialprobit} and
\ref{mod:dynprobit}: $q=1$ and the regressors are $z_{it}=(1,x_{it},y_{i,t-1})$,
$t=s+1,\ldots,T$, where the unobserved $y_{is}$ is set equal to
0. In auxiliary model \#2, $q=2$ and the regressors are: 
\begin{align*}
z_{i,s+1} & =(1,x_{i,s+1},x_{is}) & z_{it} & =(1,x_{it},y_{i,t-1},x_{i,t-1}),\ t\in\{s+2,\ldots,T\},
\end{align*}
for a total of 19 parameters. In auxiliary model \#3, $q=4$ and there
are 27 parameters: 
\begin{align*}
z_{i,s+1} & =(1,x_{i,s+1},x_{i,s+1}^{3},x_{is},x_{i,s-1})\\
z_{i,s+2} & =(1,x_{i,s+2},y_{i,s+1},x_{i,s+1},x_{is})\\
z_{i,s+3} & =(1,x_{i,s+3},y_{i,s+2},x_{i,s+2},y_{i,s+1},x_{i,s+1})\\
z_{it} & =(1,x_{it},y_{i,t-1},x_{i,t-1},y_{i,t-2},x_{i,t-2},y_{i,t-3}),\ t\in\{s+4,\ldots,T\}
\end{align*}
Finally, in auxiliary model \#4, $q=5$ and there are 41 parameters:
relative to \#3, $z_{i,s+1}$, $z_{i,s+2}$ and $z_{i,s+3}$ are augmented
by an additional lag of $x_{is}$, and 
\begin{align*}
z_{i,s+4} & =(1,x_{i,s+4},y_{i,s+3},x_{i,s+3},y_{i,s+2},x_{i,s+2},y_{i,s+1},x_{i,s+1})\\
z_{it} & =(1,x_{it},y_{i,t-1},x_{i,t-1},y_{i,t-2},x_{i,t-2},y_{i,t-3},x_{i,t-3},y_{i,t-4}),\ t\in\{s+5,\ldots,T\}.
\end{align*}

Table 3 presents the results of four sets of Monte Carlo experiments,
each with 1000 replications. There are two key findings: First, as
with Models~\ref{mod:serialprobit} and \ref{mod:dynprobit}, GII
generates estimates with very little bias. Second, increasing the
``richness'' of the auxiliary model leads to large efficiency gains
relative to auxiliary model \#1, particularly when the errors are
persistent. However, auxiliary model \#4 provides few efficiency gains
relative to auxiliary model \#3.

\subsection{Results for \modref{trichotomous}\label{sub:trichot}}

\modref{trichotomous} is a (static) three-alternative probit model
with eight unknown parameters: three coefficients in each of the two
equations for the latent utilities ($\{b_{1i}\}_{i=0}^{2}\}$ and
$\{b_{2i}\}_{i=0}^{2}$) and two parameters governing the covariance
matrix of the disturbances in these equations ($c_{1}$ and $c_{2}$).
We set $b_{10}=b_{20}=0$, $b_{11}=b_{12}=b_{21}=b_{22}=1$, $c_{2}=1$,
and consider $c_{1}\in\{0,1.33\}$ (implying that the disturbances
in the latent utilities are respectively independent, or have a correlation
of $0.8$). We set $n=2000$.

The auxiliary model is a pair of linear probability models, one for
each of the first two alternatives: 
\begin{align*}
y_{i1} & =z_{i}^{\trans}\alpha_{1}+\xi_{i1}\\
y_{i2} & =z_{i}^{\trans}\alpha_{2}+\xi_{i2},
\end{align*}
where $z_{i}$ consists of polynomial functions of the exogenous variables
$\{x_{ij}\}_{j=1}^{3}$, and $\xi_{i}\distiid N[0,\Sigma_{\xi}]$.
The auxiliary model parameters $\t=(\alpha_{1},\alpha_{2},\Sigma_{\xi})$
are estimated by OLS; this corresponds to maximum likelihood -- even
though $\Sigma_{\xi}$ is not diagonal -- because the same regressors
appear in both equations.

We conduct Monte Carlo experiments using four nested versions of the
auxiliary model. In auxiliary model \#1, $z_{i}=(1,x_{i1},x_{i2},x_{i3})$,
giving a total of 11 parameters. Auxiliary model \#2 adds all the
second-order products of these variables, as well as one third-order
product to $z_{i}$, i.e. 
\[
z_{i}=(1,x_{i1},x_{i2},x_{i3},x_{i1}^{2},x_{i2}^{2},x_{i3}^{2},x_{i1}x_{i2},x_{i1}x_{i3},x_{i2}x_{i3},x_{i1}x_{i2}x_{i3}),
\]
for a total of 25 parameters. In auxiliary model \#3, $z_{i}$ contains
all third-order products (for a total of 43 parameters) and in auxiliary
model \#4, $z_{i}$ contains all fourth-order products (for a total
of 67 parameters).

Tables 4 and 5 present the results of six sets of Monte Carlo experiments,
each with 1000 replications; the labels SML \#$i$ and GII \#$i$
are to be interpreted exactly as for Table 1. The key findings are
qualitatively similar to those for Models 1, 2, and 3. First, both
SML and GII generate estimates with very little bias. Second, auxiliary
model \#1, which contains only linear terms, leads to large efficiency
losses relative to SML (as large as 50\%). But auxiliary model \#2,
which contains terms up to second order, reduces the efficiency losses
substantially (to no more than 15\% when the errors are uncorrelated,
and to no more than 26\% when $c=1.33$). Auxiliary model \#3, which
contains terms up to third order, provides additional small efficiency
gains (the largest efficiency loss is reduced to 20\%), while auxiliary
model \#4, which contains fourth-order terms, provides few, if any,
efficiency gains relative to auxiliary model \#3. Finally, computing
estimates using GII with auxiliary model \#3 takes about 30\% less
time than computing estimates using SML with 50 draws.

\section{Conclusion}

Discrete choice models play an important role in many fields of economics,
from labor economics to industrial organization to macroeconomics.
Unfortunately, these models are usually quite challenging to estimate
(except in special cases like MNL where choice probabilities have
closed forms). Simulation-based methods like SML and MSM have been
developed that can be used for more complex models like MNP. But in
many important cases (models with initial conditions problems and
Heckman selection models being leading cases) even these methods are
very difficult to implement.

In this paper we develop and implement a new simulation-based method
for estimating models with discrete or mixed discrete/continuous outcomes.
The method is based on indirect inference. But the traditional II
approach is not easily applicable to discrete choice models because
one must deal with a non-smooth objective surface. The key innovation
here is that we develop a generalized method of indirect inference
(GII), in which the auxiliary models that are estimated on the actual
and simulated data may differ (provided that the estimates from both
models share a common probability limit). This allows us to chose
an auxiliary model for the simulated data such that we obtain an objective
function that is a smooth function of the structural parameters. This
smoothness renders GII practical as a method for estimating discrete
choice models.

Our theoretical analysis goes well beyond merely deriving the limiting
distribution of the minimizer of the GII criterion function. Rather,
in keeping with computational motivation of this paper, we show that
the proposed smoothing facilitates the convergence of derivative-based
optimizers, in the sense that the smoothing leads to a sample optimization
problem that is no more difficult than the corresponding population
problem, where the latter involves the minimization of a necessarily
smooth criterion. This provides a rigorous justification for using
standard derivative-based optimizers to compute the GII estimator,
which is also shown to inherit the limiting distribution of the (unsmoothed)
II estimator. Inferences based on the GII estimates may thus be drawn
in the standard manner, via the usual Wald statistics. Our results
on the convergence of derivative-based optimizers seem to be new to
the literature.

We also provide a set of Monte Carlo experiments to illustrate the
practical usefulness of GII. In addition to being robust and fast,
GII yields estimates with good properties in small samples. In particular,
the estimates display very little bias and are nearly as efficient
as maximum likelihood (in those cases where simulated versions of
maximum likelihood can be used) provided that the auxiliary model
is chosen judiciously.

GII could potentially be applied to a wide range of discrete and discrete/continuous
outcome models beyond those we consider in our Monte Carlo experiments.
Indeed, GII is sufficiently flexible to accommodate almost any conceivable
model of discrete choice, including, discrete choice dynamic programming
models, discrete dynamic games, etc. We hope that applied economists
from a variety of fields find GII a useful and easy-to-implement method
for estimating discrete choice models.

{\singlespacing\nocite{EM94Hdbk}

\section{References}

\bibliographystyle{econometrica}
\bibliography{asymptotics}

}

\noindent \pagebreak{}

\begin{center}
\begin{tabular}{l|cc|cc|cc|c|}
\multicolumn{8}{c}{\textbf{Table 1}}\tabularnewline
\multicolumn{8}{c}{Monte Carlo Results for Model 1}\tabularnewline
\multicolumn{8}{c}{}\tabularnewline
 & \multicolumn{2}{c|}{Mean} & \multicolumn{2}{c|}{Std.\ dev.} & \multicolumn{2}{c}{$\sigma_{\mathrm{GII}}/\sigma_{\mathrm{SML}}$} & Time \tabularnewline
 & $b$  & $r$  & $b$  & $r$  & $b$  & $r$  & (sec.) \tabularnewline
\hline 
\hline 
\multicolumn{8}{c}{}\tabularnewline
\multicolumn{8}{c}{$b=1$, $r=0$}\tabularnewline
\hline 
SML \#1  & 1.000  & $-0.002$  & 0.0387  & 0.0454  & ---  & ---  & 0.76 \tabularnewline
SML \#2  & 1.001  & $-0.000$  & 0.0373  & 0.0468  & ---  & ---  & 1.53 \tabularnewline
GII \#1  & 0.998  & $\phantom{-}0.002$  & 0.0390  & 0.0645  & 1.05  & 1.37  & 0.67 \tabularnewline
GII \#2  & 0.993  & $\phantom{-}0.001$  & 0.0386  & 0.0490  & 1.03  & 1.05  & 0.72 \tabularnewline
GII \#3  & 0.992  & $\phantom{-}0.001$  & 0.0393  & 0.0490  & 1.05  & 1.05  & 0.91 \tabularnewline
GII \#4  & 0.988  & $\phantom{-}0.001$  & 0.0390  & 0.0485  & 1.05  & 1.04  & 0.99 \tabularnewline
\hline 
\multicolumn{8}{c}{}\tabularnewline
\multicolumn{8}{c}{$b=1$, $r=0.4$}\tabularnewline
\hline 
SML \#1  & 0.995  & 0.385  & 0.0400  & 0.0413  & ---  & ---  & 0.78 \tabularnewline
SML \#2  & 0.999  & 0.392  & 0.0390  & 0.0410  & ---  & ---  & 1.54 \tabularnewline
GII \#1  & 0.998  & 0.399  & 0.0454  & 0.0616  & 1.16  & 1.50  & 0.70 \tabularnewline
GII \#2  & 0.993  & 0.396  & 0.0410  & 0.0456  & 1.05  & 1.11  & 0.72 \tabularnewline
GII \#3  & 0.991  & 0.395  & 0.0417  & 0.0432  & 1.07  & 1.05  & 0.91 \tabularnewline
GII \#4  & 0.987  & 0.392  & 0.0416  & 0.0432  & 1.07  & 1.05  & 0.97 \tabularnewline
\hline 
\multicolumn{8}{c}{}\tabularnewline
\multicolumn{8}{c}{$b=1$, $r=0.85$}\tabularnewline
\hline 
SML \#1  & 0.984  & 0.833  & 0.0452  & 0.0333  & ---  & ---  & 0.74 \tabularnewline
SML \#2  & 0.993  & 0.842  & 0.0432  & 0.0316  & ---  & ---  & 1.47 \tabularnewline
GII \#1  & 0.994  & 0.846  & 0.0791  & 0.0672  & 1.83  & 2.13  & 0.71 \tabularnewline
GII \#2  & 0.991  & 0.845  & 0.0511  & 0.0412  & 1.18  & 1.30  & 0.74 \tabularnewline
GII \#3  & 0.992  & 0.846  & 0.0492  & 0.0357  & 1.14  & 1.13  & 0.93 \tabularnewline
GII \#4  & 0.988  & 0.841  & 0.0490  & 0.0357  & 1.13  & 1.13  & 1.00 \tabularnewline
\hline 
\end{tabular}
\par\end{center}

\pagebreak{}

\begin{center}
\begin{tabular}{l|ccc|ccc|ccc|c|}
\multicolumn{11}{c}{\textbf{Table 2}}\tabularnewline
\multicolumn{11}{c}{Monte Carlo Results for Model 2}\tabularnewline
\multicolumn{11}{c}{}\tabularnewline
 & \multicolumn{3}{c|}{Mean} & \multicolumn{3}{c|}{Std.\ dev.} & \multicolumn{3}{c}{$\sigma_{\mathrm{GII}}/\sigma_{\mathrm{SML}}$} & Time \tabularnewline
 & $b_{1}$  & $r$  & $b_{2}$  & $b_{1}$  & $r$  & $b_{2}$  & $b_{1}$  & $r$  & $b_{2}$  & (sec.) \tabularnewline
\hline 
\hline 
\multicolumn{11}{c}{}\tabularnewline
\multicolumn{11}{c}{$b_{1}=1$, $r=0$, $b_{2}=0.2$}\tabularnewline
\hline 
SML \#1  & 1.000  & 0.001  & 0.200  & 0.0274  & 0.0357  & 0.0355  & ---  & ---  & ---  & 2.47 \tabularnewline
SML \#2  & 1.002  & 0.002  & 0.199  & 0.0273  & 0.0362  & 0.0365  & ---  & ---  & ---  & 4.89 \tabularnewline
GII \#1  & 0.999  & 0.001  & 0.199  & 0.0267  & 0.0571  & 0.0437  & 0.98  & 1.58  & 1.20  & 2.72 \tabularnewline
GII \#2  & 0.996  & 0.000  & 0.199  & 0.0267  & 0.0379  & 0.0379  & 0.98  & 1.05  & 1.04  & 2.80 \tabularnewline
GII \#3  & 0.995  & 0.001  & 0.199  & 0.0269  & 0.0377  & 0.0376  & 0.99  & 1.04  & 1.03  & 3.66 \tabularnewline
GII \#4  & 0.993  & 0.000  & 0.198  & 0.0270  & 0.0377  & 0.0375  & 0.99  & 1.04  & 1.03  & 4.06 \tabularnewline
\hline 
\multicolumn{11}{c}{}\tabularnewline
\multicolumn{11}{c}{$b_{1}=1$, $r=0.4$, $b_{2}=0.2$}\tabularnewline
\hline 
SML \#1  & 0.994  & 0.379  & 0.214  & 0.0278  & 0.0314  & 0.0397  & ---  & ---  & ---  & 2.42 \tabularnewline
SML \#2  & 0.999  & 0.389  & 0.206  & 0.0287  & 0.0316  & 0.0397  & ---  & ---  & ---  & 4.82 \tabularnewline
GII \#1  & 0.997  & 0.397  & 0.198  & 0.0339  & 0.0587  & 0.0544  & 1.18  & 1.86  & 1.37  & 2.73 \tabularnewline
GII \#2  & 0.994  & 0.396  & 0.198  & 0.0293  & 0.0386  & 0.0462  & 1.02  & 1.22  & 1.16  & 2.82 \tabularnewline
GII \#3  & 0.993  & 0.396  & 0.197  & 0.0289  & 0.0343  & 0.0431  & 1.01  & 1.09  & 1.09  & 3.64 \tabularnewline
GII \#4  & 0.991  & 0.395  & 0.196  & 0.0289  & 0.0348  & 0.0434  & 1.01  & 1.10  & 1.09  & 4.02 \tabularnewline
\hline 
\multicolumn{11}{c}{}\tabularnewline
\multicolumn{11}{c}{$b_{1}=1$, $r=0.85$, $b_{2}=0.2$}\tabularnewline
\hline 
SML \#1  & 0.974  & 0.831  & 0.220  & 0.0321  & 0.0174  & 0.0505  & ---  & ---  & ---  & 2.78 \tabularnewline
SML \#2  & 0.987  & 0.840  & 0.208  & 0.0327  & 0.0159  & 0.0507  & ---  & ---  & ---  & 5.47 \tabularnewline
GII \#1  & 1.000  & 0.854  & 0.183  & 0.0952  & 0.0633  & 0.1185  & 2.91  & 3.98  & 2.34  & 3.01 \tabularnewline
GII \#2  & 0.992  & 0.852  & 0.190  & 0.0417  & 0.0266  & 0.0721  & 1.28  & 1.67  & 1.42  & 2.92 \tabularnewline
GII \#3  & 0.992  & 0.851  & 0.191  & 0.0383  & 0.0179  & 0.0547  & 1.17  & 1.13  & 1.08  & 3.68 \tabularnewline
GII \#4  & 0.990  & 0.850  & 0.188  & 0.0379  & 0.0175  & 0.0548  & 1.15  & 1.10  & 1.09  & 4.06 \tabularnewline
\hline 
\end{tabular}
\par\end{center}

\pagebreak{}

\begin{center}
\begin{tabular}{l|ccc|ccc|c|}
\multicolumn{8}{c}{\textbf{Table 3}}\tabularnewline
\multicolumn{8}{c}{Monte Carlo Results for Model 3}\tabularnewline
\multicolumn{8}{c}{}\tabularnewline
 & \multicolumn{3}{c|}{Mean} & \multicolumn{3}{c}{Std.\ dev.} & Time \tabularnewline
 & $b_{1}$  & $r$  & $b_{2}$  & $b_{1}$  & $r$  & $b_{2}$  & (sec.) \tabularnewline
\hline 
\hline 
\multicolumn{8}{c}{}\tabularnewline
\multicolumn{8}{c}{$b_{1}=1$, $r=0$, $b_{2}=0.2$}\tabularnewline
\hline 
GII \#1  & 0.997  & $-0.000$  & 0.200  & 0.0272  & 0.0532  & 0.0387  & 3.91 \tabularnewline
GII \#2  & 0.994  & $-0.001$  & 0.200  & 0.0271  & 0.0387  & 0.0347  & 4.01 \tabularnewline
GII \#3  & 0.993  & $-0.001$  & 0.199  & 0.0272  & 0.0385  & 0.0345  & 4.81 \tabularnewline
GII \#4  & 0.991  & $-0.001$  & 0.199  & 0.0275  & 0.0389  & 0.0347  & 5.38 \tabularnewline
\hline 
\multicolumn{8}{c}{}\tabularnewline
\multicolumn{8}{c}{$b_{1}=1$, $r=0.4$, $b_{2}=0.2$}\tabularnewline
\hline 
GII \#1  & 0.994  & 0.397  & 0.198  & 0.0361  & 0.0518  & 0.0493  & 3.99 \tabularnewline
GII \#2  & 0.991  & 0.397  & 0.197  & 0.0309  & 0.0363  & 0.0430  & 4.00 \tabularnewline
GII \#3  & 0.990  & 0.396  & 0.196  & 0.0306  & 0.0317  & 0.0399  & 4.80 \tabularnewline
GII \#4  & 0.987  & 0.395  & 0.196  & 0.0302  & 0.0318  & 0.0400  & 5.35 \tabularnewline
\hline 
\multicolumn{8}{c}{}\tabularnewline
\multicolumn{8}{c}{$b_{1}=1$, $r=0.85$, $b_{2}=0.2$}\tabularnewline
\hline 
GII \#1  & 0.993  & 0.851  & 0.184  & 0.0936  & 0.0403  & 0.1289  & 4.41 \tabularnewline
GII \#2  & 0.986  & 0.851  & 0.191  & 0.0546  & 0.0249  & 0.0905  & 4.37 \tabularnewline
GII \#3  & 0.987  & 0.850  & 0.189  & 0.0430  & 0.0140  & 0.0598  & 4.93 \tabularnewline
GII \#4  & 0.984  & 0.849  & 0.185  & 0.0411  & 0.0136  & 0.0597  & 5.56 \tabularnewline
\hline 
\end{tabular}
\par\end{center}

\pagebreak{}

\begin{center}
\begin{tabular}{l|cc|cccc|cccc|}
\multicolumn{11}{c}{\textbf{Table 4}}\tabularnewline
\multicolumn{11}{c}{Monte Carlo Results for Model 4}\tabularnewline
\multicolumn{11}{c}{($b_{10}=0$, $b_{11}=1$, $b_{12}=1$, $b_{20}=0$, $b_{21}=1$,
$b_{22}=1$, $c_{1}=0$, $c_{2}=1$)}\tabularnewline
\multicolumn{11}{c}{}\tabularnewline
 & \multicolumn{2}{c|}{SML} & \multicolumn{4}{c|}{GII} & \multicolumn{4}{c|}{$\sigma_{\mathrm{GII}}/\sigma_{\mathrm{SML}}$}\tabularnewline
 & \#1  & \#2  & \#1  & \#2  & \#3  & \#4  & \#1  & \#2  & \#3  & \#4 \tabularnewline
\hline 
\hline 
\multicolumn{11}{c}{}\tabularnewline
\multicolumn{11}{c}{Mean}\tabularnewline
\hline 
$b_{10}$  & $\phantom{-}0.007$  & $\phantom{-}0.005$  & $\phantom{-}0.003$  & $\phantom{-}0.002$  & $\phantom{-}0.002$  & $\phantom{-}0.002$  & ---  & ---  & ---  & --- \tabularnewline
$b_{11}$  & $\phantom{-}1.000$  & $\phantom{-}1.001$  & $\phantom{-}0.995$  & $\phantom{-}0.994$  & $\phantom{-}0.992$  & $\phantom{-}0.990$  & ---  & ---  & ---  & --- \tabularnewline
$b_{12}$  & $\phantom{-}1.000$  & $\phantom{-}1.003$  & $\phantom{-}0.998$  & $\phantom{-}0.997$  & $\phantom{-}0.995$  & $\phantom{-}0.992$  & ---  & ---  & ---  & --- \tabularnewline
$b_{20}$  & $-0.001$  & $-0.003$  & $-0.006$  & $-0.004$  & $-0.004$  & $\phantom{-}0.004$  & ---  & ---  & ---  & --- \tabularnewline
$b_{21}$  & $\phantom{-}1.006$  & $\phantom{-}1.007$  & $\phantom{-}1.001$  & $\phantom{-}0.999$  & $\phantom{-}0.997$  & $\phantom{-}0.996$  & ---  & ---  & ---  & --- \tabularnewline
$b_{22}$  & $\phantom{-}1.005$  & $\phantom{-}1.007$  & $\phantom{-}1.004$  & $\phantom{-}1.000$  & $\phantom{-}0.998$  & $\phantom{-}0.996$  & ---  & ---  & ---  & --- \tabularnewline
$c_{1}$  & $\phantom{-}0.020$  & $\phantom{-}0.010$  & $\phantom{-}0.007$  & $\phantom{-}0.005$  & $\phantom{-}0.005$  & $\phantom{-}0.006$  & ---  & ---  & ---  & --- \tabularnewline
$c_{2}$  & $\phantom{-}1.004$  & $\phantom{-}1.003$  & $\phantom{-}1.006$  & $\phantom{-}1.001$  & $\phantom{-}1.001$  & $\phantom{-}1.002$  & ---  & ---  & ---  & --- \tabularnewline
\hline 
\multicolumn{11}{c}{}\tabularnewline
\multicolumn{11}{c}{Std.\ dev.}\tabularnewline
\hline 
$b_{10}$  & 0.0630  & 0.0628  & 0.0720  & 0.0666  & 0.0656  & 0.0665  & 1.15  & 1.06  & 1.04  & 1.06 \tabularnewline
$b_{11}$  & 0.0686  & 0.0686  & 0.0872  & 0.0764  & 0.0741  & 0.0743  & 1.27  & 1.11  & 1.08  & 1.08 \tabularnewline
$b_{12}$  & 0.0572  & 0.0574  & 0.0719  & 0.0667  & 0.0632  & 0.0646  & 1.25  & 1.16  & 1.10  & 1.13 \tabularnewline
$b_{20}$  & 0.0663  & 0.0657  & 0.0745  & 0.0686  & 0.0677  & 0.0676  & 1.13  & 1.04  & 1.04  & 1.03 \tabularnewline
$b_{21}$  & 0.1065  & 0.1050  & 0.1395  & 0.1128  & 0.1095  & 0.1099  & 1.33  & 1.07  & 1.04  & 1.05 \tabularnewline
$b_{22}$  & 0.1190  & 0.1174  & 0.1593  & 0.1285  & 0.1249  & 0.1244  & 1.36  & 1.09  & 1.06  & 1.06 \tabularnewline
$c_{1}$  & 0.1091  & 0.1107  & 0.1303  & 0.1276  & 0.1224  & 0.1265  & 1.18  & 1.15  & 1.11  & 1.14 \tabularnewline
$c_{2}$  & 0.1352  & 0.1325  & 0.1991  & 0.1509  & 0.1439  & 0.1421  & 1.50  & 1.14  & 1.09  & 1.07 \tabularnewline
\hline 
\multicolumn{7}{c|}{} &  &  &  & \tabularnewline
\hline 
Time  & 11.5  & 23.1  & 7.1  & 10.4  & 16.4  & 34.1  & ---  & ---  & ---  & --- \tabularnewline
\hline 
\end{tabular}
\par\end{center}

\pagebreak{}

\begin{center}
\begin{tabular}{l|cc|cccc|cccc|}
\multicolumn{11}{c}{\textbf{Table 5}}\tabularnewline
\multicolumn{11}{c}{Monte Carlo Results for Model 4}\tabularnewline
\multicolumn{11}{c}{($b_{10}=0$, $b_{11}=1$, $b_{12}=1$, $b_{20}=0$, $b_{21}=1$,
$b_{22}=1$, $c_{1}=1.33$, $c_{2}=1$)}\tabularnewline
\multicolumn{11}{c}{}\tabularnewline
 & \multicolumn{2}{c|}{SML} & \multicolumn{4}{c|}{GII} & \multicolumn{4}{c|}{$\sigma_{\mathrm{GII}}/\sigma_{\mathrm{SML}}$}\tabularnewline
 & \#1  & \#2  & \#1  & \#2  & \#3  & \#4  & \#1  & \#2  & \#3  & \#4 \tabularnewline
\hline 
\hline 
\multicolumn{11}{c}{}\tabularnewline
\multicolumn{11}{c}{Mean}\tabularnewline
\hline 
$b_{10}$  & $-0.031$  & $-0.017$  & $\phantom{-}0.000$  & $-0.001$  & $-0.000$  & $-0.001$  & ---  & ---  & ---  & --- \tabularnewline
$b_{11}$  & $\phantom{-}0.998$  & $\phantom{-}1.000$  & $\phantom{-}0.993$  & $\phantom{-}0.993$  & $\phantom{-}0.991$  & $\phantom{-}0.989$  & ---  & ---  & ---  & --- \tabularnewline
$b_{12}$  & $\phantom{-}1.016$  & $\phantom{-}1.011$  & $\phantom{-}0.998$  & $\phantom{-}0.998$  & $\phantom{-}0.996$  & $\phantom{-}0.994$  & ---  & ---  & ---  & --- \tabularnewline
$b_{20}$  & $-0.011$  & $-0.010$  & $-0.011$  & $-0.007$  & $-0.007$  & $-0.006$  & ---  & ---  & ---  & --- \tabularnewline
$b_{21}$  & $\phantom{-}0.992$  & $\phantom{-}0.999$  & $\phantom{-}1.000$  & $\phantom{-}0.997$  & $\phantom{-}0.995$  & $\phantom{-}0.991$  & ---  & ---  & ---  & --- \tabularnewline
$b_{22}$  & $\phantom{-}1.004$  & $\phantom{-}1.008$  & $\phantom{-}1.006$  & $\phantom{-}1.001$  & $\phantom{-}0.999$  & $\phantom{-}0.995$  & ---  & ---  & ---  & --- \tabularnewline
$c_{1}$  & $\phantom{-}1.269$  & $\phantom{-}1.306$  & $\phantom{-}1.347$  & $\phantom{-}1.338$  & $\phantom{-}1.335$  & $\phantom{-}1.330$  & ---  & ---  & ---  & --- \tabularnewline
$c_{2}$  & $\phantom{-}1.025$  & $\phantom{-}1.011$  & $\phantom{-}0.993$  & $\phantom{-}0.993$  & $\phantom{-}0.995$  & $\phantom{-}0.997$  & ---  & ---  & ---  & --- \tabularnewline
\hline 
\multicolumn{11}{c}{}\tabularnewline
\multicolumn{11}{c}{Std.\ dev.}\tabularnewline
\hline 
$b_{10}$  & 0.0693  & 0.0698  & 0.0789  & 0.0776  & 0.0758  & 0.0757  & 1.13  & 1.11  & 1.09  & 1.08 \tabularnewline
$b_{11}$  & 0.0587  & 0.0588  & 0.0696  & 0.0658  & 0.0632  & 0.0636  & 1.18  & 1.12  & 1.07  & 1.08 \tabularnewline
$b_{12}$  & 0.0745  & 0.0737  & 0.0883  & 0.0801  & 0.0781  & 0.0782  & 1.20  & 1.09  & 1.06  & 1.06 \tabularnewline
$b_{20}$  & 0.0766  & 0.0764  & 0.0900  & 0.0801  & 0.0786  & 0.0780  & 1.18  & 1.05  & 1.03  & 1.02 \tabularnewline
$b_{21}$  & 0.0884  & 0.0886  & 0.1140  & 0.0969  & 0.0952  & 0.0943  & 1.29  & 1.09  & 1.07  & 1.06 \tabularnewline
$b_{22}$  & 0.1106  & 0.1103  & 0.1471  & 0.1204  & 0.1176  & 0.1153  & 1.34  & 1.09  & 1.07  & 1.05 \tabularnewline
$c_{1}$  & 0.1641  & 0.1707  & 0.2454  & 0.2152  & 0.2049  & 0.2041  & 1.44  & 1.26  & 1.20  & 1.20 \tabularnewline
$c_{2}$  & 0.1229  & 0.1206  & 0.1599  & 0.1387  & 0.1338  & 0.1311  & 1.33  & 1.15  & 1.11  & 1.09 \tabularnewline
\hline 
\multicolumn{7}{c|}{} &  &  &  & \tabularnewline
\hline 
Time  & 12.7  & 25.6  & 7.4  & 10.8  & 17.1  & 34.4  & ---  & ---  & ---  & --- \tabularnewline
\hline 
\end{tabular}
\par\end{center}

\cleartooddpage{}

\setcounter{page}{1}

\renewcommand\thepage{A\arabic{page}} 

\appendix

\section{Details of optimization routines\label{app:routines}}

Both line-search methods (Gauss-Newton and quasi-Newton) involve the
use of a positive definite Hessian $\Delta_{(s)}$ in the approximating
model \eqref{fapprox}, and so the problem solved at step $s+1$ reduces
to that of ``approximately'' solving 
\begin{equation}
\min_{\alpha\in\reals}Q(\beta^{(s)}+\alpha p_{(s)}),\label{eq:Qmin}
\end{equation}
where $p_{(s)}\defeq-\Delta_{(s)}^{-1}\nabla_{(s)}$. We do not require
that $\alpha_{(s)}$ solve \eqref{Qmin} exactly; we shall require
only that it satisfy the strong Wolfe conditions,
\begin{align*}
Q(\beta^{(s)}+\alpha_{(s)}p_{(s)}) & \leq Q(\beta^{(s)})+c_{1}\alpha_{(s)}\nabla_{(s)}^{\trans}p_{(s)}\\
\smlabs{\dot{Q}(\beta^{(s)}+\alpha_{(s)}p_{(s)})^{\trans}p_{(s)}} & \leq c_{2}\smlabs{\nabla_{(s)}^{\trans}p_{(s)}}
\end{align*}
for $0<c_{1}<c_{2}<1$, where $\dot{Q}\defeq\partial_{\beta}Q$ (cf.\ (3.7)
in \citealp{NW06}). For some such $\alpha_{(s)}$, we set $\beta^{(s+1)}=\beta^{(s)}+\alpha_{(s)}p_{(s)}$.
For the Hessians $\Delta_{(s)}$, the Gauss-Newton method is only
applicable to criteria of the form $Q(\beta)=\smlnorm{g(\beta)}_{W}^{2}$,
and uses

\[
\nabla^{(s)}\defeq-(G_{(s)}^{\trans}WG_{(s)})^{\trans}G_{(s)}^{\trans}Wg(\beta^{(s)}),
\]
where $G_{(s)}\defeq[\partial_{\beta}g(\beta^{(s)})]^{\trans}$. The
Quasi-Newton method with BFGS updating starts with some initial positive
definite $\Delta_{(0)}$, and updates it according to, 
\[
\Delta_{(s+1)}=\Delta_{(s)}-\frac{\Delta_{(s)}x_{(s)}x_{(s)}^{\trans}\Delta_{(s)}}{x_{(s)}^{\trans}\Delta_{(s)}x_{(s)}}+\frac{d_{(s)}d_{(s)}^{\trans}}{d_{(s)}^{\trans}x_{(s)}},
\]
where $x_{(s)}\defeq\alpha_{(s)}p_{(s)}$ and $d_{(s)}=\nabla^{(s+1)}-\nabla^{(s)}$
(cf.\ (6.19) in \citealp{NW06}).

The trust region method considered here sets $\Delta_{(s)}=\partial_{\beta}^{2}Q(\beta^{(s)})$,
which need not be positive definite. The procedure then attempts to
approximately minimize \eqref{fapprox}, subject to the constraint
that $\smlnorm{\beta}\leq\delta_{(s)}$, where $\delta_{(s)}$ defines
the size of the trust region, which is adjusted at each iteration
depending on the value of
\[
\rho_{(s)}\defeq\frac{Q(\beta^{(s)})-Q(\beta^{(s+1})}{f_{(s)}(0)-f_{(s)}(\beta^{(s+1)})},
\]
which measures the proximity of the true reduction in $Q$ at step
$s$, with that predicted by the approximating model \eqref{fapprox};
the adjustment is made in accordance with Algorithm~4.2 in \citet{MS83SIAM}.
Various algorithms are available for approximately solving \eqref{fapprox}
in this case, but we shall assume that Algorithm~3.14 from that paper
is used.

\pagebreak[0]

\section{Proofs of theorems under high-level assumptions\label{app:mainproofs}}

Assumptions~\ref{ass:regularity} and \ref{ass:highlevel} are assumed
to hold throughout this section, including \enuref{H:deriv} with
$l=0$. Whenever we require \enuref{H:deriv} to hold for some $l\in\{1,2\}$,
this will be explicitly noted. The relationships between the theorems
and the auxiliary results (Propositions~\ref{prop:jackbind}--\ref{prop:detcvg})
is illustrated in \figref{theorems}.

\begin{figure}
\noindent \begin{centering}
\includegraphics[bb=50bp 480bp 546bp 670bp,clip,scale=0.8]{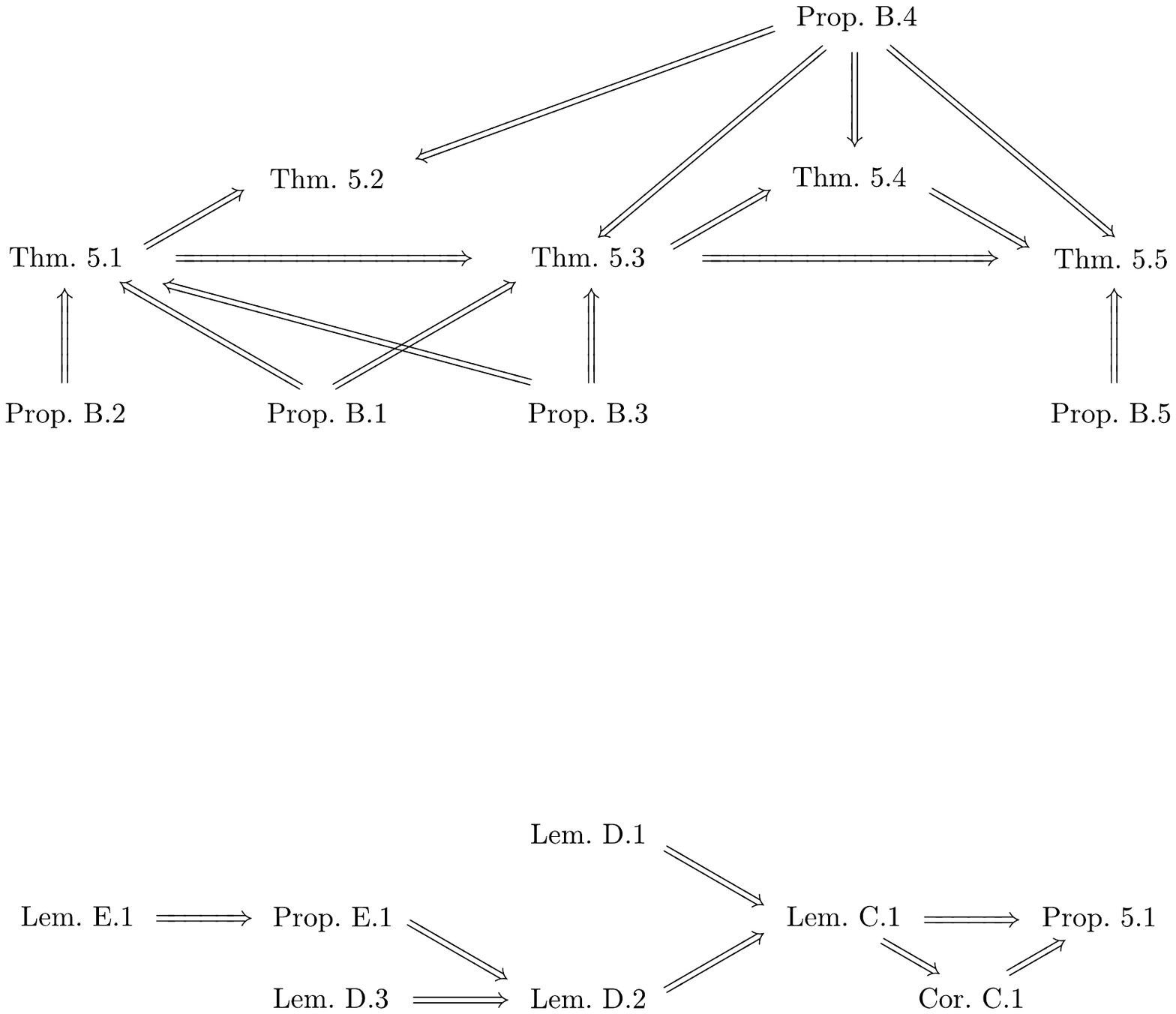}
\par\end{centering}

\protect\caption{Proofs of theorems}

\label{fig:theorems}
\end{figure}

\subsection{Preliminary results \label{app:prelimlem}}

Let $\beta_{n}\defeq\beta_{0}+n^{-1/2}\delta_{n}$ for a (possibly)
random $\delta_{n}=o_{p}(n^{1/2})$. Define 
\[
\Delta_{n}^{k}(\beta)\defeq n^{1/2}[\abv{\theta}_{n}^{k}(\beta,\lambda_{n})-\abv{\theta}_{n}^{k}(\beta_{0},\lambda_{n})]
\]
and recall that $G_{n}(\beta)\defeq\partial_{\beta}\abv{\theta}_{n}^{k}(\beta,\lambda_{n})$
and $G\defeq[\partial_{\beta}\bind(\beta_{0},0)]^{\trans}$. As in
\enuref{R:smoothing}, $\lambda_{n}=o_{p}(1)$ is an $\filt$-measurable
sequence. As per \enuref{R:jackknifing}, we fix the order of jackknifing
$k\in\{0,\ldots,k_{0}\}$ such that $n^{1/2}\lambda_{n}^{k+1}=o_{p}(1)$.
Let $\Like_{n}(\theta)\defeq\L_{n}(y,x;\theta)$ and $\L(\theta)\defeq\expect\L_{n}(\theta)$.
$\dot{\Like}_{n}$ and $\ddot{\Like}_{n}$ respectively denote the
gradient and Hessian of $\Like_{n}$, with $H\defeq\expect\ddot{\Like}_{n}(\theta)=\Like(\theta)$;
$N(\theta,\epsilon)$ denotes an open ball of radius $\epsilon$,
centered at $\theta$.
\begin{prop}
\label{prop:jackbind}~
\begin{enumerate}
\item \label{enu:smplbind}$\sup_{\beta\in\Beta}\smlnorm{\abv{\theta}_{n}^{k}(\beta,\lambda_{n})-\bind^{k}(\beta,\lambda_{n})}\inprob0$;
\item \label{enu:nobias}$\bind^{k}(\beta_{0},\lambda_{n})-\bind(\beta_{0},0)=O_{p}(\lambda_{n}^{k+1})$;
\item \label{enu:stocdiff} $\Delta_{n}^{k}(\beta_{n})=G\delta_{n}+o_{p}(1+\smlnorm{\delta_{n}})$;
\end{enumerate}
\end{prop}
 
\begin{prop}
\label{prop:auxdist} For $V=(1+\tfrac{1}{M})(\varmat-\covmat)$,
\begin{equation}
Z_{n}\defeq n^{1/2}[\abv{\theta}_{n}^{k}(\beta_{0},\lambda_{n})-\bind^{k}(\beta_{0},\lambda_{n})]-n^{1/2}(\hat{\theta}_{n}-\theta_{0})\wkc N[0,H^{-1}VH^{-1}].\label{eq:Zn}
\end{equation}

\end{prop}
 
\begin{prop}
\label{prop:unifcvg}~
\begin{enumerate}
\item \label{enu:Qunif}$Q_{nk}^{e}(\beta,\lambda_{n})\inprob Q_{k}^{e}(\beta,0)\eqdef Q^{e}(\beta)$
uniformly on $\Beta$;
\item \label{enu:Qwellsep}for every $\epsilon>0$, $\inf_{\beta\in\Beta\backslash N(\beta_{0},\epsilon)}Q^{e}(\beta)>Q(\beta_{0})$;
and
\end{enumerate}
\end{prop}
 
\begin{prop}
\label{prop:derivcvg}~ If \enuref{H:deriv} holds for $l=1$, then
\begin{enumerate}
\item \label{enu:Jacobian} $G_{n}(\beta_{n})\inprob G$; and
\end{enumerate}
if \enuref{H:deriv} holds for $l\in\{1,2\}$ then, uniformly on $\Beta$,
\begin{enumerate}[resume]
\item \label{enu:bindderiv}$\sup_{\beta\in\Beta}\smlnorm{\partial_{\beta}^{l}\abv{\theta}_{n}^{k}(\beta,\lambda_{n})-\partial_{\beta}^{l}\theta(\beta,0)}=o_{p}(1)$;
and
\item \label{enu:Qderiv} $\partial_{\beta}^{l}Q_{nk}^{e}(\beta,\lambda_{n})\inprob\partial_{\beta}^{l}Q_{k}^{e}(\beta,0)=\partial_{\beta}^{l}Q(\beta)$.
\end{enumerate}
\end{prop}

For the next result, let $U:\Gamma\setmap\reals$ be twice continuously
differentiable with a global minimum at $\gamma^{\ast}$. Let $R_{U}\defeq\{\gamma\in\Gamma\mid\smlnorm{\partial_{\gamma}U(\gamma)}<\epsilon\}$
for some $\epsilon>0$, and $S_{U}\defeq\{\gamma\in R_{U}\mid\varrho_{\min}[\partial_{\gamma}^{2}U(\gamma)]\geq0\}$.
Applying a routine $r\in\{\gn,\qn,\tn\}$ to $U$ yields the iterates
$\{\gamma^{(s)}\}$; let
\[
\abv{\gamma}(\gamma^{(0)},r)\defeq\begin{cases}
\gamma^{(s^{\ast})} & \text{if }\gamma^{(s)}\in R_{U}\text{ for some }s\in\naturals\\
\gamma^{(0)} & \text{otherwise},
\end{cases}
\]
where $s^{\ast}$ denotes the smallest $s$ for which $\gamma^{(s)}\in R_{U}$.
When $r=\tn$, the definition of $\abv{\gamma}(\gamma^{(0)},\tn)$
is analogous, but with $S_{U}$ in place of $R_{U}$. In the statement
of the next result, $\Gamma_{0}\defeq\{\gamma\in\Gamma\mid U(\gamma)\leq U(\gamma_{1})\}$
for some $\gamma_{1}\in\Gamma$, and is a compact set with $\gamma^{\ast}\in\intr\Gamma_{0}$.
For a function $m:\Gamma\elmap\reals^{d_{m}}$, let $M(\gamma)\defeq[\partial_{\gamma}m(\gamma)]^{\trans}$
denote its Jacobian. 
\begin{prop}
\label{prop:detcvg} Let $r\in\{\qn,\tn\}$, and suppose that in addition
to the preceding, either
\begin{enumerate}
\item $r=\gn$ and $U(\gamma)=\smlnorm{m(\gamma)}^{2}$, with $\inf_{\gamma\in\Gamma_{0}}\sigma_{\min}[M(\gamma)]>0$;
or
\item $r=\qn$ and $U$ is strictly convex on $\Gamma_{0}$;
\end{enumerate}
then $\abv{\gamma}(\gamma^{(0)},r)\in R_{U}\intsect\Gamma_{0}$ for
all $\gamma^{(0)}\in\Gamma_{0}$. Alternatively, if $r=\tn$, then
$\abv{\gamma}(\gamma^{(0)},r)\in S_{U}\intsect\Gamma_{0}$ for all
$\gamma^{(0)}\in\Gamma_{0}$.
\end{prop}

\subsection{Proofs of Theorems~\ref{thm:limitdist}--\ref{thm:optim}}

Throughout this section, $\beta_{n}\defeq\beta_{0}+n^{-1/2}\delta_{n}$
for a (possibly) random $\delta_{n}=o_{p}(n^{1/2})$. Let $Q_{n}^{\Wald}(\beta)\defeq Q_{nk}^{\Wald}(\beta,\lambda_{n})$,
$Q_{n}^{\LR}(\beta)\defeq Q_{nk}^{\LR}(\beta,\lambda_{n})$, and $\abv{\theta}_{n}(\beta)\defeq\abv{\theta}_{n}^{k}(\beta,\lambda_{n})$.
\begin{proof}[Proof of \thmref{limitdist}]
 We first consider the Wald estimator. We have
\[
n[Q_{n}^{\Wald}(\beta_{n})-Q_{n}^{\Wald}(\beta_{0})]=2n^{1/2}[\abv{\theta}_{n}^{k}(\beta_{0})-\hat{\theta}_{n}]^{\trans}W_{n}\Delta_{n}^{k}(\beta_{n})+\Delta_{n}^{k}(\beta_{n})^{\trans}W_{n}\Delta_{n}^{k}(\beta_{n}).
\]
For $Z_{n}$ as defined in \eqref{Zn}, we see that by \propref{jackbind}\enuref{nobias}
and \enuref{R:jackknifing}
\begin{equation}
n^{1/2}[\abv{\theta}_{n}^{k}(\beta_{0})-\hat{\theta}]=Z_{n}+n^{1/2}[\bind^{k}(\beta_{0},\lambda_{n})-\theta_{0}]=Z_{n}+o_{p}(1),\label{eq:binddiffZ}
\end{equation}
whence by \propref{jackbind}\enuref{stocdiff}, 
\begin{equation}
n[Q_{n}^{\Wald}(\beta_{n})-Q_{n}^{\Wald}(\beta_{0})]=2Z_{n}^{\trans}WG\delta_{n}+\delta_{n}^{\trans}G^{\trans}WG\delta_{n}+o_{p}(1+\smlnorm{\delta_{n}}+\smlnorm{\delta_{n}}^{2}).\label{eq:Waldexp}
\end{equation}

Now consider the LR estimator. Twice continuous differentiability
of the likelihood yields 
\begin{align*}
n[Q_{n}^{\LR}(\beta)-Q_{n}^{\LR}(\beta_{0})] & =-n[\Like_{n}(\abv{\theta}_{n}^{k}(\beta_{n}))-\Like_{n}(\abv{\theta}_{n}^{k}(\beta_{0}))]\\
 & =-n^{1/2}\dot{\Like}_{n}(\abv{\theta}_{n}^{k}(\beta_{0}))^{\trans}\Delta_{n}^{k}(\beta_{n})-\frac{1}{2}\Delta_{n}^{k}(\beta_{n})^{\trans}\ddot{\L}_{n}(\abv{\theta}_{n}^{k}(\beta_{0}))\Delta_{n}^{k}(\beta_{n})\\
 & \qquad\qquad\qquad+o_{p}(\smlnorm{\Delta_{n}^{k}(\beta_{n})}^{2})
\end{align*}
where by \propref{jackbind}\enuref{nobias} and \enuref{H:limitdist},
\begin{align}
n^{1/2}\dot{\L}_{n}[\abv{\theta}_{n}^{k}(\beta_{0})] & =n^{1/2}\dot{\L}_{n}(\theta_{0})+\ddot{\L}_{n}(\theta_{0})n^{1/2}[\abv{\theta}_{n}^{k}(\beta_{0})-\theta_{0}]+o_{p}(1)\nonumber \\
 & =H[Z_{n}+n^{1/2}(\bind^{k}(\beta_{0},\lambda_{n})-\theta_{0})]\nonumber \\
 & =HZ_{n}+o_{p}(1)\label{eq:LHlimit}
\end{align}
for $Z_{n}$ as in \eqref{Zn}. Thus by \propref{jackbind}\enuref{stocdiff},
\begin{align}
n[Q_{n}^{\LR}(\beta_{n})-Q_{n}^{\LR}(\beta_{n})] & =-Z_{n}^{\trans}HG\delta_{n}-\frac{1}{2}\delta_{n}^{\trans}G^{\trans}HG\delta_{n}+o_{p}(1+\smlnorm{\delta_{n}}+\smlnorm{\delta_{n}}^{2}).\label{eq:LRexp}
\end{align}

Consistency of $\hat{\beta}_{nk}^{e}$ follows from parts \enuref{Qunif}
and \enuref{Qwellsep} of \propref{unifcvg} and Corollary~3.2.3
in \citet{VVW96}. Thus by applying Theorem~3.2.16 in \citet{VVW96}
-- or more precisely, the arguments following their (3.2.17) -- to
\eqref{Waldexp} and \eqref{LRexp}, we have
\begin{equation}
n^{1/2}(\hat{\beta}_{nk}^{e}-\beta_{0})=-(G^{\trans}U_{e}G)^{-1}G^{\trans}U_{e}Z_{n}+o_{p}(1)\label{eq:hatbetalimit}
\end{equation}
for $U_{e}$ as in \eqref{UVe}; the result now follows by \propref{auxdist}.
\end{proof}
 
\begin{proof}[Proof of \thmref{feasible}]
 We first note that, in consequence of \enuref{H:stocheq} and \thmref{limitdist},
$\hat{\beta}_{nk}^{e}\inprob\beta_{0}$, $\hat{\theta}_{n}\inprob\theta_{0}$,
and $\hat{\theta}_{n}^{m}\defeq\hat{\theta}_{n}^{m}(\beta_{nk}^{e},\lambda_{n})\inprob\theta_{0}$.
Part~\enuref{feas:hess} then follows from \enuref{R:scorevar},
\enuref{H:unifcmpct}, and Lemma~2.4 in \citet{New94}. Defining
$\dot{\like}_{i}^{m}(\theta_{0})\defeq\dot{\like}_{i}^{m}(\beta_{0},0;\theta_{0})$
for $m\in\{1,\ldots,M\}$ and
\[
\varsigma_{i}^{\trans}\defeq\begin{bmatrix}\dot{\like}_{i}^{0}(\theta_{0})^{\trans} & \dot{\like}_{i}^{1}(\beta_{0},0;\theta_{0})^{\trans} & \cdots & \dot{\like}_{i}^{M}(\beta_{0},0;\theta_{0})^{\trans}\end{bmatrix},
\]
\enuref{H:unifcmpct} and \enuref{H:limitdist} further imply that
\[
A^{\trans}\left(\frac{1}{n}\sum_{i=1}^{n}s_{ni}s_{ni}^{\trans}\right)A\inprob A^{\trans}(\expect\varsigma_{i}\varsigma_{i}^{\trans})A=A^{\trans}\begin{bmatrix}\varmat & \covmat & \cdots & \covmat\\
\covmat & \varmat & \cdots & \covmat\\
\vdots & \vdots & \ddots & \vdots\\
\covmat & \covmat & \cdots & \varmat
\end{bmatrix}A=V.
\]
Part\ \enuref{feas:G} is an immediate consequence of \propref{derivcvg}\enuref{Jacobian}.
\end{proof}
 
\begin{proof}[Proof of \thmref{rootdist}]
 We first prove part \enuref{setcons}. Let $\dot{Q}_{n}^{e}(\beta)\defeq\partial_{\beta}Q_{n}^{e}(\beta)$
and $\dot{Q}^{e}(\beta)\defeq\partial_{\beta}Q^{e}(\beta,0)$. Since
$\beta_{0}\in\intr\Beta$ and $Q_{n}^{e}(\beta)\inprob Q(\beta)$
uniformly on $\Beta$, the global minimum of $Q_{n}^{e}$ is interior
to $\Beta$, w.p.a.1., whence $R_{nk}^{e}$ is non-empty w.p.a.1.
Letting $\{\tilde{\beta}_{n}\}$ denote a (random) sequence with $\tilde{\beta}_{n}\in R_{nk}^{e}$
for all $n$ sufficiently large, we have by \propref{derivcvg}\enuref{Qderiv}
that 
\begin{equation}
\dot{Q}^{e}(\tilde{\beta}_{n})=\dot{Q}_{n}^{e}(\tilde{\beta}_{n})+o_{p}(1)=o_{p}(1+c_{n})=o_{p}(1).\label{eq:Qdotcvg}
\end{equation}
Since $\dot{Q}^{e}$ is continuous and $\Beta$ compact, it follows
that $d(\tilde{\beta}_{n},R^{e})\inprob0$, whence $d_{L}(R_{nk}^{e},R^{e})\inprob0$.

We now turn to part \enuref{consistentroot}. Recall $\beta_{n}=\beta_{0}+n^{1/2}\delta_{n}$
for some $\delta_{n}=o_{p}(n^{1/2})$. For the Wald criterion, taking
$\delta_{n}$ such that $\beta_{n}\in R_{nk}^{\Wald}$ gives
\[
o_{p}(1)=n^{1/2}\dot{Q}_{n}^{\Wald}(\beta_{n})^{\trans}=2[n^{1/2}(\abv{\theta}_{n}^{k}(\beta_{n})-\hat{\theta}_{n})]^{\trans}WG_{n}(\beta_{n})
\]
where, for $Z_{n}$ as in \eqref{Zn},
\[
n^{1/2}(\abv{\theta}_{n}^{k}(\beta_{n})-\hat{\theta}_{n})=n^{1/2}(\abv{\theta}_{n}^{k}(\beta_{0})-\hat{\theta}_{n})+\Delta_{n}^{k}(\beta_{n})=Z_{n}+G\delta_{n}+o_{p}(1+\smlnorm{\delta_{n}})
\]
by \eqref{binddiffZ}, \enuref{R:jackknifing}, and parts \enuref{nobias}
and \enuref{stocdiff} of \propref{jackbind}. Hence, using \propref{derivcvg}\enuref{Jacobian},
\begin{equation}
o_{p}(1)=2[\delta_{n}^{\trans}G^{\trans}WG+Z_{n}^{\trans}WG]+o_{p}(1+\smlnorm{\delta_{n}}).\label{eq:scoreexpW}
\end{equation}

Similarly, for the LR criterion, taking $\beta_{n}\in R_{nk}^{\LR}$
in this case gives
\[
o_{p}(1)=n^{1/2}\partial_{\beta}Q_{n}^{\LR}(\beta_{n})^{\trans}=n^{1/2}\dot{\L}_{n}[\abv{\theta}_{n}^{k}(\beta_{n})]^{\trans}G_{n}(\beta_{n})
\]
where by the twice continuous differentiability of the likelihood,
\propref{jackbind}\enuref{stocdiff} and \eqref{LHlimit}, 
\begin{align*}
n^{1/2}\dot{\L}_{n}[\abv{\theta}_{n}^{k}(\beta_{n})] & =n^{1/2}\dot{\L}_{n}[\abv{\theta}_{n}^{k}(\beta_{0})]+\ddot{\L}_{n}(\abv{\theta}_{n}^{k}(\beta_{0}))\Delta_{n}^{k}(\beta_{n})+o_{p}(\smlnorm{\Delta_{n}^{k}(\beta_{n})})\\
 & =HZ_{n}+HG\delta_{n}+o_{p}(1+\smlnorm{\delta_{n}}).
\end{align*}
Thus by \propref{derivcvg}\enuref{Jacobian},
\begin{equation}
o_{p}(1)=\delta_{n}^{\trans}G^{\trans}HG+Z_{n}^{\trans}HG+o_{p}(1+\smlnorm{\delta_{n}}).\label{eq:scoreexpLR}
\end{equation}

By specializing \eqref{scoreexpW} and \eqref{scoreexpLR} to the
case where $\delta_{n}=n^{1/2}(\tilde{\beta}_{nk}^{e}-\beta_{0})$,
for $\tilde{\beta}_{nk}^{e}$ satisfying the requirements of part
\enuref{consistentroot} of the theorem, we see that for $U_{e}$
as in \eqref{UVe}, 
\[
n^{1/2}(\tilde{\beta}_{nk}^{e}-\beta_{0})=-(G^{\trans}U_{e}G)^{-1}G^{\trans}U_{e}Z_{n}+o_{p}(1)=n^{1/2}(\hat{\beta}_{nk}^{e}-\beta_{0})+o_{p}(1)
\]
for $e\in\{\Wald,\LR\}$, in consequence of \eqref{hatbetalimit}.

Finally, we turn to part \enuref{consistexists}. Let $\hat{\beta}_{n}$
denote the minimizer of $Q_{n}^{e}(\beta)$, which lies in $R_{nk}$
w.p.a.1., by part \enuref{setcons}, and $\tilde{\beta}_{n}$ another
(random) sequence satisfying the requirements of part \enuref{consistexists}.
By \propref{unifcvg} and the consistency of $\hat{\beta}_{n}$ (\thmref{limitdist}),
\[
Q^{e}(\beta_{0})+o_{p}(1)=Q_{n}^{e}(\hat{\beta}_{n})+o_{p}(1)\geq Q_{n}^{e}(\tilde{\beta}_{n})\geq Q_{n}^{e}(\hat{\beta}_{n})=Q^{e}(\beta_{0})+o_{p}(1).
\]
Thus $Q^{e}(\tilde{\beta}_{n})=Q_{n}^{e}(\tilde{\beta}_{n})+o_{p}(1)\inprob Q^{e}(\beta_{0})$,
also by \propref{unifcvg}; whence $\tilde{\beta}_{n}\inprob\beta_{0}$,
since $Q^{e}$ has a well-separated minimum at $\beta_{0}$.
\end{proof}
 
\begin{proof}[Proof of \thmref{root2nd}]
 Let $\ddot{Q}_{n}^{e}(\beta)\defeq\partial_{\beta}^{2}Q_{n}^{e}(\beta)$,
$\ddot{Q}^{e}(\beta)\defeq\partial_{\beta}^{2}Q^{e}(\beta,0)$, and
$\{\tilde{\beta}_{n}\}$ be a (random) sequence with $\tilde{\beta}_{n}\in S_{nk}^{e}$
for all $n$ sufficiently large. Then by \propref{derivcvg}\enuref{Qderiv},
\[
\Prob\{\varrho_{\min}[\ddot{Q}^{e}(\tilde{\beta}_{n})]<-\epsilon\}=\Prob\{\varrho_{\min}[\ddot{Q}_{n}^{e}(\tilde{\beta}_{n})]+o_{p}(1)<-\epsilon\}\leq\Prob\{o_{p}(1)<-\epsilon\}\goesto0
\]
for any $\epsilon>0$. Hence by \thmref{rootdist}\enuref{setcons},
the continuity of $\dot{Q}^{e}$ and $\ddot{Q}^{e}$, $d(\tilde{\beta}_{n},S^{e})\inprob0$.
Part \enuref{consistentroot} follows immediately from the corresponding
part of \thmref{rootdist} and the fact that $S_{nk}^{e}\subseteq R_{nk}^{e}$.
\end{proof}
 
\begin{proof}[Proof of \thmref{optim}]
 For each $r\in\{\gn,\qn,\tn\}$, suppose that there exists a $\Beta_{0}\subseteq\Beta$
such that $U=Q_{n}^{e}(\beta)\defeq Q_{nk}^{e}(\beta,\lambda_{n})$
satisfies the corresponding part of \propref{detcvg}, w.p.a.1. Then
\[
\Prob\{\abv{\beta}_{nk}^{e}(\beta^{(0)},r)\in R_{nk}^{e}\intsect\Beta_{0}\sep\forall\beta^{(0)}\in\Beta_{0}\}\inprob0
\]
for $r\in\{\gn,\qn\}$, and also for $r=\tn$ with $S_{nk}^{e}$ in
place of $R_{nk}^{e}$; we may take $c_{n}=o_{p}(n^{-1/2})$ in the
definition $R_{nk}^{e}$. Further, $R^{e}\intsect\Beta_{0}=\{\beta_{0}\}$
under \GN{} and \QN{}, while $S^{e}\intsect\Beta_{0}=\{\beta_{0}\}$
under \TR{}. Thus, when $r\in\{\gn,\qn\}$ we have w.p.a.1, 
\[
\sup_{\beta^{(0)}\in\Beta_{0}}d(\abv{\beta}_{nk}^{e}(\beta^{(0)},r),\beta_{0})\leq d_{L}(R_{nk}^{e}\intsect\Beta_{0},\{\beta_{0}\})=o_{p}(n^{-1/2})
\]
with the final estimate following by \thmref{rootdist}. When $r=\tn$,
the preceding holds with $S_{nk}^{e}$ in place of $R_{nk}^{e}$,
in this case via \thmref{root2nd}.

It thus remains to verify that the requirements of \propref{detcvg}
hold w.p.a.1. When $r=\gn$, it follows from \propref{derivcvg}\enuref{Jacobian},
the continuity of $\sigma_{\min}(\cdot)$ and \GN{} that 
\[
0<\inf_{\beta\in\Beta_{0}}\sigma_{\min}[G(\beta)]=\inf_{\beta\in\Beta_{0}}\sigma_{\min}[G_{n}(\beta)]+o_{p}(1),
\]
whence $\inf_{\beta\in\Beta_{0}}\sigma_{\min}[G_{n}(\beta)]>0$ w.p.a.1.
When $r=\qn$, it follows from \propref{derivcvg}\enuref{Qderiv}
and \QN{} that 
\[
0<\inf_{\beta\in\Beta_{0}}\varrho_{\min}[\partial_{\beta}^{2}Q^{e}(\beta)]=\inf_{\beta\in\Beta_{0}}\varrho_{\min}[\partial_{\beta}^{2}Q_{n}^{e}(\beta)]+o_{p}(1)
\]
whence $Q_{n}^{e}$ is strictly convex on $\Beta_{0}$ w.p.a.1. When
$r=\tn$, there are no additional conditions to verify.
\end{proof}

\subsection{Proofs of Propositions~\ref{prop:jackbind}--\ref{prop:detcvg}}
\begin{proof}[Proof of \propref{jackbind}]
 Part~\enuref{smplbind} follows by \enuref{H:stocheq} and the
continuous mapping theorem. Part~\enuref{nobias} is immediate from
\eqref{kextrap}. For part~\enuref{stocdiff}, we note that for $\beta_{n}=\beta_{0}+n^{1/2}\delta_{n}$
with $\delta_{n}=o_{p}(n^{1/2})$ as above,
\begin{multline*}
\Delta_{n}^{k}(\beta_{n})=n^{1/2}[\abv{\theta}_{n}^{k}(\beta_{n},\lambda_{n})-\bind^{k}(\beta_{n},\lambda_{n})]\\
-n^{1/2}[\abv{\theta}_{n}^{k}(\beta_{0},\lambda_{n})-\bind^{k}(\beta_{0},\lambda_{n})]+n^{1/2}[\bind^{k}(\beta_{n},\lambda_{n})-\bind^{k}(\beta_{0},\lambda_{n})].
\end{multline*}
Since $\abv{\theta}_{n}^{k}$ is a linear combination of the $\hat{\theta}_{n}^{m}$'s,
it is clear from \enuref{H:stocheq} that the first two terms converge
jointly in distribution to identical limits (since $\beta_{n}\inprob\beta_{0}$).
For the final term, continuous differentiability of $\bind^{k}$ (\enuref{R:uniqueness}
above) entails that 
\begin{align*}
n^{1/2}[\bind^{k}(\beta_{n},\lambda_{n})-\bind^{k}(\beta_{0},\lambda_{n})] & =[\partial_{\beta}\bind^{k}(\beta_{0},\lambda_{n})]^{\trans}(\beta_{n}-\beta_{0})+o_{p}(\smlnorm{\beta_{n}-\beta_{0}})\\
 & =G\delta_{n}+o_{p}(1+\smlnorm{\delta_{n}}).
\end{align*}

\end{proof}
 
\begin{proof}[Proof of \propref{auxdist}]
 Note first that 
\begin{multline*}
n^{1/2}[\abv{\theta}_{n}^{k}(\beta_{0},\lambda_{n})-\bind^{k}(\beta_{0},\lambda_{n})]=\sum_{r=0}^{k}\gamma_{rk}\cdot n^{1/2}[\abv{\theta}_{n}(\beta_{0},\delta^{r}\lambda_{n})-\bind(\beta_{0},\delta^{r}\lambda_{n})]\\
=\frac{1}{M}\sum_{m=1}^{M}\sum_{r=0}^{k}\gamma_{rk}\psi_{n}^{m}(\beta_{0},\delta^{r}\lambda_{n})\wkc\frac{1}{M}\sum_{m=1}^{M}\psi^{m}(\beta_{0},0),
\end{multline*}
by \eqref{kextrap}, \eqref{jackknifed}, \enuref{H:stocheq} and
$\sum_{r=0}^{k}\gamma_{rk}=1$. By \enuref{H:stocheq}, this holds
jointly with
\[
n^{1/2}(\hat{\theta}_{n}-\theta_{0})\wkc\psi^{0}(\beta_{0},0).
\]
Since \enuref{H:limitdist} implies that $\psi^{m}(\beta_{0},0)=H^{-1}\phi^{m}$,
the limiting variance of $Z_{n}$ is equal to
\[
\var\left[\psi^{0}(\beta_{0},0)-\frac{1}{M}\sum_{m=1}^{M}\psi^{m}(\beta_{0},0)\right]=H^{-1}\var\left[\phi^{0}-\frac{1}{M}\sum_{m=1}^{M}\phi^{m}\right]H^{-1}=H^{-1}VH^{-1}
\]
where the final equality follows from \enuref{H:limitdist} and straightforward
calculations.
\end{proof}
 
\begin{proof}[Proof of \propref{unifcvg}]
 We first prove part~\enuref{Qunif}. For the Wald estimator, this
is immediate from \propref{jackbind}\enuref{smplbind}. For the LR
estimator, it follows from \propref{jackbind}\enuref{smplbind},
\enuref{H:unifcmpct} and the continuous mapping theorem (arguing
as on pp.~144f.\ of \citealp{Billing68}), that 
\[
Q_{nk}^{\LR}(\beta)=(\L_{n}\compose\abv{\theta}_{n}^{k})(\beta,\lambda_{n})\inprob(\L\compose\bind^{k})(\beta,0)=Q^{\LR}(\beta),
\]
uniformly on $\Beta$.

For part~\enuref{Qwellsep}, we note that $\beta\elmap\bind^{k}(\beta,0)$
is continuous by \enuref{R:uniqueness}, while the continuity of $\L$
is implied by \enuref{H:unifcmpct}, since $\L_{n}$ is continuous.
Thus $Q^{e}$ is continuous for $e\in\{\Wald,\LR\}$, and by \enuref{R:injective}
is uniquely minimized at $\beta_{0}$. Hence $\beta\elmap Q^{e}(\beta)$
has a well-separated minimum, which by \enuref{R:correct} is interior
to $\Beta$.
\end{proof}
 
\begin{proof}[Proof of \propref{derivcvg}]
 Part~\enuref{bindderiv} is immediate from \enuref{H:deriv}, \eqref{jackknifed}
and the continuous mapping theorem; it further implies part~\enuref{Jacobian}.
For part~\enuref{Qderiv}, recall $\dot{Q}_{n}^{e}(\beta)=\partial_{\beta}Q_{n}^{e}(\beta)$,
and $G_{n}(\beta)=[\partial_{\beta}\abv{\theta}_{n}^{k}(\beta)]^{\trans}$.
Then we have
\begin{align*}
\dot{Q}_{n}^{\Wald}(\beta) & =G_{n}(\beta)^{\trans}W_{n}[\abv{\theta}_{n}(\beta)-\hat{\theta}_{n}] & \dot{Q}_{n}^{\LR}(\beta) & =G_{n}(\beta)^{\trans}\dot{\Like}_{n}[\abv{\theta}_{n}^{k}(\beta)].
\end{align*}
Part~\enuref{Jacobian}, and similar arguments as were used are used
in the proof of part~\enuref{Qunif} of \propref{unifcvg}, yield
that $\dot{Q}_{n}^{e}(\beta)\inprob\partial_{\beta}Q^{e}(\beta,0)\eqdef\dot{Q}^{e}(\beta)$
uniformly on $\Beta$. The proof that the second derivatives converge
uniformly is analogous.
\end{proof}
 
\begin{proof}[Proof of \propref{detcvg}]
 For $r=\gn$, the result follows by Theorem~10.1 in \citet{NW06};
for $r=\qn$, by their Theorem~6.5; and for $r=\tn$, by Theorem~4.13
in \citet{MS83SIAM}.
\end{proof}

\section{Sufficiency of the low-level assumptions\label{app:sufficiency}}

We shall henceforth maintain both Assumptions~\ref{ass:lowlevel}
and \ref{ass:regularity}, and address the question of whether these
are sufficient for \assref{highlevel}; that is, we shall prove \propref{sufficiency}.
The main steps leading to the proof are displayed in \figref{propproof}.

\begin{figure}
\noindent \begin{centering}
\includegraphics[bb=50bp 240bp 546bp 330bp,clip,scale=0.8]{figures}
\par\end{centering}

\protect\caption{Proof of \propref{sufficiency}}

\label{fig:propproof}
\end{figure}

Recall that, as per \enuref{L:auxiliary}, the auxiliary model is
the Gaussian SUR displayed in \eqref{reggeneral} above. For simplicity,
we shall consider only the case where $\Sigma_{\xi}$ is unrestricted,
but our arguments extend straightforwardly to the case where $\Sigma_{\xi}$
is block diagonal (as would typically be imposed when $T>1$). Recall
that $\theta$ collects the elements of $\alpha$ and $\Sigma_{\xi}^{-1}$.
Fix an $m\in\{0,1,\ldots M\}$, and define
\[
\xi_{ri}(\alpha)\defeq y_{r}(z_{i};\beta,\lambda)-\alpha_{xr}^{\trans}\Pi_{xr}x(z_{i})-\alpha_{yr}^{\trans}\Pi_{yr}y(z_{i};\beta,\lambda),
\]
temporarily suppressing the dependence of $y$ (and hence $\xi_{ri}$)
on $m$. Collecting $\xi_{i}\defeq(\xi_{1i},\ldots,\xi_{d_{y}i})^{\trans}$,
the average log-likelihood of the auxiliary model can be written as
\begin{align*}
\Like_{n}(y,x;\theta) & =\frac{1}{n}\sum_{i=1}^{n}\like(y_{i},x_{i};\theta)=-\frac{1}{2}\log2\pi-\frac{1}{2}\log\det\Sigma_{\xi}-\frac{1}{2}\tr\left[\Sigma_{\xi}^{-1}\frac{1}{n}\sum_{i=1}^{n}\xi_{i}(\alpha)\xi_{i}(\alpha)^{\trans}\right].
\end{align*}
Deduce that there are functions $L$ and $l$, which are three times
continuously differentiable in both arguments (at least on $\intr\Theta$),
such that 
\begin{align}
\L_{n}(y,x;\theta) & =L(T_{n};\theta) & \like(y_{i},x_{i};\theta) & =l(t_{i};\theta)\label{eq:suffstatlike}
\end{align}
where 
\[
t_{i}^{m}(\beta,\lambda)=\begin{bmatrix}y(z_{i}^{m};\beta,\lambda)\\
x(z_{i}^{m})
\end{bmatrix}\begin{bmatrix}y(z_{i}^{m};\beta,\lambda)^{\trans} & x(z_{i}^{m})^{\trans}\end{bmatrix}
\]
and $T_{n}^{m}\defeq\vekh(\mathcal{T}_{n}^{m})$, for 
\begin{equation}
\mathcal{T}_{n}^{m}(\beta,\lambda)\defeq\frac{1}{n}\sum_{i=1}^{n}t_{i}^{m}(\beta,\lambda)t_{i}^{m}(\beta,\lambda)^{\trans}.\label{eq:BigT}
\end{equation}

Further, direct calculation gives
\begin{align}
\partial_{\alpha_{xr}}\like_{i}(\theta) & =\sum_{s=1}^{d_{y}}\sigma^{rs}\xi_{si}(\alpha)\Pi_{xr}x(z_{i}) & \partial_{\alpha_{yr}}\like_{i}(\theta) & =\sum_{s=1}^{d_{y}}\sigma^{rs}\xi_{si}(\alpha)\Pi_{yr}y(z_{i};\beta,\lambda)\label{eq:score1}
\end{align}
and
\begin{equation}
\partial_{\sigma^{rs}}\like_{i}(\theta)=\frac{1}{2}\sigma_{rs}-\frac{1}{2}\xi_{ri}(\alpha)\xi_{si}(\alpha).\label{eq:score2}
\end{equation}
Since the elements of the score vector $\dot{\like}_{i}(\theta)=\partial_{\theta}\like_{i}(\theta)$
necessarily take one of the forms displayed in \eqref{score1} or
\eqref{score2}, we may conclude that, for any compact subset $A\subset\Theta$,
there exists a $C_{A}$ such that
\begin{equation}
\expect\sup_{\theta\in A}\smlnorm{\dot{\like}_{i}(\theta)}^{2}\leq C_{A}\expect\smlnorm{z_{i}}^{4}<\infty\label{eq:finitescore}
\end{equation}
with the second inequality following from \enuref{L:moment}.

Regarding the maximum likelihood estimator (MLE), we note that the
concentrated average log-likelihood is given by
\[
\Like_{n}(y,x;\alpha)=-\frac{d_{y}}{2}(\log2\pi+1)-\frac{1}{2}\log\det\left[\frac{1}{n}\sum_{i=1}^{n}\xi_{i}(\alpha)\xi_{i}(\alpha)^{\trans}\right]=L_{c}(T_{n};\alpha)
\]
which is three times continuously differentiable in $\alpha$ and
$T_{n}$, so long as $\mathcal{T}_{n}$ is non-singular. By the implicit
function theorem, it follows that $\hat{\alpha}_{n}$ may be regarded
as a smooth function of $T_{n}$. Noting the usual formula for the
ML estimates of $\Sigma_{\xi}$, this holds also for the components
of $\theta$ referring to $\Sigma_{\xi}^{-1}$, whence
\begin{equation}
\hat{\theta}_{n}^{m}(\beta,\lambda)=h[T_{n}^{m}(\beta,\lambda)]\label{eq:suffstatest}
\end{equation}
for some $h$ that is twice continuously differentiable on the set
where $\mathcal{T}_{n}^{m}$ has full rank. Under \enuref{L:nonsingular},
this occurs uniformly on $\Beta\times\Lambda$ w.p.a.1., and so to
avoid tiresome circumlocution, we shall simply treat $h$ as if it
were everywhere twice continuously differentiable throughout the sequel.
Letting $T(\beta,\lambda)\defeq\expect T_{n}^{0}(\beta,\lambda)$,
we note that the population binding function is given by
\begin{equation}
\bind(\beta,\lambda)=h[T(\beta,\lambda)].\label{eq:bind}
\end{equation}

Define $\varphi_{n}^{m}(\beta,\lambda)\defeq n^{1/2}[T_{n}^{m}(\beta,\lambda)-T(\beta,\lambda)]$,
and let $[\varphi^{m}(\beta,\lambda)]_{m=0}^{M}$ denote a vector-valued
continuous Gaussian process on $\Beta\times\Lambda$ with covariance
kernel
\[
\cov(\varphi^{m_{1}}(\beta_{1},\lambda_{1}),\varphi^{m_{2}}(\beta_{2},\lambda_{2}))=\cov(T_{n}^{m_{1}}(\beta_{1},\lambda_{1}),T_{n}^{m_{2}}(\beta_{2},\lambda_{2})).
\]
Note that \enuref{L:moment}, in particular the requirement that $\expect\smlnorm{z_{i}}^{4}<\infty$,
ensures that this covariance exists and is finite.
\begin{lem}
\label{lem:Tlimits}~
\begin{enumerate}
\item \label{enu:FCLT} $\varphi_{n}^{m}(\beta,\lambda)\wkc\varphi^{m}(\beta,\lambda)$
in $\ell^{\infty}(\Beta\times\Lambda)$, jointly for $m\in\{0,\ldots,M\}$;
and
\item \label{enu:derivULLN}if \eqref{lambdacond} holds for $l^{\prime}=l\in\{1,2\}$,
then
\begin{equation}
\sup_{\beta\in\Beta}\smlnorm{\partial_{\beta}^{l}T_{n}^{m}(\beta,\lambda_{n})-\partial_{\beta}^{l}T(\beta,0)}=o_{p}(1)\label{eq:levderiv}
\end{equation}

\end{enumerate}
\end{lem}

By an application of the delta method, we thus have
\begin{cor}
\label{cor:thetalims} For $\dot{h}(\beta,\lambda)\defeq\partial_{\beta}h[T(\beta,\lambda)]$,
\begin{equation}
\psi_{n}^{m}(\beta,\lambda)\defeq n^{1/2}[\hat{\theta}_{n}^{m}(\beta,\lambda)-\bind(\beta,\lambda)]\wkc\dot{h}(\beta,\lambda)\varphi^{m}(\beta,\lambda)\eqdef\psi^{m}(\beta,\lambda)\label{eq:psicvg}
\end{equation}
in $\ell^{\infty}(\Beta\times\Lambda)$, jointly for $m\in\{0,\ldots,M\}$.
\end{cor}

The proof of \lemref{Tlimits} appears in \appref{Tlimitproof}.
\begin{proof}[Proof of \propref{sufficiency}]
 \enuref{H:twicediff} follows from the twice continuous differentiability
of $L$ in \eqref{suffstatlike}. The first part of \enuref{H:unifcmpct}
is an immediate consequence of \lemref{Tlimits}\enuref{FCLT} and
the smoothness of $L$; the second part is implied by \eqref{finitescore}
and Lemma~2.4 in \citet{New94}. \enuref{H:stocheq} follows from
\corref{thetalims}, and immediately entails that, for $\beta_{n}=\beta_{0}+o_{p}(1)$
and $m\in\{1,\ldots,M\}$, $\psi_{n}^{m}(\beta_{n},\lambda_{n})=\psi_{n}^{m}(\beta_{0},0)+o_{p}(1)$,
where
\[
\psi_{n}^{m}(\beta_{0},0)=n^{1/2}[\hat{\theta}_{n}^{m}(\beta_{0},0)-\bind(\beta_{0},0)]=-H^{-1}\frac{1}{n^{1/2}}\sum_{i=1}^{n}\dot{\like}_{i}^{m}(\beta_{0},0;\theta_{0})+o_{p}(1)
\]
for $m\in\{0,1,\ldots,M\}$; the final equality follows from the consistency
of $\hat{\theta}_{n}$ (as implied by \corref{thetalims}) and the
arguments used to prove Theorem~3.1 in \citet{New94}. By definition,
$\phi_{n}^{m}\defeq n^{-1/2}\sum_{i=1}^{n}\dot{\like}_{i}^{m}(\beta_{0},0;\theta_{0})$
whereupon the rest of \enuref{H:limitdist} follows by the central
limit theorem, in view of \enuref{L:iid} and \eqref{finitescore}.
Finally, \enuref{H:deriv} follows from \eqref{suffstatest}, \eqref{bind},
\lemref{Tlimits}\enuref{derivULLN} and the chain rule.
\end{proof}

\section{Proof of \lemref{Tlimits}\label{app:Tlimitproof}}

For the purposes of the proofs undertaken in this section, we may
suppose without loss of generality that $\tilde{D}=I_{d_{y}}$ in
\enuref{L:linear}, $\gamma(\beta)=\beta$ in \enuref{L:gamma}, and
$\smlnorm K_{\infty}\leq1$. Recalling \eqref{ytilde} above, we have
\begin{equation}
y_{r}(\beta,\lambda)=\widx_{r}(\beta)\cdot\prod_{s\in\set S_{r}}K_{\lambda}[\vidx_{s}(\beta)]\eqdef\widx_{r}(\beta)\cdot\kprod(\set S_{r};\beta,\lambda).\label{eq:yr}
\end{equation}
Let $\dot{K}$ and $\ddot{K}$ respectively denote the first and second
derivatives of $K$. For future reference, we here note that
\begin{align}
\partial_{\beta}y_{r}(\beta,\lambda) & =z_{wr}\cdot\kprod(\set S_{r};\beta,\lambda)+\lambda^{-1}w_{r}(\beta)\sum_{s\in\set S_{r}}z_{vs}\cdot\kprod_{s}(\set S_{r};\beta,\lambda)\label{eq:grad}\\
 & \eqdef D_{r1}(\beta,\lambda)+\lambda^{-1}D_{r2}(\beta,\lambda)\nonumber 
\end{align}
where $z_{vr}\defeq\Pi_{vr}^{\trans}z$, $z_{wr}\defeq\Pi_{wr}^{\trans}z$
and $\kprod_{s}(\set S;\beta,\lambda)\defeq\dot{K}_{\lambda}[v_{s}(\beta)]\cdot\kprod(\set S\backslash\{s\};\beta,\lambda)$;
and
\begin{align}
\partial_{\beta}^{2}y_{r}(\beta,\lambda) & =\lambda^{-1}\sum_{s\in\set S_{r}}[z_{wr}z_{vs}^{\trans}+z_{vs}z_{wr}^{\trans}]\cdot\kprod_{s}(\set S_{r};\beta,\lambda)\label{eq:hess}\\
 & \qquad\qquad\qquad+\lambda^{-2}w_{r}(\beta)\sum_{s\in\set S_{r}}\sum_{t\in\set S_{r}}z_{vs}z_{vt}^{\trans}\cdot\kprod_{st}(\class S_{r};\beta,\lambda)\nonumber \\
 & \eqdef\lambda^{-1}H_{r1}(\beta,\lambda)+\lambda^{-2}H_{r2}(\beta,\lambda)\nonumber 
\end{align}
for
\[
\kprod_{st}(\class S;\beta,\lambda)\defeq\begin{cases}
\ddot{K}_{\lambda}[v_{s}(\beta)]\cdot\kprod(\set S\backslash\{s\};\beta,\lambda) & \text{if }s=t,\\
\dot{K}_{\lambda}[v_{s}(\beta)]\cdot\dot{K}_{\lambda}[v_{t}(\beta)]\cdot\kprod(\set S\backslash\{s,t\};\beta,\lambda) & \text{if }s\neq t.
\end{cases}
\]

\subsection{Proof of part~(ii)\label{app:parttwoproof}}

In view of \eqref{BigT}, the scalar elements of $T_{n}(\beta,\lambda)$
that depend on $(\beta,\lambda)$ take either of the following forms:
\begin{align}
\tau_{n1}(\beta,\lambda) & \defeq\expect_{n}[y_{r}(\beta,\lambda)y_{s}(\beta,\lambda)] & \tau_{n2}(\beta,\lambda) & \defeq\expect_{n}[y_{r}(\beta,\lambda)x_{t}]\label{eq:taus}
\end{align}
for some $r,s\in\{1,\ldots,d_{y}\}$, or $t\in\{1,\ldots,d_{x}\}$,
where $\expect_{n}f(\beta,\lambda)\defeq\frac{1}{n}\sum_{i=1}^{n}f(z_{i};\beta,\lambda)$.
(Throughout the following, all statements involving $r$, $s$ and
$t$ should be interpreted as holding for all possible values of these
indices.) For $k\in\{1,2\}$ and $l\in\{0,1,2\}$, define $\tau_{k}(\beta,\lambda)\defeq\expect\tau_{nk}(\beta,\lambda)$
-- a typical scalar element of $T(\beta,\lambda)$ -- and $\tau_{k}^{[l]}(\beta,\lambda)\defeq\expect\partial_{\beta}^{l}\tau_{nk}(\beta,\lambda)$.
Thus part~\enuref{derivULLN} of \lemref{Tlimits} will follow once
we have shown that 
\begin{equation}
\partial_{\beta}^{l}\tau_{nk}(\beta,\lambda_{n})=\tau_{k}^{[l]}(\beta,\lambda_{n})+o_{p}(1)=\partial_{\beta}^{l}\tau_{k}(\beta,0)+o_{p}(1)\label{eq:tauderiv}
\end{equation}
uniformly in $\beta\in\Beta$. The second equality in \eqref{tauderiv}
is implied by
\begin{lem}
\label{lem:targetcvg}$\tau_{k}^{[l]}(\beta,\lambda_{n})\inprob\partial_{\beta}^{l}\tau_{k}(\beta,0)$,
uniformly on $\Beta$, for $k\in\{1,2\}$ and $l\in\{0,1,2\}$.
\end{lem}

The proof appears at the end of this section. We turn next to the
first equality in \eqref{tauderiv}. We require the following definitions.
A function $F:\set Z\elmap\reals$ is an \emph{envelope} for the class
$\class F$ if $\sup_{f\in\class F}\smlabs{f(z)}\leq F(z)$. For a
probability measure $\Qrob$ and a $p\in(1,\infty)$, let $\smlnorm f_{p,\Qrob}\defeq(\expect_{\Qrob}\smlabs{f(z_{i})}^{p})^{1/p}$.
$\class F$ is \emph{Euclidean} for the envelope $F$ if
\[
\sup_{\Qrob}N(\epsilon\smlnorm F_{1,\Qrob},\class F,L_{1,\Qrob})\leq C_{1}\epsilon^{-C_{2}}
\]
for some $C_{1}$ and $C_{2}$ (depending on $\class F$), where $N(\epsilon,\class F,L_{1,\Qrob})$
denotes the minimum number of $L_{1,\Qrob}$-balls of diameter $\epsilon$
needed to cover $\class F$. For a parametrized family of functions
$g(\beta,\lambda)=g(z;\beta,\lambda):\spc Z\elmap\reals^{d_{1}\times d_{2}}$,
let $\class F(g)\defeq\{g(\beta,\lambda)\mid(\beta,\lambda)\in\Beta\times\Lambda\}$.
Since $\Beta$ is compact, we may suppose without loss of generality
that $\Beta\subseteq\{\beta\in\reals^{d_{\beta}}\mid\smlnorm{\beta}\leq1\}$,
whence recalling \eqref{vwidx} and \eqref{widx2} above, 
\[
\smlabs{w_{r}(z;\beta)}\leq W_{r}\leq\begin{cases}
\smlnorm z & \text{if }r\in\{1,\ldots d_{w}\}\\
1 & \text{if }r\in\{d_{w}+1,\ldots d_{y}\}.
\end{cases}
\]
Thus by Lemma~22 in \citet{NP87AS}
\begin{enumerate}[{label=\textnormal{\smaller[0.76]{\Alph{section}\arabic*}}}]
\item for $\lprod\in\{\kprod,\kprod_{s},\kprod_{st}\}$, $s,t\in\{1,\ldots,d_{y}\}$
and $\class S\subseteq\{1,\ldots,d_{v}\}$, the class
\[
\class F(\lprod,\set S)\defeq\{\lprod(\set S;\beta,\lambda)\mid(\beta,\lambda)\in\Beta\times\Lambda\}
\]
is Euclidean with constant envelope; and
\item for $r\in\{1,\ldots,d_{y}\}$, $\class F(w_{r})$ is Euclidean for
$W_{r}$.
\end{enumerate}
It therefore follows by a slight adaptation of the proof of Theorem~9.15
in \citet{Kos08book} that
\begin{enumerate}[resume, resume*]
\item \label{enu:eclFy}$\class F(y_{r})$ is Euclidean for $W_{r}$;
\item \label{enu:eclDy}$\class F(y_{r}D_{s1})$ and $\class F(y_{r}D_{s2})$
are Euclidean for $W_{r}W_{s}\smlnorm z$
\item \label{enu:eclDx}$\class F(x_{t}D_{s1})$ and $\class F(x_{t}D_{s2})$
are Euclidean for $W_{s}\smlnorm z^{2}$;
\item \label{enu:eclDD}$\class F(D_{s1}D_{r1}^{\trans})$, $\class F(D_{s1}D_{r2}^{\trans})$,
$\class F(D_{s2}D_{r1}^{\trans})$ and $\class F(D_{s2}D_{r2}^{\trans})$
are Euclidean for $W_{r}W_{s}\smlnorm z^{2}$;
\item \label{enu:eclHy}$\class F(y_{s}H_{r1})$ and $\class F(y_{s}H_{r2})$
are Euclidean for $W_{r}W_{s}\smlnorm z^{2}$; and
\item \label{enu:eclHx}$\class F(x_{t}H_{r1})$ and $\class F(x_{t}H_{r2})$
are Euclidean for $W_{s}\smlnorm z^{3}$.
\end{enumerate}
Let $\mu_{n}f\defeq\frac{1}{n}\sum_{i=1}^{n}[f(z_{i})-\expect f(z_{i})]$.
Using the preceding facts, and the uniform law of large numbers given
as \propref{LLN} below, we may prove
\begin{lem}
\label{lem:derivLLN} The convergence
\begin{equation}
\sup_{\beta\in\Beta}\mu_{n}\smlabs{\partial_{\beta}^{l}[y_{s}(\beta,\lambda_{n})y_{r}(\beta,\lambda_{n})]}+\sup_{\beta\in\Beta}\mu_{n}\smlabs{x_{t}\partial_{\beta}^{l}y_{r}(\beta,\lambda_{n})}=o_{p}(1).\label{eq:derivLLN}
\end{equation}
holds for $l=0$, and also for $l\in\{1,2\}$ if \eqref{lambdacond}
holds with $l^{\prime}=l$.
\end{lem}

The first equality in \eqref{levderiv} now follows, and thus part~\enuref{derivULLN}
of \lemref{Tlimits} is proved.
\begin{proof}[Proof of \lemref{targetcvg}]
 Suppose $l=2$; the proof when $l=1$ is analogous (and is trivial
when $l=0$). Noting that
\begin{equation}
\partial_{\beta}^{2}(y_{r}y_{s})=y_{s}\partial_{\beta}^{2}y_{r}+(\partial_{\beta}y_{r})(\partial_{\beta}y_{s})^{\trans}+(\partial_{\beta}y_{s})(\partial_{\beta}y_{r})^{\trans}+y_{r}\partial_{\beta}^{2}y_{s},\label{eq:d2yrys}
\end{equation}
it follows from \eqref{grad}, \eqref{hess}, \enuref{eclDD} and
\enuref{eclHy} that for every $\lambda\in(0,1]$, 
\[
\smlnorm{\partial_{\beta}^{2}(y_{r}y_{s})}\lesssim\lambda^{-2}W_{r}W_{s}(\smlnorm z^{2}\pmax1),
\]
which does not depend on $\beta$, and is integrable by \enuref{L:moment}.
(Here $a\lesssim b$ denotes that $a\leq Cb$ for some constant $C$
not depending on $b$.) Thus by the dominated derivatives theorem,
the second equality in
\[
\tau_{1}^{[2]}(\beta,\lambda)=\expect\partial_{\beta}^{2}\tau_{n1}(\beta,\lambda)=\partial_{\beta}^{2}\expect\tau_{n1}(\beta,\lambda)=\partial_{\beta}^{2}\tau_{1}(\beta,\lambda)
\]
holds for every $\lambda\in(0,1]$; the other equalities follow from
the definitions of $\tau_{k}^{[l]}$ and $\tau_{k}$. Deduce that,
so long as $\lambda_{n}>0$ (as per the requirements of \propref{sufficiency}
above),
\[
\tau_{1}^{[2]}(\beta,\lambda_{n})=\partial_{\beta}^{2}\tau_{1}(\beta,\lambda_{n})\inprob\partial_{\beta}^{2}\tau_{1}(\beta,0)
\]
by the uniform continuity of $\partial_{\beta}^{2}\tau_{1}$ on $\Beta\times\Lambda$.
A similar reasoning -- but now using \enuref{eclHx} -- gives the
same result for $\tau_{2}^{[2]}$.
\end{proof}

The proof of \lemref{derivLLN} requires the following result. Let
$\filtg_{\widx,x}$ denote the $\sigma$-field generated by $\err_{\widx}(z_{i})$
and $x(z_{i})$, and let $\err_{\vidx}$ denote those elements of
$\err$ that are not present in $\err_{\widx}$. Recall that $\err_{\vidx}\indep\filtg_{\widx,x}$.
\begin{lem}
\label{lem:2ndmom}For every $p\in\{0,1,2\}$, $s,t\in\{1,\ldots,d_{v}\}$,
$\class S\subseteq\{1,\ldots,d_{v}\}$ and $\lprod\in\{\kprod_{s},\kprod_{st}\}$
\begin{equation}
\expect[\smlnorm{z_{\vidx s}}^{p}\smlnorm{z_{\vidx t}}^{p}\lprod(\set S;\beta,\lambda)^{2}\mid\filtg_{\widx,x}]\lesssim\lambda\expect[\smlnorm{z_{\vidx s}}^{p}\smlnorm{z_{\vidx t}}^{p}\mid\filtg_{\widx,x}].\label{eq:condexpbound}
\end{equation}
\end{lem}
\begin{proof}
Note that for any $\lprod\in\{\kprod_{s},\kprod_{st}\}$, 
\[
\lprod(\class S;\beta,\lambda)\lesssim L_{\lambda}[\nu_{s}(\beta)]
\]
where $L(x)=\max\{\smlabs{\dot{K}(x)},\smlabs{\ddot{K}(x)}\}$. Let
$d$ denote the dimensionality of $\err_{\vidx}$, and fix a $\beta\in\Beta$.
By \enuref{L:density} and \enuref{L:marginals}, there is a $k\in\{1,\ldots d\}$,
possibly depending on $\beta$, and an $\epsilon>0$ which does not,
such that
\[
\vidx_{s}(\beta)=\vidx_{s}^{\ast}(\beta)+\beta_{k}^{\ast}\err_{\vidx k}
\]
with $\smlabs{\beta_{k}^{\ast}}\geq\epsilon$ and $\vidx_{s}^{\ast}(\beta)\indep\err_{\vidx k}$.
Let $\filtg_{\widx,x}^{\ast}\defeq\filtg_{\widx,x}\pmax\sigma(\{\err_{\vidx l}\}_{l\neq k})$,
so that $\vidx_{s}^{\ast}(\beta)$ is $\filtg_{\widx,x}^{\ast}$-measurable,
and let $f_{k}$ denote the density of $\err_{\vidx k}$. Then for
any $q\in\{0,\ldots,4\}$, 
\begin{align}
\expect\left[\smlabs{\err_{\vidx k}}^{q}\lprod(\set S;\beta,\lambda)^{2}\mid\filtg_{\widx,x}^{\ast}\right] & \lesssim\expect\left[\smlabs{\err_{\vidx k}}^{q}L_{\lambda}^{2}(\nu_{s}^{\ast}(\beta)+\beta_{k}^{\ast}\err_{\vidx k})\mid\filtg_{\widx,x}^{\ast}\right]\nonumber \\
 & =\int_{\reals}\smlabs u^{q}L_{\lambda}^{2}(\vidx_{s}^{\ast}(\beta)+\beta_{k}^{\ast}u)f_{k}(u)\diff u\nonumber \\
 & \lesssim(\beta_{k}^{\ast})^{-1}\lambda\int_{\reals}L^{2}(u)\diff u\cdot\sup_{u\in\reals}\smlabs u^{q}f_{k}(u)\nonumber \\
 & \lesssim\epsilon^{-1}\lambda,\label{eq:condbound}
\end{align}
since $\sup_{u\in\reals}\smlabs u^{q}f_{k}(u)<\infty$ under \enuref{L:density}.
Finally, we may partition $z_{\vidx s}=(z_{\vidx s}^{\ast\trans},\err_{\vidx k})^{\trans}$
and $z_{\vidx t}=(z_{\vidx t}^{\ast\trans},\err_{\vidx k})^{\trans}$,
with the possibility that $z_{\vidx s}=z_{\vidx s}^{\ast}$ and $z_{\vidx t}=z_{\vidx t}^{\ast}$.
Then by \eqref{condbound},
\[
\expect\left[\smlnorm{z_{\vidx s}}^{p}\smlnorm{z_{\vidx t}}^{p}\lprod(\set S;\beta,\lambda)^{2}\mid\filtg_{\widx,x}^{\ast}\right]\lesssim\lambda\smlnorm{z_{\vidx s}^{\ast}}^{p}\smlnorm{z_{\vidx t}^{\ast}}^{p}\leq\lambda\smlnorm{z_{\vidx s}}^{p}\smlnorm{z_{\vidx t}}^{p}.
\]
The result now follows by the law of iterated expectations.
\end{proof}
 
\begin{proof}[Proof of \lemref{derivLLN}]
 We shall only provide the proof for first term on the left side
of \eqref{derivLLN}, when $l=2$; the proof in all other cases are
analogous, requiring appeal only to \propref{LLN} (or Theorem~2.4.3
in \citealp{VVW96}, when $l=0$) and the appropriate parts of \enuref{eclFy}--\enuref{eclHx}.

Recalling the decomposition of $\partial_{\beta}^{2}(y_{r}y_{s})$
given in \eqref{d2yrys} above, we are led to consider
\begin{equation}
(\partial_{\beta}y_{r})(\partial_{\beta}y_{s})^{\trans}=D_{s1}D_{r1}^{\trans}+\lambda^{-1}D_{s2}D_{r1}^{\trans}+\lambda^{-1}D_{s1}D_{r2}^{\trans}+\lambda^{-2}D_{s2}D_{r2}^{\trans}\label{eq:H:DD}
\end{equation}
and
\begin{equation}
y_{s}\partial_{\beta}^{2}y_{r}=\lambda^{-1}y_{s}H_{r1}+\lambda^{-2}y_{s}H_{r2}.\label{eq:H:yH}
\end{equation}
Note that by \lemref{2ndmom}, and \enuref{L:moment} 
\begin{align*}
\expect\smlnorm{y_{s}H_{r2}}^{2} & \lesssim\expect\left[\smlabs{\widx_{s}(\beta)}^{2}\smlabs{\widx_{r}(\beta)}^{2}\sum_{s\in\set S_{r}}\sum_{t\in\set S_{r}}\expect\left[\smlnorm{z_{vs}}^{2}\smlnorm{z_{vt}}^{2}\smlabs{\kprod_{st}(\class S_{r};\beta,\lambda)}^{2}\mid\filtg_{\widx,x}\right]\right]\\
 & \lesssim\lambda\expect\left[W_{s}^{2}W_{r}^{2}\sum_{s\in\set S_{r}}\sum_{t\in\set S_{r}}\expect\smlnorm{z_{vs}}^{2}\smlnorm{z_{vt}}^{2}\right]\\
 & \lesssim\lambda
\end{align*}
and analogously for each of $H_{r1}$, $D_{s1}D_{r1}^{\trans}$, $D_{s2}D_{r1}^{\trans}$,
$D_{s1}D_{r2}^{\trans}$ and $D_{s2}D_{r2}^{\trans}$. By \enuref{eclDD}
and \enuref{eclHy}, the classes formed from these parametrized functions
are Euclidean, with envelopes that are $p_{0}$-integrable under \enuref{L:moment}
($p_{0}\geq2$). 

Application of \propref{LLN} to each of the terms in \enuref{eclDD}
and \enuref{eclHy}, with $\lambda$ playing the role of $\delta^{-1}$
there, thus yields the result. Negligibility of the final terms in
\eqref{H:DD} and \eqref{H:yH} entail the most stringent conditions
on the rate at which $\lambda_{n}$ may shrink to zero, due to the
multiplication of these by $\lambda^{-2}$.
\end{proof}

\subsection{Proof of part~(i)}

The typical scalar elements of $T_{n}$ are as displayed in \eqref{taus}
above, i.e.\ they are averages of random functions of the form $\zeta_{1}(\beta,\lambda)\defeq y_{r}(\beta,\lambda)y_{s}(\beta,\lambda)$
or $\zeta_{2}(\beta,\lambda)\defeq x_{t}y_{r}(\beta,\lambda)$, for
$r,s\in\{1,\ldots,d_{y}\}$ and $t\in\{1,\ldots,d_{x}\}$. It follows
from \enuref{eclFy} that $\class F(\zeta_{1})$ and $\class F(\zeta_{2})$
are Euclidean, with envelopes $F_{1}\defeq W_{r}W_{s}$ and $F_{2}\defeq\smlnorm zW_{r}$
respectively. Since both envelopes are square integrable under \enuref{L:moment},
we have
\[
\sup_{\Qrob}N(\epsilon\smlnorm{F_{k}}_{2,\Qrob},\class F(\zeta_{k}),L_{2,\Qrob})\leq C_{1}^{\prime}\epsilon^{-C_{2}^{\prime}}
\]
for $k\in\{1,2\}$. Hence \eqref{psicvg} follows by Theorem~2.5.2
in \citet{VVW96}.

\section{A uniform-in-bandwidth law of large numbers\label{app:ullnproof}}

This section provides a uniform law of large numbers (ULLN) for certain
classes of parametrized functions, broad enough to cover products
involving $K_{\lambda}[\vidx_{s}(\beta)]$, and such generalizations
as appear in \lemref{2ndmom} above. Our ULLN holds \emph{uniformly}
in the inverse `bandwith' parameter $\delta=\lambda^{-1}$; in this
respect, it is related to some of the results proved in \citet{EM05AS}.
However, while their arguments could be adapted to our problem, these
would lead to stronger conditions on the bandwidth: in particular,
$p$ would have to be replaced by $2p$ in \propref{LLN} below. (On
the other hand, their results yield explicit rates of uniform convergence,
which are not of concern here.)

Consider the (pointwise measurable) function class
\[
\class F_{\Delta}\defeq\{z\elmap f_{(\gamma,\delta)}(z)\mid(\gamma,\delta)\in\Gamma\times\Delta\},
\]
and put $\class F\defeq\class F_{[1,\infty)}$. The functions $f_{(\gamma,\delta)}:\spc Z\setmap\reals^{d}$
satisfy:
\begin{enumerate}[{label=\textnormal{\smaller[0.76]{\Alph{section}\arabic*}}}]
\item \label{enu:varbnd}$\sup_{\gamma\in\Gamma}\expect\smlnorm{f_{(\gamma,\delta)}(z_{0})}^{2}\lesssim\delta^{-1}$
for every $\delta>0$.
\end{enumerate}
Let $F:\spc Z\setmap\reals$ denote an envelope for $\class F$, in
the sense that
\[
\sup_{(\gamma,\delta)\in\Gamma\times[1,\infty)}\smlnorm{f_{(\gamma,\delta)}(z)}\leq F(z)
\]
for all $z\in\spc Z$. We will suppose that $F$ may be chosen such
that, additionally,
\begin{enumerate}[resume, resume*]
\item \label{enu:Fp}$\expect\smlabs{F(z_{0})}^{p}<\infty$; and
\item \label{enu:entropy}$\sup_{\Qrob}N(\epsilon\smlnorm F_{1,\Qrob},\class F,L_{1,\Qrob})\leq C\epsilon^{-d}$
for some $d\in(0,\infty)$.
\end{enumerate}
Let $\{\abv{\delta}_{n}\}$ denote a real sequence with $\abv{\delta}_{n}\geq1$,
and $\Delta_{n}\defeq[1,\abv{\delta}_{n}]$.
\begin{prop}
\label{prop:LLN} Under \enuref{varbnd}--\enuref{entropy}, if $n^{1-1/p}/\abv{\delta}_{n}^{2m-1}\log(\abv{\delta}_{n}\pmax n)\goesto\infty$
for some $m\geq1$, then

\begin{equation}
\sup_{(\gamma,\delta)\in\Gamma\times\Delta_{n}}\delta^{m}\smlnorm{\mu_{n}f_{(\gamma,\delta)}}=o_{p}(1).\label{eq:growingLLN}
\end{equation}

\end{prop}
 
\begin{rem}
Suppose $\delta_{n}$ is an $\filt$-measurable sequence for which
$n^{1-1/p}/\delta_{n}^{2m-1}\log(\delta_{n}\pmax n)\inprob\infty$.
Then for every $\epsilon>0$, there exists a deterministic sequence
$\{\abv{\delta}_{n}\}$ satisfying the requirements of \propref{LLN},
and for which $\limsup_{n\goesto\infty}\Prob\{\delta_{n}\leq\abv{\delta}_{n}\}>1-\epsilon$.
Deduce that 
\[
\sup_{\gamma\in\Gamma}\delta_{n}^{m}\smlnorm{\mu_{n}f_{(\gamma,\delta_{n})}}=o_{p}(1).
\]

\end{rem}

The proof requires the following
\begin{lem}
\label{lem:bernstein} Suppose $\class F$ is a (pointwise measurable)
class with envelope $F$, satisfying
\begin{enumerate}
\item \label{enu:unifbound}$\smlnorm F_{\infty}\leq\tau$;
\item \label{enu:unifvar}$\sup_{f\in\class F}\smlnorm f_{2,\Prob}\leq\sigma$;
and
\item \label{enu:unifcover} $\sup_{\Qrob}N(\epsilon\smlnorm F_{1,\Qrob},\class F,L_{1,\Qrob})\leq C\epsilon^{-d}$.
\end{enumerate}
Let $\theta\defeq\tau^{-1/2}\sigma$, $m\in\naturals$ and $x>0$.
Then there exist $C_{1},C_{2}\in(0,\infty)$, not depending on $\tau$,
$\sigma$ or $x$, such that
\begin{equation}
\Prob\left\{ \sigma^{-2}\sup_{f\in\class F}\smlabs{\mu_{n}f}>x\right\} \leq C_{1}\exp[-C_{2}n\theta^{2}(1+x^{2})+d\log(\theta^{-2}x^{-1})]\label{eq:llnbnd}
\end{equation}
for all $n\geq\tfrac{1}{8}x^{-2}\theta^{-2}$.
\end{lem}
 
\begin{proof}[Proof of \propref{LLN}]
 We first note that, by \enuref{Fp}, 
\[
\max_{i\leq n}\smlabs{F(z_{i})}=o_{p}(n^{-1/p})
\]
and so, letting $f_{(\gamma,\delta)}^{n}(z)\defeq f_{(\gamma,\delta)}(z)\indic\{F(z)\leq n^{1/p}\}$,
we have
\[
\Prob\left\{ \sup_{(\gamma,\delta)\in\Gamma\times\Delta_{n}}\delta^{m}\smlabs{\mu_{n}[f_{(\gamma,\delta)}-f_{(\gamma,\delta)}^{n}]}=0\right\} \leq\Prob\left\{ \max_{i\leq n}\smlabs{F(z_{i})}>n^{1/p}\right\} =o(1).
\]
It thus suffices to show that \eqref{growingLLN} holds when $f_{(\gamma,\delta)}$
is replaced by $f_{(\gamma,\delta)}^{n}$. Since \enuref{varbnd}
and \enuref{entropy} continue to hold after this replacement, it
suffices to prove \eqref{growingLLN} when \enuref{Fp} is replaced
by the condition that $\smlnorm F_{\infty}\leq n^{1/p}$, which shall
be maintained throughout the sequel. (The dependence of $f$ and $F$
upon $n$ will be suppressed for notational convenience.)

Letting $\delta_{k}\defeq\e^{k}$, define $\Delta_{nk}\defeq[\delta_{k},\delta_{k+1}\pmin\abv{\delta}_{n}]$
for $k\in\{0,\ldots,K_{n}\}$, where $K_{n}\defeq\log\abv{\delta}_{n}$;
observe that $\Delta_{n}=\Union_{k=0}^{K_{n}}\Delta_{nk}$. Set
\[
\class F_{nk}\defeq\{z\elmap f_{(\gamma,\delta)}(z)\mid(\gamma,\delta)\in\Gamma\times\Delta_{nk}\}
\]
and note that $\smlnorm F_{\infty}\leq n^{1/p}$ and $\sup_{f\in\class F_{nk}}\smlnorm f_{2,\Prob}\leq\delta_{k}^{-1/2}$.
Under \enuref{entropy}, we may apply apply \lemref{bernstein} to
each $\class F_{nk}$, with $(\tau,\sigma)=(n^{1/p},\delta_{k}^{-1/2})$
and $x=\delta_{k}^{1-m}\epsilon$, for some $\epsilon>0$. There thus
exist $C_{1},C_{2}\in(0,\infty)$ depending on $\epsilon$ such that
\begin{align}
\Prob\left\{ \sup_{(\gamma,\delta)\in\Gamma\times\Delta_{n}}\delta^{m}\smlabs{\mu_{n}f_{(\gamma,\delta)}}>\epsilon\right\}  & \leq\sum_{k=0}^{K_{n}}\Prob\left\{ \delta_{k}^{m}\sup_{(\gamma,\delta)\in\Gamma\times\Delta_{nk}}\smlabs{\mu_{n}f_{(\gamma,\delta)}}>\e^{-1}\epsilon\right\} \nonumber \\
 & \leq C_{1}\sum_{k=0}^{K_{n}}\exp[-C_{2}n\theta_{nk}^{2}\delta_{k}^{2(1-m)}+d\log(\theta_{nk}^{-2}\delta_{k}^{m-1})]\label{eq:ullnbnd}
\end{align}
where $\theta_{nk}\defeq n^{-1/2p}\delta_{k}^{-1/2}$, provided 
\begin{equation}
n\geq\tfrac{1}{8}\delta_{k}^{2(m-1)}\theta_{nk}^{-2}\epsilon^{-2}\sep\forall k\in\{0,\ldots,K_{n}\}\impliedby n^{1-1/p}/\abv{\delta}_{n}^{2m-1}\geq\tfrac{1}{8}\epsilon^{-2},\label{eq:ncond}
\end{equation}
which holds for all $n$ sufficiently large. In obtaining \eqref{ncond}
we have used $\delta_{k}\leq\abv{\delta}_{n}$ and $\theta_{nk}\geq n^{-1/2p}\abv{\delta}_{n}^{-1/2}$,
and these further imply that \eqref{ullnbnd} may be bounded by 
\[
C_{1}(\log\abv{\delta}_{n})\exp[-C_{2}n^{1-1/p}\abv{\delta}_{n}^{-2m-1}(1+\epsilon^{2})+d\log(\abv{\delta}_{n}^{m}n^{1/p})]\goesto0
\]
as $n\goesto\infty$. Thus \eqref{growingLLN} holds.
\end{proof}
 
\begin{proof}[Proof of \lemref{bernstein}]
 Suppose \enuref{unifcover} holds. Define $\class G\defeq\{\tau^{-1}f\mid f\in\class F\}$,
and $G\defeq\tau^{-1}F$. Then
\[
\sup_{g\in\class G}\smlnorm g_{2,\Prob}\leq\tau^{-1}\sup_{f\in\class F}\smlnorm f_{2,\Prob}\leq\tau^{-1/2}\sigma\eqdef\theta;
\]
$\smlnorm g_{\infty}\leq1$ for all $g\in\class G$; and since $\smlnorm{G_{n}}_{1,\Qrob}\leq1$,
$N(\epsilon,\class G,L_{1,\Qrob})\leq C\epsilon^{-d}$. Hence, by
arguments given in the proof of Theorem~II.37 in \citet{Pollard84},
there exist $C_{1},C_{2}>0$, depending on $x$, such that
\[
\Prob\left\{ \sigma^{-2}\sup_{f\in\class F}\smlabs{\mu_{n}f}>x\right\} =\Prob\left\{ \sup_{g\in\class G}\smlabs{\mu_{n}g}>\theta^{2}x\right\} \leq C_{1}\exp[-C_{2}n\theta^{2}(1+x^{2})+d\log(\theta^{-2}x^{-1})]
\]
for all $n\geq\tfrac{1}{8}x^{-2}\theta^{-2}$.
\end{proof}
\cleartooddpage

\section{Index of key notation\label{app:notation}}

\newref{ass}{name = Ass.~}

\newref{sub}{name = Sec.~}

\newref{sec}{name = Sec.~}

\newref{rem}{name = Rem.~}

\newref{thm}{name = Thm.~}

\newref{prop}{name = Prop.~}

\newref{app}{name = App.~}

\subsubsection*{Greek and Roman symbols}

Listed in (Roman) alphabetical order. Greek symbols are listed according
to their English names: thus $\Omega$, as `omega', appears before
$\theta$, as `theta'.

\noindent %
\begin{longtable}[l]{>{\raggedright}p{0.13\textwidth}>{\raggedright}p{0.67\textwidth}l}
$\beta$, $\beta_{0}$, $\Beta$ & structural model parameters, true value, parameter space \cellfill & \secref{model}\tabularnewline[\doublerulesep]
$\hat{\beta}_{nk}^{e}$ & GII estimator; near-minimizer of $Q_{nk}^{e}$ \cellfill & \subref{giidist}\tabularnewline[\doublerulesep]
$\abv{\beta}_{nk}^{e}(\beta^{(0)},r)$ & terminal value for routine $r$ started at $\beta^{(0)}$ \cellfill & \eqref{terminal}\tabularnewline[\doublerulesep]
$c_{n}$ & tuning sequence in the definition of $R_{nk}^{e}$ \cellfill & \subref{optperform}\tabularnewline[\doublerulesep]
$d_{\beta}$, $d_{\theta}$, \ldots{} & dimensionality of $\beta$, $\theta$, etc. \cellfill & \subref{ii}\tabularnewline[\doublerulesep]
$\expect_{n}f$ & sample average, $\frac{1}{n}\sum_{i=1}^{n}f(z_{i})$ \cellfill & \appref{parttwoproof}\tabularnewline[\doublerulesep]
$\eta_{it}$ & stochastic components of the structural model \cellfill & \secref{model}\tabularnewline[\doublerulesep]
$\filt$ & $\sigma$-field supporting all observed and simulated variates \cellfill & \subref{general}\tabularnewline[\doublerulesep]
$G$ & Jacobian of the population binding function \cellfill & \subref{giidist}\tabularnewline[\doublerulesep]
$G_{n}(\beta)$ & Jacobian of the smoothed sample binding function \cellfill & \remref{cvgconditions}\tabularnewline[\doublerulesep]
$\gamma(\beta)$ & (re-)parametrizes the structural model \cellfill & \eqref{vwidx}\tabularnewline[\doublerulesep]
$\gamma_{rk}$ & jackknifing weights \cellfill & \eqref{kextrap}\tabularnewline[\doublerulesep]
$H$ & auxiliary model (population) log-likelihood Hessian \cellfill & \subref{giidist}\tabularnewline[\doublerulesep]
$J$ & total number of alternatives \cellfill & \secref{model}\tabularnewline[\doublerulesep]
$k$ & order of jackknifing (unless otherwise defined) \cellfill & \enuref{R:jackknifing}\tabularnewline[\doublerulesep]
$k_{0}$ & maximum order (less $1$) of differentiability of $\beta\elmap\theta(\beta,\lambda)$
\cellfill & \enuref{R:uniqueness}\tabularnewline[\doublerulesep]
$K$, $K_{\lambda}$ & smoothing kernel, $K_{\lambda}(x)\defeq K(\lambda^{-1}x)$ \cellfill & \eqref{ytilde}\tabularnewline[\doublerulesep]
$\kprod$, $\kprod_{s}$, $\kprod_{st}$ & product of kernel-type functions \cellfill & \appref{ullnproof}\tabularnewline[\doublerulesep]
$\like(y_{i},x_{i};\theta)$ & $i$th contribution to auxiliary model log-likelihood \cellfill & \eqref{thetahat}\tabularnewline[\doublerulesep]
$\like(\beta,\lambda;\theta)$ & abbreviates $\like(y_{i}(\beta,\lambda),x_{i};\theta)$ & \subref{general}\tabularnewline[\doublerulesep]
$\ell^{\infty}(D)$ & space of bounded functions on the set $D$ \cellfill & \enuref{H:stocheq}\tabularnewline[\doublerulesep]
$\Like_{n}(y,x;\theta)$ & auxiliary model average log-likelihood \cellfill & \eqref{thetahat}\tabularnewline[\doublerulesep]
$\lambda$, $\Lambda$ & smoothing parameter, set of allowable values \cellfill & \subref{proposal}\tabularnewline[\doublerulesep]
$m$ & indexes the simulated dataset; $m=0$ denotes the data \cellfill & \subref{ii}\tabularnewline[\doublerulesep]
$M$ & total number of simulations \cellfill & \subref{ii}\tabularnewline[\doublerulesep]
$\mu_{n}f$ & centered sample average, $\frac{1}{n}\sum_{i=1}^{n}[f(z_{i})-\expect f(z_{i})]$
\cellfill & \appref{parttwoproof}\tabularnewline[\doublerulesep]
$n$  & total number of individuals \cellfill & \secref{model}\tabularnewline[\doublerulesep]
$N(\theta,\epsilon)$ & open ball of radius $\epsilon$ centered at $\theta$ \cellfill & \appref{prelimlem}\tabularnewline[\doublerulesep]
$\vidx_{r}(z;\beta)$ & linear index in structural model \cellfill & \eqref{vidx}\tabularnewline[\doublerulesep]
$\widx_{r}(z;\beta)$ & linear index in structural model \cellfill & \eqref{vidx}\tabularnewline[\doublerulesep]
$\Omega(U,V)$ & variance matrix function \cellfill & \eqref{Omega}\tabularnewline[\doublerulesep]
$p_{0}$ & order of moments possessed by model variates \cellfill & \enuref{L:moment}\tabularnewline[\doublerulesep]
$\phi_{n}^{m}$, $\phi^{m}$ & standardized auxiliary sample score and its weak limit \cellfill & \eqref{psitophi}\tabularnewline[\doublerulesep]
$\psi_{n}^{m}$, $\psi^{m}$ & centered auxiliary estimator process and its weak limit \cellfill & \enuref{H:stocheq}\tabularnewline[\doublerulesep]
$Q_{nk}^{e}$ & sample criterion for estimator $e$ (jackknifed) \cellfill & \subref{bias}\tabularnewline[\doublerulesep]
$Q_{k}^{e}$ & large-sample (unsmoothed) limit of $Q_{nk}^{e}$; note $Q_{k}^{e}=Q^{e}$
\cellfill & \subref{bias}\tabularnewline[\doublerulesep]
$\covmat$ & auxiliary model score covariance, $\expect\dot{\like}_{i}^{m}(\theta_{0})\dot{\like}_{i}^{m^{\prime}}(\theta_{0})^{\trans}$
for $m^{\prime}\neq m$ & \eqref{scorecovar}\tabularnewline[\doublerulesep]
$R_{nk}^{e}$, $R^{e}$ & set of near-roots of $Q_{nk}^{e}$, exact roots of $Q^{e}$ \cellfill & \subref{optperform}\tabularnewline[\doublerulesep]
$\varrho_{\min}(A)$ & smallest eigenvalue of symmetric matrix $A$ \cellfill & \subref{general}\tabularnewline[\doublerulesep]
$S_{nk}^{e}$, $S^{e}$ & subset of $R_{nk}^{e}$, $R^{e}$ satisfying second-order conditions
\cellfill & \eqref{S}\tabularnewline[\doublerulesep]
$\sigma_{\min}(B)$ & smallest singular value of matrix $B$ \cellfill & \subref{convergence}\tabularnewline[\doublerulesep]
$\varmat$ & auxiliary model score variance, $\expect\dot{\like}_{i}^{m}(\theta_{0})\dot{\like}_{i}^{m}(\theta_{0})^{\trans}$
\cellfill & \eqref{scorecovar}\tabularnewline[\doublerulesep]
$T$ & total number of time periods \cellfill & \secref{model}\tabularnewline[\doublerulesep]
$\theta$, $\Theta$ & auxiliary model parameters, parameter space \cellfill & \subref{ii}\tabularnewline[\doublerulesep]
$\theta_{0}$ & pseudo-true parameters implied by $\beta_{0}$ \cellfill & \subref{ii}\tabularnewline[\doublerulesep]
$\hat{\theta}_{n}$ & data-based estimate of $\theta$ \cellfill & \subref{ii}\tabularnewline[\doublerulesep]
$\hat{\theta}_{n}^{m}(\beta,\lambda)$ & simulation-based estimate of $\theta$ \cellfill & \eqref{simlest}\tabularnewline[\doublerulesep]
$\theta^{k}(\beta,\lambda)$ & population binding function (smoothed, jackknifed) \cellfill & \eqref{kextrap}\tabularnewline[\doublerulesep]
$\abv{\theta}_{n}^{k}(\beta,\lambda)$ & sample binding function (smoothed, jackknifed) \cellfill & \eqref{jackknifed}\tabularnewline[\doublerulesep]
$u_{itj}$ & utility of individual $i$ from alternative $j$ in period $t$ \cellfill & \secref{model}\tabularnewline[\doublerulesep]
$u_{itj}^{m}(\beta)$ & simulated utilities at $\beta$ \cellfill & \subref{proposal}\tabularnewline[\doublerulesep]
$U_{e}$ & ``Hessian'' component of limiting variance \cellfill & \eqref{UVe}\tabularnewline[\doublerulesep]
$V_{e}$ & ``score'' component of limiting variance \cellfill & \eqref{UVe}\tabularnewline[\doublerulesep]
w.p.a.1 & with probability approaching one \cellfill & \thmref{rootdist}\tabularnewline[\doublerulesep]
$W_{r}(z)$ & envelope for $\widx_{r}(z;\beta)$ \cellfill & \subref{general}\tabularnewline[\doublerulesep]
$W_{n}$, $W$ & Wald weighting matrix and its probability limit \cellfill & \subref{ii}\tabularnewline[\doublerulesep]
$x_{it}$ & exogenous covariates for individual $i$ in period $t$ \cellfill & \secref{model}\tabularnewline[\doublerulesep]
$y_{itj}$ & set $=1$ if individual $i$ chooses $j$ in period $t$ \cellfill & \secref{model}\tabularnewline[\doublerulesep]
$y_{itj}^{m}(\beta,\lambda)$ & smoothed simulated choice indicators at $\beta$ \cellfill & \subref{proposal}\tabularnewline[\doublerulesep]
$z_{i}^{m}$ & collects $x_{i}$ and $\eta_{i}^{m}$ \cellfill & \subref{general}\tabularnewline[\doublerulesep]
\end{longtable}

\subsubsection*{Symbols not connected to Greek or Roman letters}

Ordered alphabetically by their description.

\noindent %
\begin{longtable}[l]{>{\raggedright}p{0.13\textwidth}>{\raggedright}p{0.67\textwidth}l}
$\wkc$ & weak convergence (\citealp{VVW96}) \cellfill & \ref{enu:H:stocheq}\tabularnewline[\doublerulesep]
$\inprob$ & convergence in probability \cellfill & \secref{asymptotics}\tabularnewline[\doublerulesep]
$\smlnorm x$, $\smlnorm x_{A}$ & Euclidean norm, $A$-weighted norm of $x$ \cellfill & \subref{ii}\tabularnewline[\doublerulesep]
$\dot{f}$, $\ddot{f}$ & gradient, hessian of $f$ \cellfill & \subref{giidist}\tabularnewline[\doublerulesep]
$\partial_{\beta}f$, $\partial_{\beta}^{2}f$ & gradient, hessian of $f$ w.r.t.\ $\beta$ \cellfill & \remref{reparam}\tabularnewline[\doublerulesep]
$\lesssim$ & left side bounded by the right side times a constant \cellfill & \appref{parttwoproof}\tabularnewline[\doublerulesep]
$\smlnorm f_{p,\Qrob}$ & $L^{p}(\Qrob)$ norm of $f$, i.e.\ $(\expect_{\Qrob}\smlabs{f(z_{i})}^{p})^{1/p}$
\cellfill & \appref{parttwoproof}\tabularnewline[\doublerulesep]
\end{longtable}
\end{document}